\documentclass[11pt]{article}
\usepackage[utf8x]{inputenc}
\usepackage[letterpaper,top=1in,bottom=1in,left=1in,right=1in]{geometry}
\newcommand{\Spacapan}{Špacapan{}}

\usepackage{comment}

\usepackage{amsmath}
\usepackage{amssymb}
\usepackage{amsthm}
\usepackage{hyperref}
\usepackage{fancyhdr}
\usepackage{booktabs}
\newcommand{\ra}[1]{\renewcommand{\arraystretch}{#1}}
\usepackage{graphicx}
\usepackage{tikz}
\usepackage{titlesec}
\usepackage{enumitem}

\usepackage{authblk}

\newcommand{\pa}[1]{\left(#1\right)}

\newcommand{\cO}{\mathcal{O}}
\newcommand{\NN}{\mathbb{N}}

\DeclareMathOperator{\Interior}{Int}

\iffalse
    \newcommand{\togreen}[2][]{}
    \newcommand{\toyellow}[2][]{}
    \newcommand{\tored}[2][]{}
    \newcommand{\toimp}[2][]{}
    \newcommand{\changed}[2][]{}
\else
    \usepackage[disable]{todonotes}
    \tikzset{notestyleraw/.append style={rectangle}}
    \newcommand{\togreen}[2][]{\todo[color=green!40, #1]{#2}}
    \newcommand{\toyellow}[2][]{\todo[color=yellow!40, #1]{#2}}
    \newcommand{\tored}[2][]{\todo[color=red!40, #1]{#2}}
    \newcommand{\toimp}[2][]{\todo[color=blue!20, #1]{#2}}
    \newcommand{\changed}[2][]{\todo[color=blue!20, #1]{#2}}
\fi

\newenvironment{case}
  {%
  \begin{description}}
  {\end{description}}

\newcommand{\ar}[1]{\renewcommand{\arraystretch}{#1}}
\newcommand{\gaphack}{\rule{0pt}{1.5\normalbaselineskip}}
\newcommand{\gaphacksingle}{\rule{0pt}{1.0\normalbaselineskip}}

\newtheorem{claim}{Claim}
\newtheorem{lemma}{Lemma}
\newtheorem{definition}{Definition}

\newtheorem{proposition}{Proposition}
\newtheorem{theorem}{Theorem}
\newtheorem{remark}{Remark}
\newtheorem{corollary}{Corollary}

\usetikzlibrary{backgrounds}
\tikzset{every node/.style={draw,circle,fill=black,minimum size=4pt,inner sep=0pt}}
\tikzset{based/.style={ultra thick,blue}}
\newcommand{\ppdraw}[1]{%
    \def\points{{#1}}
    \foreach \pa [count=\c,remember=\pa as \pb] in \points {
        \pgfextra{\ifnum\c=0\relax\else \draw (\pa) -- (\pb);\fi}
    }
}

\newcommand{\beforeafter}[3][1.0]{%
    \begin{center}%
        \begin{tikzpicture}[baseline=(current bounding box.center), scale=#1]
            #2
        \end{tikzpicture}%
        \hspace{0.5cm} $\Rightarrow$ \hspace{0.5cm}%
        \begin{tikzpicture}[baseline=(current bounding box.center), scale=#1] 
            #3
        \end{tikzpicture}%
    \end{center}%
    }

\begin{document}
\title{Triangulations Admit Dominating Sets of Size $2n/7$.}
\author[1]{Aleksander B. G. Christiansen\thanks{Research supported by VILLUM Foundation grant 37507 ``Efficient Recomputations for Changeful Problems''}}
\author[1]{Eva Rotenberg$^{\ast \dagger}$}
\author[1]{Daniel Rutschmann\thanks{Research supported by Eva Rotenberg's Carlsberg Foundation Young Researcher Fellowship CF21-0302 - ``Graph Algorithms with Geometric Applications''}}
\affil[1]{DTU Compute, Technical University of Denmark}

\thispagestyle{empty}

\maketitle

\begin{abstract}
We show that every planar triangulation on $n> 10$ vertices has a dominating set of size $2n / 7 = n/3.5$. This approaches the $n/4$ bound conjectured by Matheson and Tarjan~\cite{matheson_dominating_1996}
, and improves significantly on the previous best bound of $17 n/53 \approx n / 3.117$ by Špacapan~\cite{spacapan_domination_2020}.

From our proof it follows that every $3$-connected $n$-vertex near-triangulation (except for $3$ sporadic examples) has a dominating set of size $n/3.5$. On the other hand, for $3$-connected near-triangulations, we show a lower bound of $3(n-1)/11 \approx n/3.666$, demonstrating that the conjecture by Matheson and Tarjan~\cite{matheson_dominating_1996} cannot be strengthened to $3$-connected near-triangulations.

Our proof uses a penalty function that, aside from the number of vertices, penalises vertices of degree $2$ and specific constellations of neighbours of degree $3$ along the boundary of the outer face. To facilitate induction, we not only consider near-triangulations, but a wider class of graphs (\emph{skeletal} triangulations), allowing us to delete vertices more freely. Our main technical contribution is a set of \emph{attachments}, that are small graphs we inductively attach to our graph, in order both to remember whether existing vertices are already dominated, and that serve as a tool in a divide and conquer approach. Along with a well-chosen potential function, we thus both remove and add vertices during the induction proof.

We complement our proof with a constructive algorithm that returns a dominating set of size $\le 2n/7$. Our algorithm has a quadratic running time.
\end{abstract}

\thispagestyle{empty}

\newpage
\setcounter{page}{1}

\section{Introduction}

A \emph{dominating set} in an $n$ vertex graph $G$ is a subset $S$ of the vertices of $G$ such that every vertex in $G$ either is in $S$ or neighbours a vertex in $S$. 
When studying dominating sets, one is typically interested in making them as small as possible.
The minimum size of a dominating set in $G$ is denoted by $\gamma(G)$. 
Instead of studying the minimum dominating set for a particular graph, Matheson and Tarjan~\cite{matheson_dominating_1996} originally asked if one can determine an upper bound on $\gamma$ for classes of graphs. 
In particular, they studied two classes of graphs: plane \emph{triangulations} and plane \emph{near-triangulations}. Here a plane graph refers to a planar graph, i.e.\ a graph that may be embedded in the plane in such a way that no two edges cross, together with such a crossing-free embedding in the plane. 
A plane graph is \emph{internally triangulated} if every bounded face is bounded by a triangle. 
A \emph{near-triangulation} is a 2-connected internally triangulated plane graph, and a \emph{triangulation} is a near-triangulation with exactly three boundary vertices. 
Matheson and Tarjan~\cite{matheson_dominating_1996} showed that for any plane near-triangulation $G$, it holds that $ \gamma(G) \le \frac{n}{3} $. 
They also showed that this result is tight in the sense that there exists an infinite family of plane near-triangulation such that for every graph in the infinite family, the minimum dominating set has size exactly a third of the number of vertices in the graph, i.e.\ $ \gamma(G) = \frac{n}{3} $. 
However, for triangulations they were only able to prove an upper bound of  $\gamma(G) \le \frac{n}{3}$ and provide an infinite family where every graph in the family required $\frac{n}{4}$ points to dominate.
Aside from some small sporadic examples, they conjectured that asymptotically $\gamma(G) \le \lfloor n / 4 \rfloor$ when $G$ is a sufficiently large triangulation. 
This problem has proved difficult to approach, and for over 20 years, there were no improvements that applied to all triangulations. 
Recently, \Spacapan~\cite{spacapan_domination_2020} gave the first improved bound for general triangulations, when he showed that, in every large enough triangulation,
$\gamma(G) \le \lfloor 17n/53\rfloor \approx n / 3.117$.

In broad terms, the problem has been approached in two different ways. Either 1) papers have tried to find combinatorial objects -- like a colouring or a Hamiltonian cycle -- with certain properties that allows one to extract a small dominating set, or 2) one has attempted some sort of inductive or reduction based approach in order to try and iteratively reduce the problem complexity until it can be handled directly. 
The problem is elusive, as the above approaches has to deal with two obstructions: Firstly, the bound does not hold for small values of $n$ as there are small, sporadic counter examples, which means that one has to be careful when reducing the problem. Secondly, it seems difficult to pin-point enough structure in general triangulations to guarantee a combinatorial object with strong enough properties. 
This has motivated researchers to either restrict the problem to sub-classes of (near-)triangulations containing more structure like for instance triangulations with maximum degree $6$ \cite{KingPelsmajer10,LiuPelsmajer11}, Hamiltonian triangulations \cite{plummer_dominating_2020} or maximal outerplanar graphs \cite{campos_dominating_2013, tokunaga_dominating_2013}, or to consider broader classes of graphs in which it is easier to reduce the problem to one of smaller complexity \cite{spacapan_domination_2020}. 
See Table~\ref{tbl:upper_bounds} for an overview of known upper bounds. 

\begin{table*}
	\centering
\begin{tabular}{@{}llll@{}}
\toprule
Reference                                                                & Class of graphs                                                                  & Size of dom.\ set                                        & Comment                                                                                                                                         \\ \midrule
Matheson \& Tarjan~\cite{matheson_dominating_1996}                                                       & Near-triangulations                                                              & $\frac{n}{3}$                                                         &                                                                                                                                                 \\
	\begin{tabular}[c]{@{}l@{}}Campos \& Wakabayashi~\cite{campos_dominating_2013}\\ Tokunaga~\cite{tokunaga_dominating_2013}\end{tabular} & \begin{tabular}[c]{@{}l@{}}Maximal outer-\\ planar graphs\end{tabular}           & $\lceil \frac{n+t}{4} \rceil$                                         & \gaphack\begin{tabular}[c]{@{}l@{}}Here $t$ is the no.\textbackslash of\\ degree-2 vertices.\end{tabular}                                               \\
King \& Pelsmajer~\cite{KingPelsmajer10}                                                        & \gaphack\begin{tabular}[c]{@{}c@{}}Plane triangulations\\ with max-degree 6\end{tabular} & $\frac{n}{4}$                                                         &                                                                                                                                                 \\
Liu \& Pelsmajer~\cite{LiuPelsmajer11}                                                         & \begin{tabular}[c]{@{}c@{}}Plane triangulations\\ with max-degree 6\end{tabular} & $\frac{n}{6} + c$                                                     & For some constant $c$.                                                                                                                          \\
Plummer, Ye \& Zha~\cite{plummer_dominating_2016, plummer_dominating_2020}                                                       & \ar{1.0}\gaphack\begin{tabular}[c]{@{}l@{}}Hamiltonian trian-\\ gulations\end{tabular}           & $\frac{5n}{16}$                                                        &             For $n \geq 23$.                                                                                                                                    \\
\Spacapan~\cite{spacapan_domination_2020}                                                                 & Triangulations                                                                   & $\frac{17n}{53}$                                                      &                                                                                                                                                 \\ \midrule
New                                                                      & Triangulations                                                                   & $\frac{2n}{7}$                                                        &                                                                                                                                                 \\
New                                                                      &\gaphack\begin{tabular}[c]{@{}l@{}}3-connected near-\\ triangulations\end{tabular}       & $\frac{2n}{7}$                                                        &                                                                                                                                                 \\ \bottomrule
\end{tabular}
\caption{\label{tbl:upper_bounds} Upper bounds for the size of a minimum dominating set for various graph classes.}
\end{table*}

More specifically, in the first line of research: King and Pelsmajer~\cite{KingPelsmajer10} confirmed the conjecture for graphs of maximum degree 6, and Liu and Pelsmajer~\cite{LiuPelsmajer11} strengthened this result to show that in fact for these graphs $\gamma(G) \leq  \frac{n}{6} + c$ for some constant $c$. Plummer, Ye and Zha~\cite{plummer_dominating_2016} studied first $4$-connected plane triangulations, which in particular are Hamiltonian~\cite{thomassen1983theorem} and have minimum degree at least $4$, and showed the existence of a dominating set of size $\leq \max \{\lceil \frac{2n}{7}\rceil, \lfloor\frac{5n}{16}\rfloor\}$. Then in~\cite{plummer_dominating_2020}, they showed that for Hamiltonian triangulations of size at least $23$ it holds that $\gamma(G) \leq \frac{5n}{16}$. 
Finally, in maximal outerplanar graphs, even more fine-grained results are known: Campos and Wakabayashi \cite{campos_dominating_2013} showed $\gamma(G) \le \lfloor (n+t)/4 \rfloor$ where $t$ is the number of degree-2 vertices. In the three last results, a good understanding of the obstructions to achieving an $n/4$ bound, such as degree-2 vertices, is key. Tokunaga \cite{tokunaga_dominating_2013} gave an elegant proof of this bound via a coloring method.

On the other hand, \Spacapan~\cite{spacapan_domination_2020} considered a more general class of graphs that he denoted \emph{weak near-triangulations}. 
He showed how to reduce weak near-triangulations while staying inside the graph class, until one eventually ends up with an irreducible weak near-triangulation. 
These irreducible weak near-triangulations contained enough structure for \Spacapan{} to subsequently construct a small dominating set, if one begins with a triangulation. 
However, in \Spacapan{}'s framework one has to argue this in a manual fashion separately from the arguments that handle the reductions. 

In our approach, we extend the reduction step to make the construction of the small dominating set automatic.
Similarly to \Spacapan, we consider a more general class of graphs, however, in our case, we consider what we call \emph{skeletal triangulations}. 
In order to avoid having to extract the dominating set manually, we employ a penalty function that in a more fine-grained manner accounts for the cost of performing certain reductions. 
This penalty function not only penalises degree-2 vertices (more specifically degree-2 cut vertices and `ears'), but also a new type of attachment that we call \emph{facial bad 5-wheels}. 
To illustrate the importance of penalising these 5-wheels, we show an infinite family of near-triangulations with no degree-2 vertices in which the smallest dominating set has size $\frac{3n}{10}$.
Furthermore, we show that our analysis using this penalty function is tight in the sense that there exists an infinite family of near-triangulations which contain none of the penalised attachments
and admit no dominating sets with fewer than $\frac{2n}{7}$ vertices.
Finally,  we show that only penalising attachments arising from a 2-cut is not sufficient to achieve an $\frac{n}{4}$ bound for non-penalised near-triangulations, as we provide an infinite family of 3-connected near-triangulations with $\gamma(G) = 3n/11 - O(1)$. Interestingly, this indicates a big difference between what is conjectured for triangulations and what holds for 3-connected near-triangulations. 
In Table~\ref{tbl:lower_bounds} we give an overview over known lower bounds and the new lower bounds we introduce in this paper.
We introduce the lower bound constructions in Section~\ref{sec:lower_bounds}.
\begin{table*}
\begin{center}
\begin{tabular}{@{}lll@{}}
\toprule
Reference                                & Class of graphs                                                                                           & Size of dom.\ set \\ \midrule
Matheson \& Tarjan~\cite{matheson_dominating_1996}                       & Near-triangulations.                                                                                       & $\frac{n}{3}$                  \\
Matheson \& Tarjan~\cite{matheson_dominating_1996} & \gaphacksingle Triangulations.                                                                                            & $\frac{n}{4}$                  \\ \midrule
New                                      & \begin{tabular}[c]{@{}l@{}}Near-triangulations\\  with minimum degree $3$.\end{tabular}                                & $\frac{3n}{10}$                \\
New                                      & \gaphack\begin{tabular}[c]{@{}l@{}}3-connected near-\\ triangulations.\end{tabular}                                & $\frac{3n}{11}$                \\
New                                      & \gaphack\begin{tabular}[c]{@{}l@{}}Near-triangulations with neither\\ bad 5-wheels nor degree 2 vertices.\end{tabular} & $\frac{2n}{7}$                 \\
New                                      & \gaphacksingle Eulerian triangulations. & $\frac{n}{4}$                  \\ \bottomrule
\end{tabular}
\caption{\label{tbl:lower_bounds} Lower bounds for the size of a minimum dominating set for various graph classes.}
\end{center}
\vspace{-1em}
\end{table*}

Since we conduct our inductive argument over a broader class of graphs, we can reduce very aggressively while staying in the same class of graphs, but we now have the added difficulty of also carrying the penalty function along, as we reduce. 
In order to be able to do so, we apply two techniques. 
Firstly, we show how to encode the fact that some vertices might already be dominated in our candidate dominated set, while staying in the same graph class. 
To do so, we fuse small attachments to the graph and thus increase the number of vertices and create small cuts. 
Secondly, in order to be able to handle this broader class of graphs, we study small cuts of size $\leq 2$, and show that we may replace one side of the cut by one of a finite list of examples that `acts as' the cut that was just replaced on the rest of the graph. 
This allows us to assume that $G$ is ``almost'' 3-connected, which makes a deletion-based induction proof feasible.
We elaborate further on this in the proof-overview section.  In the next section, we sum up our contributions.
 
\subsection{Our contributions}
The following generalization of near-triangulations allows for cut vertices,
which gives some added flexibility when deleting vertices.

\begin{definition}[Skeletal triangulation]
    A \emph{skeletal triangulation} is a connected internally triangulated planar graph in which every vertex has degree $\geq 2$.
\end{definition}
Every near-triangulation is a skeletal triangulation. In fact, every connected weak-near triangulation
\cite{spacapan_domination_2020} is a skeletal triangulation, but not vice-versa.

\begin{definition}[Problematic configurations]
    Let $G$ be a skeletal triangulation. An \emph{ear} in $G$ is a facial triangle
    with at least one vertex of degree $2$ in $G$.
    A \emph{bad 5-wheel} is a subgraph $H \subseteq G$ isomorphic to the 5-wheel such that the
    outer 4-cycle in $H$ contains at least two consecutive $G$-boundary vertices of degree 3, called a \emph{3-pair}. (See Figure~\ref{fig:introdegtwoexample}.)
\end{definition}

\paragraph{Upper bounds.} Our main result is the following
\begin{theorem} \label{theo:main}
    Let $G$ be a skeletal triangulation on $n > 10$ vertices.
    Let $e$, $f$ and $t$ be the number of ears, bad 5-wheels and degree-2 cut vertices in $G$, respectively.
    Then, $\gamma(G) \le \lfloor \frac{n + e/2 + f/2 + t/2}{3.5}\rfloor$.
\end{theorem}
\begin{corollary}\label{coro:main}
    Let $G$ be a triangulation or 3-connected near-triangulation on $n > 10$ vertices.
    Then, $\gamma(G) \le \lfloor \frac{n}{3.5} \rfloor$.
\end{corollary}
\begin{proof}
    As $G$ is 3-connected, there are no ears or cut vertices and at most one bad 5-wheel.
    If there is a bad 5-wheel, then $G$ has exactly four boundary vertices,
    two of which have degree at least four. Deleting the boundary edge between
    the latter destroys the 5-wheel without creating any new problematic configurations.
    Finally, Theorem~\ref{theo:main} gives the result with $e = f = t = 0$.
\end{proof}

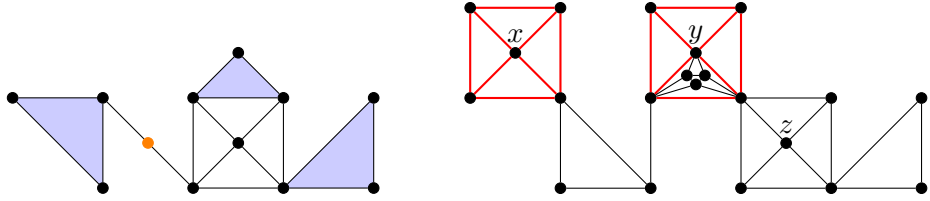
\begin{figure}
    \begin{center} \begin{tikzpicture}[scale=0.6]
        \node (A) at (0, 0) {};
        \node (B) at (2, 0) {};
        \node (C) at (2, 2) {};
        \node (D) at (0, 2) {};
        \node (E) at (1, 1) {};
        \node (T) at (1, 3) {};
        \node (L) at (-1, 1) [orange] {};
        \node (M) at (-2, 2) {};
        \node (N) at (-4, 2) {};
        \node (O) at (-2, 0) {};
        \node (R) at (4, 0) {};
        \node (S) at (4, 2) {};

        \draw (A) -- (B) -- (C) -- (D) -- (A);
        \draw (A) -- (E) -- (C) (B) -- (E) -- (D);
        \draw (C) -- (T) -- (D);
        \draw (A) -- (L) -- (M);
        \draw (M) -- (N) -- (O) -- (M);
        \draw (B) -- (R) -- (S) -- (B);
        \begin{scope}[on background layer]
            \path [fill=blue, opacity=0.2] (M.center) to (N.center) to (O.center) to cycle;
            \path [fill=blue, opacity=0.2] (C.center) to (T.center) to (D.center) to cycle;
            \path [fill=blue, opacity=0.2] (B.center) to (R.center) to (S.center) to cycle;
        \end{scope}
    \end{tikzpicture}\hspace{1cm}
    \begin{tikzpicture}[scale=0.6]
        \node (A) at (0, 0) {};
        \node (B) at (2, 0) {};
        \node (C) at (2, 2) {};
        \node (D) at (0, 2) {};
        \node (E) at (1, 1) [label=above:$y$] {};
        \node (c) at (1, 0.3) {};
        \node (b) at (0.8, 0.5) {};
        \node (a) at (1.2, 0.5) {};
        \node (W) at (3, -1) [label=above:$z$] {};
        \node (X) at (2, -2) {};
        \node (Y) at (4, -2) {};
        \node (Z) at (4, 0) {};
        \node (I) at (0, -2) {};
        \node (J) at (-2, -2) {};
        \node (P) at (-2, 0) {};
        \node (Q) at (-2, 2) {};
        \node (R) at (-4, 2) {};
        \node (S) at (-4, 0) {};
        \node (O) at (-3, 1) [label=above:$x$] {};
        \node (M) at (6, -2) {};
        \node (N) at (6, 0) {};

        \draw[red,thick] (A) -- (B) -- (C) -- (D) -- (A) -- (E) -- (C);
        \draw[red,thick] (B) -- (E) -- (D);
        \draw (B) -- (X) -- (Y) -- (Z) -- (B) -- (W) -- (Y);
        \draw (X) -- (W) -- (Z);
        \draw (A) -- (I) -- (J) -- (P) -- (I);
        \draw[red,thick] (P) -- (Q) -- (R) -- (S) -- (P);
        \draw[red,thick] (P) -- (O) -- (R) (Q) -- (O) -- (S);
        \draw (Y) -- (M) -- (N) -- (Y);
        \draw (A) -- (c) -- (B) -- (a) -- (E) -- (b) -- (A);
        \draw (a) -- (b) -- (c) -- (a);
    \end{tikzpicture}
    \end{center}
    \caption{(Left) a skeletal triangulation with three ears (blue) and a degree-2 cut vertex (orange).
	(Right) A skeletal triangulation with two bad 5-wheels (red), centered at $x$ and $y$.
	Note that there is no bad 5-wheel centered at $z$,
	as there are no two consecutive boundary vertices of degree $3$ on that 5-wheel.}
    \label{fig:introdegtwoexample}
\end{figure}

\paragraph{Lower bounds.} \label{sec:lower_bounds}
The following (infinite) families of examples motivate
our definition of skeletal triangulations and our selection of problematic configurations.
Matheson and Tarjan \cite{matheson_dominating_1996} constructed near-triangulations with $\gamma(G) = n/3$ and triangulations with $\gamma(G) = n/4$,
see Figure~\ref{fig:whypenalty}. The limiting factor in these examples are vertices of degree 2 and 3, respectively.
We construct near-triangulations with (a) no degree-2 vertices, no bad 5-wheels
and $\gamma(G) = 2n / 7$ and (b) no-degree 2 vertices and $\gamma(G) = 3n / 10$, see Figure~\ref{fig:examples}.
This shows that $n / 3.5$ is best possible given our choice of problematic configurations
and that penalizing both degree-2 vertices and bad 5-wheels is necessary to achieve the $n / 3.5$ bound.

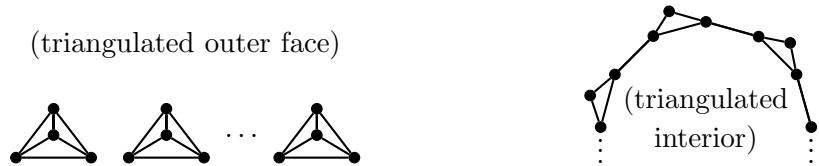
\begin{figure}[h!]
\begin{center}
\begin{tikzpicture}[scale=0.5]
\node (A) at (0, 0) {};
\node (B) at (2, 0) {};
\node (C) at (1, .6) {};
\node (D) at (1, 1.3) [] {};

\node (A1) at (3, 0) {};
\node (B1) at (5, 0) {};
\node (C1) at (4, .6) {};
\node (D1) at (4, 1.3) {};

\node[fill=none,draw=none] (text) at (6,.6) {$\ldots$};

\node (A2) at (7, 0) {};
\node (B2) at (9, 0) {};
\node (C2) at (8, .6) {};
\node (D2) at (8, 1.3) {};

\node[rectangle, draw=none,fill=none] (text) at (4.5,3) {(triangulated outer face)};
	
	\draw[thick] (A) -- (B) -- (C) -- (D) -- (A) -- (C) -- (D) -- (B);
	\draw[thick] (A1) -- (B1) -- (C1) -- (D1) -- (A1) -- (C1) -- (D1) -- (B1);
	\draw[thick] (A2) -- (B2) -- (C2) -- (D2) -- (A2) -- (C2) -- (D2) -- (B2);
\end{tikzpicture} \hfil
\begin{tikzpicture}[scale=0.7]
	\node (A) at (0, 2) {};
	\node (B) at (1,1.72) {};
	\node (C) at (1.72,1) {};
	\node (D) at (2, 0) {};

	\node (Z) at (-1,1.73) {};
	\node (Y) at (-1.73,1) {};
	\node (X) at (-2, 0) {};
	
	\node[fill=none,draw=none] (dotsI) at (2,-.3) {$\vdots$};
	\node[fill=none,draw=none] (dotsII) at (-2,-.3) {$\vdots$};
	\node[rectangle, draw=none,fill=none] (textI) at (0,.5) {(triangulated};
	\node[rectangle, draw=none,fill=none] (textII) at (0,-.25) {interior)};

	\node (e1) at (-2.2, .6) {};
	\node (e2) at (-.7, 2.2) {};
	\node (e3) at (1.6, 1.6) {};
	
	\draw[thick] (X) -- (Y) -- (Z) -- (A) -- (B) -- (C) -- (D);
	\draw[thick] (X) -- (e1) -- (Y) -- (Z) -- (e2) -- (A) -- (B) -- (e3) -- (C) -- (D);
\end{tikzpicture} 
\end{center}
\vspace{-1em}
\caption{(Left) copies of $K_4$ with the outer face triangulated arbitrarily shows that $n/4$ is needed~\cite{matheson_dominating_1996}. (Right) an outer-planar near-triangulation where every third vertex of the outer face has degree $2$ motivates penalising ears~\cite{matheson_dominating_1996}. \label{fig:whypenalty}}
\end{figure}

\begin{figure}[h!]
\begin{tikzpicture}[scale=0.6,
	specialvertex/.style={}]
	
\node (A) at (-1.5, 0) {};
\node (B) at (1.5, 0) {};
\node (C) at (0, .8) {};
\node (D) at (0, 2) {};

\node[specialvertex] (E1) at (-3,1.5) {};
\node[specialvertex] (F1) at (-1.5,3.5) {};
\node (G1) at (-1.5,1.75) {};

\node[specialvertex] (E2) at (3,1.5) {};
\node[specialvertex] (F2) at (1.5,3.5) {};
\node (G2) at (1.5,1.75) {};

\draw[] (E1) -- (G1) -- (A) -- (E1) -- (F1) -- (D) -- (G1) -- (F1);
\draw[] (E2) -- (G2) -- (B) -- (E2) -- (F2) -- (D) -- (G2) -- (F2);
\draw[] (A) -- (B) -- (C) -- (D) -- (A) -- (C) -- (D) -- (B);

\draw[very thick, brown!50!black] (A) -- (B);

\end{tikzpicture}\hfill
\begin{tikzpicture}[scale=0.35,
	specialvertex/.style={rectangle,fill=teal}]
	
	\node (Af) at (-1.5, 0) {};
	\node (Bf) at (1.5, 0) {};
	\node (Cf) at (0, .8) {};
	\node (Df) at (0, 2) {};
	
	\node (E1) at (-3,1.5) {};
	\node (F1) at (-1.5,3.5) {};
	\node (G1) at (-1.5,1.75) {};
	
	\node (E2) at (3,1.5) {};
	\node (F2) at (1.5,3.5) {};
	\node (G2) at (1.5,1.75) {};
	
	\draw[] (E1) -- (G1) -- (Af) -- (E1) -- (F1) -- (Df) -- (G1) -- (F1);
	\draw[] (E2) -- (G2) -- (Bf) -- (E2) -- (F2) -- (Df) -- (G2) -- (F2);
	\draw[] (Af) -- (Bf) -- (Cf) -- (Df) -- (Af) -- (Cf) -- (Df) -- (Bf);
	
	\draw[very thick, brown!50!black] (Af) -- (Bf);

\begin{scope}[rotate around={-75:(0,-4)}]
\node (As) at (-1.5, 0) {};
\node (Bs) at (1.5, 0) {};
\node (Cs) at (0, .8) {};
\node (Ds) at (0, 2) {};

\node (E1) at (-3,1.5) {};
\node (F1) at (-1.5,3.5) {};
\node (G1) at (-1.5,1.75) {};

\node (E2) at (3,1.5) {};
\node (F2) at (1.5,3.5) {};
\node (G2) at (1.5,1.75) {};

\draw[] (E1) -- (G1) -- (As) -- (E1) -- (F1) -- (Ds) -- (G1) -- (F1);
\draw[] (E2) -- (G2) -- (Bs) -- (E2) -- (F2) -- (Ds) -- (G2) -- (F2);
\draw[] (As) -- (Bs) -- (Cs) -- (Ds) -- (As) -- (Cs) -- (Ds) -- (Bs);
\draw[very thick, brown!50!black] (As) -- (Bs);
\end{scope}

\begin{scope}[rotate around={75:(0,-4)}]
	\node (At) at (-1.5, 0) {};
	\node (Bt) at (1.5, 0) {};
	\node (Ct) at (0, .8) {};
	\node (Dt) at (0, 2) {};
	
	\node (E1) at (-3,1.5) {};
	\node (F1) at (-1.5,3.5) {};
	\node (G1) at (-1.5,1.75) {};
	
	\node (E2) at (3,1.5) {};
	\node (F2) at (1.5,3.5) {};
	\node (G2) at (1.5,1.75) {};
	
	\draw[] (E1) -- (G1) -- (At) -- (E1) -- (F1) -- (Dt) -- (G1) -- (F1);
	\draw[] (E2) -- (G2) -- (Bt) -- (E2) -- (F2) -- (Dt) -- (G2) -- (F2);
	\draw[] (At) -- (Bt) -- (Ct) -- (Dt) -- (At) -- (Ct) -- (Dt) -- (Bt);
	\draw[very thick, brown!50!black] (At) -- (Bt);
\end{scope}

\draw[] (Af) -- (Bt);
\draw[] (As) -- (Bf);
	
\node[rectangle, draw=none,fill=none] (textI) at (0,-3) {(triangulated};
\node[rectangle, draw=none,fill=none] (textII) at (0,-4.5) {interior)};

\end{tikzpicture}\hfill
\begin{tikzpicture}[scale=0.6,
	spex/.style={}]
	
	\node (A) at (-2, 0) {};
	\node (B) at (2, 0) {};
	\node[spex] (C) at (0, .8) {};
	\node (D) at (0, 1.7) {};
	\node[spex] (E) at (0, 4) {};
	\node[spex] (F) at (-.4, 2.2) {};
	\node[spex] (G) at (.4, 2.2) {};
	
	\draw (A) -- (B) -- (C) -- (D) -- (A) -- (C) -- (D) -- (B) -- (E) -- (A) -- (F) -- (G) -- (D) -- (F) -- (E) -- (G) -- (B);

	\draw[very thick, brown!50!black] (A) -- (B);

\end{tikzpicture}

\caption{(Left) a $10$-vertex graph that requires a dominating set of size $3$. (Middle) attaching this to every second edge of a triangulated polygon yields a class with $10k$ vertices and $\gamma = 3k$. (Right) a $7$ vertex graph demanding $2$ vertices to dominate. By a similar construction as in (middle), this gives a class with $7k$ vertices and $\gamma = 2k$. \label{fig:examples}}
\end{figure}
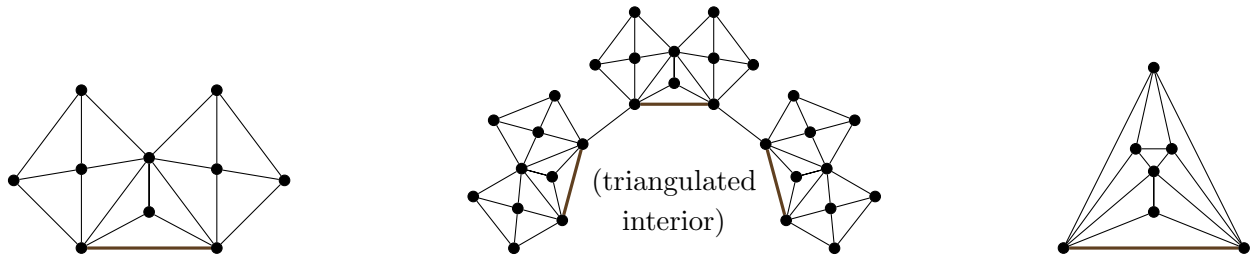

Requiring skeletal triangulations to (a) be connected avoids disjoint unions of octahedra, with $\gamma(G) = n / 3$,
and (b) have minimum degree 2 avoids caterpillars with $\gamma(G) = n / 2$.
A penalty of $\frac{1}{2}$ on degree-2 cut vertices is motivated by the example shown in Figure~\ref{fig:balloons}.

So far, our lower bounds for near-triangulations describe classes of graphs with many chords. It is natural to think that chords, or two-cuts, are the sole reason $n/4$ does not suffice for these graphs. 
In \cite{tokunaga_dominating_2013}, Tokunaga  conjectured that every 3-connected near-triangulation satisfies $\gamma(G) \le \lfloor (n+2)/4 \rfloor$.
We construct 3-connected triangulations with $\gamma(G) = 3n/11 - O(1)$, refuting this conjecture, see Figure~\ref{fig:triconnected} (left and middle).
To our surprise, this either shows a stark difference between triangulations and 3-connected near-triangulations, or is a counter-indication of the $n/4$ conjecture.
In particular, a proof of the $n/4$ conjecture might have to approach triangulations via 4-connected triangulations and separating triangles, in order 
to break through this $3n/11$ barrier.

Finally, we construct a triangulation with no odd-degree
vertices and $\gamma(G) = n/4$, see Figure~\ref{fig:triconnected} (right). Placing disjoint copies of this graph and carefully triangulating the outer face (similar to Figure~\ref{fig:whypenalty} (left)) yields an infinite class of even graphs with $\gamma(G) = n/4$. In particular, the conjectured $n/4$ bound is best possible even in the absence of degree-3 vertices.

\begin{figure}[h]
\begin{center}
\begin{tikzpicture}[scale=0.5,
	spex/.style={rectangle,fill=blue}]
	\begin{scope}[rotate around={90:(0,-4)}]	
		\node (A) at (-2, 0) {};
		\node (B) at (2, 0) {};
		\node (C) at (0, .8) {};
		\node (D) at (0, 1.7) {};
		\node (E) at (0, 4) {};
		\node (F) at (-.4, 2.2) {};
		\node (G) at (.4, 2.2) {};
		\node (H) at (0, -1) {};
		\node (I) at (0, -2.5) {};
		\node[spex] (J) at (0, -4) {};
		
		\draw (H) -- (A) -- (B) -- (C) -- (D) -- (A) -- (C) -- (D) -- (B) -- (E) -- (A) -- (F) -- (G) -- (D) -- (F) -- (E) -- (G) -- (B) -- (H) -- (I) -- (J);
	\end{scope}
\end{tikzpicture}\hspace{.3\linewidth}
\begin{tikzpicture}[scale=0.25,
	spex/.style={rectangle,fill=blue}]
   \begin{scope}[rotate around={25:(0,-15)}]
		\node (A) at (-2, 0) {};
		\node (B) at (2, 0) {};
		\node (C) at (0, .8) {};
		\node (D) at (0, 1.7) {};
		\node (E) at (0, 4) {};
		\node (F) at (-.4, 2.2) {};
		\node (G) at (.4, 2.2) {};
		\node (H) at (0, -1) {};
		\node (I) at (0, -2.5) {};
		\node[spex] (J1) at (0, -4) {};
		
		\draw (H) -- (A) -- (B) -- (C) -- (D) -- (A) -- (C) -- (D) -- (B) -- (E) -- (A) -- (F) -- (G) -- (D) -- (F) -- (E) -- (G) -- (B) -- (H) -- (I) -- (J1);
	\end{scope}
   \begin{scope}[rotate around={0:(0,0)}]	
	\node (A) at (-2, 0) {};
	\node (B) at (2, 0) {};
	\node (C) at (0, .8) {};
	\node (D) at (0, 1.7) {};
	\node (E) at (0, 4) {};
	\node (F) at (-.4, 2.2) {};
	\node (G) at (.4, 2.2) {};
	\node (H) at (0, -1) {};
	\node (I) at (0, -2.5) {};
	\node[spex] (J2) at (0, -4) {};
	
	\draw (H) -- (A) -- (B) -- (C) -- (D) -- (A) -- (C) -- (D) -- (B) -- (E) -- (A) -- (F) -- (G) -- (D) -- (F) -- (E) -- (G) -- (B) -- (H) -- (I) -- (J2);
\end{scope}
   \begin{scope}[rotate around={-25:(0,-15)}]	
	\node (A) at (-2, 0) {};
	\node (B) at (2, 0) {};
	\node (C) at (0, .8) {};
	\node (D) at (0, 1.7) {};
	\node (E) at (0, 4) {};
	\node (F) at (-.4, 2.2) {};
	\node (G) at (.4, 2.2) {};
	\node (H) at (0, -1) {};
	\node (I) at (0, -2.5) {};
	\node[spex] (J3) at (0, -4) {};
	
	\draw (H) -- (A) -- (B) -- (C) -- (D) -- (A) -- (C) -- (D) -- (B) -- (E) -- (A) -- (F) -- (G) -- (D) -- (F) -- (E) -- (G) -- (B) -- (H) -- (I) -- (J3);
\end{scope}
\begin{scope}[rotate around={-50:(0,-15)}]	
	\node[draw=none,fill=none] (J4) at (0, -4) {};
\end{scope}
\begin{scope}[rotate around={50:(0,-15)}]	
	\node[draw=none,fill=none] (J0) at (0, -4) {};
\end{scope}
\draw (J0) -- (J1) -- (J2) -- (J3) -- (J4);
\end{tikzpicture}
\end{center}
\vspace{-1em}
\caption{(Left) A gadget $G$ on $10$ vertices with $\gamma(G) = 3$. (Right) Attaching copies of this gadget to each vertex of a triangulated polygon yields a
skeletal triangulation with $k$ degree-2 cut vertices and $\gamma = (n+\frac{k}{2})/3.5$.\label{fig:balloons}}
\end{figure}
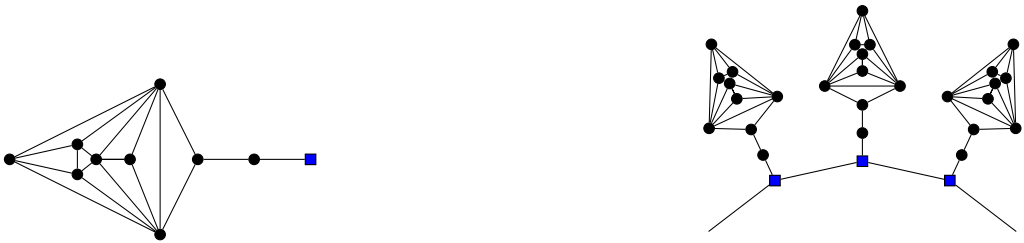

\begin{figure}[h!]
\begin{tikzpicture}[scale=0.5,
	spex/.style={circle,fill=blue,inner sep=2.5pt}]
	
	\node (A) at (-2, 0) {};
	\node (B) at (2, 0) {};
	\node (C) at (0, -.8) {};
	\node[rectangle] (D) at (0, -2.5) {};
	\node[rectangle] (E) at (0, 2.5) {};
	\node (F) at (-.4, 1) {};
	\node (G) at (.4, 1) {};
	\node (H) at (0, .5) {};
	\node (I) at (-3.5, -1) {};
	\node (J) at (-3.5, 1) {};
	\node (K) at (-2.75,-.5) {};
	\node[spex] (X) at (5,0) {};
	
	\draw (C) -- (A) -- (H) -- (B) -- (G) -- (H) -- (F) -- (A) -- (E) -- (F) -- (G) -- (E) -- (B) -- (D) -- (A) -- (B) -- (C) -- (D) -- (K) -- (A) -- (J) -- (K) -- (I) -- (J) -- (E) -- (X) -- (B) -- (X) -- (D) -- (I) ;
	
	\draw[thick, blue!50!black] (D) -- (X) -- (B) -- (X) -- (E);
\end{tikzpicture}\hfill
\begin{tikzpicture}[scale=0.25,
	spex/.style={circle,fill=blue,inner sep=1pt}]

	\node (A) at (-2, 0) {};
	\node (B) at (2, 0) {};
	\node (C) at (0, -.8) {};
	\node[rectangle] (D) at (0, -2.5) {};
	\node[rectangle] (E) at (0, 2.5) {};
	\node (F) at (-.4, 1) {};
	\node (G) at (.4, 1) {};
	\node (H) at (0, .5) {};
	\node (I) at (-3.5, -1) {};
	\node (J) at (-3.5, 1) {};
	\node (K) at (-2.75,-.5) {};
	\node[spex] (X) at (5,0) {};

\begin{scope}[rotate around={-75:(5,0)}]
\node (A1) at (-2, 0) {};
\node (B1) at (2, 0) {};
\node (C1) at (0, -.8) {};
\node[rectangle] (D1) at (0, -2.5) {};
\node[rectangle] (E1) at (0, 2.5) {};
\node (F1) at (-.4, 1) {};
\node (G1) at (.4, 1) {};
\node (H1) at (0, .5) {};
\node (I1) at (-3.5, -1) {};
\node (J1) at (-3.5, 1) {};
\node (K1) at (-2.75,-.5) {};
\end{scope}

\begin{scope}[rotate around={-150:(5,0)}]
	\node (A2) at (-2, 0) {};
	\node (B2) at (2, 0) {};
	\node (C2) at (0, -.8) {};
	\node[rectangle] (D2) at (0, -2.5) {};
	\node[rectangle] (E2) at (0, 2.5) {};
	\node (F2) at (-.4, 1) {};
	\node (G2) at (.4, 1) {};
	\node (H2) at (0, .5) {};
	\node (I2) at (-3.5, -1) {};
	\node (J2) at (-3.5, 1) {};
	\node (K2) at (-2.75,-.5) {};
\end{scope}

\begin{scope}[rotate around={-225:(5,0)}]
\node[fill=none,draw=none] (D3) at (0, -2.5) {};
\end{scope}

	\draw (C) -- (A) -- (H) -- (B) -- (G) -- (H) -- (F) -- (A) -- (E) -- (F) -- (G) -- (E) -- (B) -- (D) -- (A) -- (B) -- (C) -- (D) -- (K) -- (A) -- (J) -- (K) -- (I) -- (J) -- (E) -- (X) -- (B) -- (X) -- (D) -- (I) ;
	
	\draw[thick, blue!50!black] (D) -- (X) -- (B) -- (X) -- (E);

	\draw (C1) -- (A1) -- (H1) -- (B1) -- (G1) -- (H1) -- (F1) -- (A1) -- (E1) -- (F1) -- (G1) -- (E1) -- (B1) -- (D1) -- (A1) -- (B1) -- (C1) -- (D1) -- (K1) -- (A1) -- (J1) -- (K1) -- (I1) -- (J1) -- (E1) -- (X) -- (B1) -- (X) -- (D1) -- (I1) ;
	\draw[thick, blue!50!black] (D1) -- (X) -- (B1) -- (X) -- (E1);

	\draw (C2) -- (A2) -- (H2) -- (B2) -- (G2) -- (H2) -- (F2) -- (A2) -- (E2) -- (F2) -- (G2) -- (E2) -- (B2) -- (D2) -- (A2) -- (B2) -- (C2) -- (D2) -- (K2) -- (A2) -- (J2) -- (K2) -- (I2) -- (J2) -- (E2) -- (X) -- (B2) -- (X) -- (D2) -- (I2) ;
	\draw[thick, blue!50!black] (D2) -- (X) -- (B2) -- (X) -- (E2);
	
	\draw (E) -- (D1); 
	\draw (E1) -- (D2);
	\draw[dotted] (E2) -- (D3);
\end{tikzpicture}\hfill
\begin{tikzpicture}[scale=0.6,yscale=.5,
spex/.style={circle,fill=blue,inner sep=3pt}]		
\begin{scope}[rotate around={90:(0,0)}]
\node (A) at (4,0) {};
\node (B1) at (-5,4) {};
\node (B2) at (-5,-4) {};
\node (C1) at (-3,.5) {};
\node (C2) at (-3,-.5) {};
\node (D) at (-4,0) {};
\node (E1) at (-1.35,1.15) {};
\node (F1) at (-1.5,1.75) {};
\node (G1) at (-2.35,1.55) {};
\node (E2) at (-1.35,-1.15) {};
\node (F2) at (-1.5,-1.75) {};
\node (G2) at (-2.35,-1.55) {};
\end{scope}

\draw (A) -- (B1) -- (C1) -- (A) -- (E1) -- (F1) -- (G1) -- (C1) -- (E1) -- (G1) -- (B1) -- (F1) -- (A) -- (B2) -- (D) -- (C1) -- (C2) -- (D) -- (B1) -- (B2) -- (C2) -- (A) -- (E2) -- (F2) -- (G2) -- (C2) -- (E2) -- (G2) -- (B2) -- (F2) -- (A);
\end{tikzpicture}
\caption{(Left) Even if the rightmost large vertex is added to the dominating set for free, it still requires $3$ vertices to dominate the remaining $11$. (Middle) Identifying several copies of these by the large vertex and adding edges between the rectangular vertices ($\blacksquare$) to make it $3$-connected yields a graph class with $n=11k+1$ and $\gamma = 3k$. (Right) An even graph with $\gamma = n/4$. \label{fig:triconnected}} 
\end{figure}
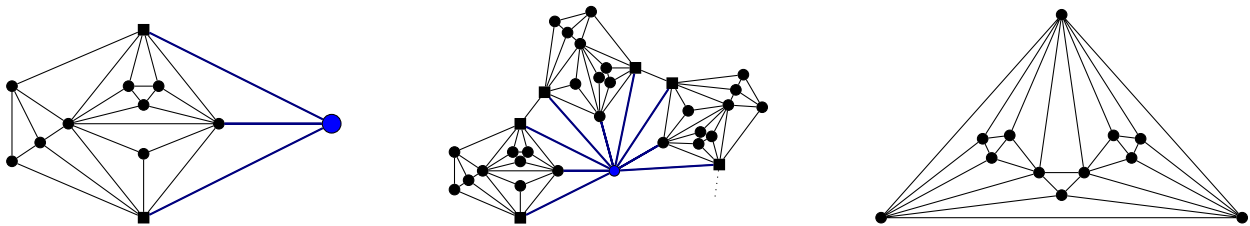

\vspace{-1em}
\paragraph{Algorithm}
We complement our upper bound of Theorem~\ref{theo:main} and Corollary~\ref{coro:main} by a quadratic-time algorithm. The algorithm takes as input an $n$-vertex 
skeletal triangulation $G$ with $\Phi = n+(e+f+t)/2$, the algorithm outputs a dominating set of size $\le 2\Phi/7$. Particularly, if $G$ is a triangulation, it outputs a dominating set of size $\le 2n/7$.

\subsection{Further related work}
The original bound due to Matheson and Tarjan~\cite{matheson_dominating_1996} has been extended to other surfaces than the sphere.
In~\cite{HonjoKawarabayashi} and~\cite{Plummer} it is shown that the $\frac{n}{3}$ bound holds for a larger class of graphs, including those embedded on a torus, the Klein bottle, and the projective plane.
In~\cite{FuruyaMatsumoto}, these results are further extended to all triangulations embedded on a closed surface. 
A related question is 
to upper bound the domination number of planar graphs with small diameter~\cite{GoddardHenning,MacGillivray}. 
Here, it is shown that all sufficiently large planar graphs with diameter $3$ can be dominated by at most $6$ vertices. 

Total domination of a graph differs from domination in that every vertex must have a neighbour in the totally dominating set, regardless of whether the vertex belongs to the set itself. Lemanska, Zuazua, and Zylinski~\cite{LemanskaZZ17} study the total domination number of maximal outerplanar graphs, and show that $2n/5$ vertices suffice to totally dominate this class of graphs, as is also shown in~\cite{DorflingHJ16}. Similarly to our problem, they also have to consider reducing one side of a two-cut. Since maximal outerplanar graphs allow induction over the dual tree, this allows them to provide a simple and elegant proof. In~\cite{ClaverolGHHMMT21}, these bounds for maximal outerplanar graphs serve as a stepping stone for improving the bound for general triangulations; namely via what in retrospect can be interpreted as a form of \emph{attachment} as the ones introduced in the paper at hand. 


\section{Main Techniques}\label{sec:techniques}

In this section, we develop the techniques needed for our proof.
Our proof is by induction and, roughly speaking, consists of two main parts.
In the first part, we deal with any small cuts, such as bridges, cut vertices or chords.
Here, a general classification scheme for dominating sets along small cuts allows us to replace
one side of the cut by a finite list of ``minimal attachments'', which can then be checked by hand.
In the second part, $G$ is 3-connected, and we want to add a vertex to our dominating set
and delete many of its neighbors without creating too many problematic configurations.
Here, some of our minimal attachments help keep track of which vertices are in the dominating
set or are already dominated. Also, in many cases, these deletions create bridges or cut vertices,
which necessitates working with skeletal triangulations.

\subsection{Skeletal triangulations with small vertex cuts}

\paragraph{Fusing}
Let $u \in G$ be a cut vertex in a skeletal triangulation,
whose removal splits $G$ into two components $C_1, C_2$.
For $i \in \{1, 2\}$, let $G_i$ be the graph induced by $C_i \cup \{u\}$
and let $u_i \in G_i$ correspond to $u \in G$. See Figure~\ref{fig:fusing} on page~\pageref{fig:fusing}.
If $\deg_{G_i}(u) = 1$, then $G_i$ need not be a skeletal triangulation, but this is the only obstruction.

\begin{definition}
A \emph{rooted skeletal triangulation} $(G, u)$ \emph{with root $u$} is a connected triangulated planar graph in which every vertex
except possibly $u$ has degree $\ge 2$.
\end{definition}

In the above setting, $(G_1, u_1)$ and $(G_2, u_2)$ are both rooted skeletal triangulations.
The following operation reconstructs $G$ from $G_1$ and $G_2$:

\begin{definition}[Fusing]
Let $(G_1, u_1)$ and $(G_2, u_2)$ be rooted skeletal triangulations.
We \emph{fuse} $(G_1, u_1)$ to $(G_2, u_2)$
by taking the disjoint union $G_1 \sqcup G_2$ and identifying $u_1$ with $u_2$.
\end{definition}

In the above setting, $G$ is the graph obtained by fusing $(G_1, u_1)$ to $(G_2, u_2)$.

\paragraph{Classifying dominating sets} In the same setting, let $S \subseteq G$ be a dominating set.
Put $S_1 = S \cap G_1$, then $S_1$ dominates all vertices in $G_1 - u_1$, and $u_1$ is either (a) contained in $S_1$, (b) dominated by $S_1$, or (c) not dominated by $S_1$.
Intuitively speaking, extending $S_1$ to a small dominating set in $G$ is easiest in case (a) and most difficult in case (c).
In fact, this case distinction perfectly describes which vertices in $G_2$ still have to be dominated.
This motivates the following definition.

\begin{definition}[Acts as]
A \emph{rooted dominating set} $S$ in $(G, u)$ is a set that dominates every vertex except maybe $u$.
$\gamma(G, u)$ denotes the size of a minimum rooted dominating set.
In the following, each case excludes all the previous ones.
We say $G$ \emph{acts as}
\begin{itemize}[itemsep=1pt, parsep=0pt, topsep=4pt, itemindent=16pt]
    \item [AB] if $G$ has a rooted dominating set of size $\gamma(G, u)$ that contains $u$,
    \item [LR] if $G$ has a rooted dominating set of size $\gamma(G, u)$ that dominates $u$, and
\item [Nope] otherwise.
\end{itemize}
\end{definition}
Figure~\ref{fig:smallfuse} on page~\pageref{fig:smallfuse} depicts the smallest rooted skeletal triangulation of each act-as type.
The following lemma illustrates how small ABs and small LRs can be used to ``remember'' that
certain vertices are required to be in $S$ or are already dominated.

\begin{lemma}[Forcing and Covering] \label{lemm:forcecover}
Let $G$ be a skeletal triangulation, with boundary vertex $u$. Let $s \in \NN$ be arbitrary.
\begin{enumerate}
\item Let $H_1$ be obtained by fusing a small AB to $(G, u)$. Then, $H_1$ has a dominating set of size $s$ if and only if $G$ has a dominating set of size $s$ that contains $u$.
\item Let $H_2$ be obtained by fusing a small LR to $(G, u)$. Then, $H_2$ has a dominating set of size $s+1$ if and only if $G$ has a set of size $s$ that dominates all vertices except maybe $u$.
\end{enumerate}
\end{lemma}
\begin{proof}
Straight-forward.
\end{proof}

The following lemma illustrates that rooted skeletal triangulations of the same act-as type are essentially interchangeable.
This enables us to use a divide-and-conquer approach later on.
\begin{lemma}[Fusing replacement] \label{lemm:fusereplace}
Let $(G_1, u_1), (G_2, u_2)$ be skeletal triangulations.
Let $H$ be obtained by fusing $(G_2, u_2)$ to $(G_1, u_1)$.
Then $\gamma(H) - \gamma(G_2, u_2)$ depends only on $(G_1, u_1)$ and on the act-as type of $(G_2, u)$
(but not on the precise graph-structure of $G_2$).
\end{lemma}
\begin{proof}
Looking at suitable rooted dominating sets shows that
\[
\gamma(H) = \gamma(G_1, u_1) + \gamma(G_2, u_2) + c
\]
where $c=-1$ if both $G_1$ and $G_2$ act as AB, $c=1$ if both act as Nope, and $c=0$ otherwise.
\end{proof}

\paragraph{Near-triangulations and chords} The above machinery allows us to deal with bridges and cut vertices.
For chords (=2-vertex-cuts) in near triangulations, we use similar techniques:
\begin{definition}
A \emph{rooted near-triangulation} $(G, u, v)$ is a near-triangulation $G$ with boundary edge $\{u, v\}$.
A \emph{rooted dominating set} $S \subseteq G$ is a set that dominates every vertex except maybe $u$ and $v$.
$\gamma(G, u, v)$ denotes the size of a minimum rooted dominating set.
If $(G_1, u_1, v_1)$ and $(G_2, u_2, v_2)$ are rooted near-triangulation,
then we \emph{attach} the latter to the former by taking the disjoint union $G_1 \sqcup G_2$ and identifying
$u_1$ with $u_2$ and $v_1$ with $v_2$.
\end{definition}
The attaching operation creates a chord, see Figure~\ref{fig:attaching} on page~\pageref{fig:attaching}.
Having two root vertices greatly increases the number of acts-as types.
\begin{definition}
Let $(G, u, v)$ be a rooted near-triangulation. Let $\gamma = \gamma(G, u, v)$.
In the following, each case excludes all the previous ones. We say $G$ \emph{acts as}
\begin{itemize}[itemsep=1pt, parsep=0pt, topsep=4pt, itemindent=16pt]
\item[A+B] if $G$ has a dominating set of size $\gamma$ that contains $u$ and $v$,
\item[OR] if $G$ has two dominating sets of size $\gamma$, one that contains $u$ and one that contains $v$,
\item[A] if $G$ has a dominating set of size $\gamma$ that contains $u$,
\item[B] if $G$ has a dominating set of size $\gamma$ that contains $v$,
\item[AND] if $G$ has a dominating set of size $\gamma$ and a dominating set of size $\gamma+1$ that contains both $u$ and $v$,
\item[L+R]if $G$ has a dominating set of size $\gamma$,
\item[OCTA] if $G$ has two rooted dominating set of size $\gamma$, one that dominates $u$ and one that dominates $v$, plus
a dominating set of size $\gamma+1$ that contains $u$ and $v$.
\item[L OR R] if $G$ has two rooted dominating sets of size $\gamma$, that, respectively, dominate $u$ and $v$, 
\item[L] if $G$ has a rooted dominating set of size $\gamma$ that dominates $u$,
\item[R] if $G$ has a rooted dominating set of size $\gamma$ that dominates $v$, and
\item[None] if otherwise.
\end{itemize}
\end{definition}
Here, the list of cases considered is tailored to our proof and non-exhaustive.
For example, in the L, R, None cases, we could also distinguish whether $G$ has a dominating
set of size $\gamma+1$ that contains both $u$ and $v$. Figure~\ref{fig:smallattach} on page~\pageref{fig:smallattach} gives an example of each case.

\paragraph{General $k$-vertex cuts} The idea of considering ``rooted'' instances is a technical contribution that we hope has applications in other 
classes of graphs. It can be generalized to $k$-vertex cuts for
any $k \ge 1$: Pick $k$ distinguished vertices.
For each of those vertices, we may (a) require it to be in the dominating set, (b) require it to be dominated, or (c) not require anything.
This yields $3^k$ combinations of restrictions in total.
The acts-as type is the $3 \times \dots \times 3$ tensor that describes how much each restriction increases the size of a minimum rooted dominating set.
One can show that the entries in such a tensor decrease along each dimension and decrease by at most 1 at a time, and that the number of such tensors is $\le 3^{k \cdot 2^{(k-1)}}$. 

\subsection{Penalty functions}
To facilitate a divide-and-conquer approach that deals with bridges, cut vertices and chords,
we want to generalize Theorem~\ref{theo:main} to the rooted setting.
There, we should only count problematic configurations that remain even after a fusing operation.

\begin{definition}[Penalty function]
    If $G$ is a skeletal triangulation on $n$ vertices, define $\Phi(G) = n + e/2 + f/2 + t/2$ where $e, f$ and $t$ are the number of ears, bad 5-wheels and degree-2 cut vertices in $G$, respectively.
    
    If $(G, u)$ is a rooted skeletal triangulation on $n+1$ vertices, define $\phi(G, u) = n + e/2 + f/2 + t/2 + r/2$ where $e$ is the number of ears containing a degree-2 vertex other than $u$,
    $f$ is the number of bad 5-wheels with a 3-pair disjoint from $u$, $t$ is the number of degree-2 cut vertices not equal to $u$, and $r$ is $1$ if $\deg_G(u) = 1$ and zero otherwise.
    
    If $(G, u, v)$ is a rooted near-triangulation on $n+2$ vertices, define $\phi(G, u, v) = n + e/2 + f/2 + t/2$ where $e$ is the number of ears containing a degree-2 vertex other than $u$ or $v$,
    $f$ is the number of bad 5-wheels with a 3-pair disjoint from $\{u, v\}$, and $t$ is the number of degree-2 cut vertices not equal to $u$ or $v$.
\end{definition}
The following properties follow immediately from the definitions:
\begin{align*}
    \phi(G, u)+1 &\le \Phi(G) \le \phi(G, u) + 1.5\\
    \phi(G, u, v) + 2 &\le \Phi(G) \le \phi(G, u, v) + 2.5
\end{align*}
If $G$ is obtained by attaching $(G_2, u_2, v_2)$ to $(G_1, u_1, v_1)$, then
\[
    \Phi(G) = \phi(G_1, u_1, v_1) + \phi(G_2, u_2, v_2) + 2.
\]
If $G$ is obtained by fusing $(G_2, u_2)$ to $(G_1, u_1)$, then
\[
    \Phi(G) \le \phi(G_1, u_1) + \phi(G_2, u_2) + 1,
\]
with equality if $\deg_{G_1}(u_1) \ne 1 \ne \deg_{G_2}(u_2)$.

Theorem~\ref{theo:main} states that, for any skeletal triangulation $G$ on $n > 10$ vertices, $\gamma(G) \le \lfloor \Phi(G) / 3.5 \rfloor$.
Using Theorem~\ref{theo:main}, we can show the following:
\begin{corollary}[Skeletal triangulation acts-as bounds] \label{coro:skeletalbound}
    Let $(G, u)$ be a rooted skeletal triangulation. If $(G, u)$ acts as
    \begin{itemize}[itemsep=1pt, parsep=0pt, topsep=4pt, itemindent=16pt]
        \item[AB] then $\phi(G, u) \ge 3.5 \cdot \gamma(G, u) - 1$,
        \item[LR] then $\phi(G, u) \ge 3.5 \cdot \gamma(G, u)$, 
        \item[Nope] then $\phi(G, u) \ge 3.5 \cdot \gamma(G, u) + 1.5$.
    \end{itemize}
\end{corollary}
\begin{proof}
    If $G$ has $\le 10$ vertices, check by hand. In practice, only three specific triangulations have to be checked.
    Suppose $(G, u)$ acts as AB. Let $H$ be obtained by attaching a small LR to $(G, u)$. Then
    $\Phi(H) \le \phi(G, u) + 3.5 + 1$ and $\gamma(H) = \gamma(G, u) + 1$. Theorem~\ref{theo:main} yields
    \[
        \phi(G, u) \ge \Phi(H) - 4.5 \ge 3.5 \gamma(H) - 4.5 = 3.5\gamma(G, u) - 1.
    \]
    Suppose $(G, u)$ acts as LR. Let $H$ be obtained by attaching a small AB to $(G, u)$. Then
    $\Phi(H) \le \phi(G, u) + 2.5 + 1$ and $\gamma(H) = \gamma(G, u) + 1$. Theorem~\ref{theo:main} yields
    \[
        \phi(G, u) \ge \Phi(H) - 3.5 \ge 3.5 \gamma(H) - 3.5 = 3.5 \gamma(G, u).
    \]
    Suppose $(G, u)$ acts as Nope. If $\deg_{G}(u) \ne 1$, then $G$ is a skeletal triangulation, with
    $\Phi(G) \le \phi(G, u) + 1.5$ and $\gamma(G) = \gamma(G, u) + 1$. Theorem~\ref{theo:main} yields
    \[
        \phi(G, u) \ge \Phi(G) - 1.5 \ge 3.5 \gamma(G) - 1.5 = 3.5 \gamma(G, u) + 2
    \]
    If $\deg_G(u) = 1$, then let $v$ be the neighbor of $u$ and let $H = G - u$.
    Then $(H, v)$ acts as LR, $\phi(H, v) \le \phi(G, u) - 1.5$ and $\gamma(H, u) = \gamma(G, u)$. The LR case yields 
    \[
        \phi(G, u) \ge \phi(H, v) + 1.5 \ge 3.5 \gamma(H, v) + 1.5 = 3.5 \gamma(G, u) + 1.5.\qedhere
    \]
\end{proof}

Note that these bounds are tight in the examples in Figure~\ref{fig:smallfuse} on page~\pageref{fig:smallfuse}.
For rooted near-triangulations, there are analogous bounds, but those are not tight in all cases.

\begin{corollary}[Near-triangulation acts-as bounds] \label{coro:nearbound}
    Let $(G, u, v)$ be a rooted near-triangulation. If $(G, u, v)$ acts as
    \begin{itemize}[itemsep=1pt, parsep=0pt, topsep=4pt, itemindent=140pt]
        \item [A+B, OR] then $\phi(G, u, v) \geq 3.5 \cdot \gamma(G, u, v) - 2$,
        \item [A, B] then $\phi(G, u, v) \geq 3.5 \cdot \gamma(G, u, v) - 1$,
        \item [AND, L+R, OCTA, L OR R] then $\phi(G, u, v) \geq 3.5 \cdot \gamma(G, u, v)$,
        \item [L, R] then $\phi(G, u, v) \geq 3.5 \cdot \gamma(G, u, v) + 0.5$,
        \item [None] then $\phi(G, u, v) \geq 3.5 \cdot \gamma(G, u, v) + 1.5$.
    \end{itemize}
\end{corollary}
\begin{proof}
    Similar to the proof of Corollary~\ref{coro:skeletalbound}. We omit the details.
\end{proof}

\paragraph{Lower-bound examples} These bounds suggest that the most efficient building blocks for
lower-bound examples typically act as A+B  or OR. Indeed, the building blocks in Figure~\ref{fig:examples} both act as A+B.
The left building block in Figure~\ref{fig:triconnected} acts as A on the bottom edge and as B on the top edge,
and was found by enumerating 3-connected near-triangulations with Plantri~\cite{brinkmann2007fast} and filtering for large domination numbers and
interesting combinations of acts-as types. Filtering for acts-as types enables us to find this building block at $n=12$ already,
even though the constructed example only exceeds the $\lfloor n/4 \rfloor$ bound at $n=21$.

\subsection{The divide-and-conquer technique} \label{sect:divcong}
Consider a skeletal triangulation.
The machinery we introduced so far allows us to assume that, for any bridge, cut vertex, or 2-vertex-cut in $G$,
one side of the cut has constant size. We illustrate this in the case of cut vertices.

Let $G$ be a skeletal triangulation obtained by fusing $(G_2, u_2)$ to $(G_1, u_1)$.
Suppose, for example, that $(G_2, u_2)$ acts as AB.
Let $H$ be obtained by fusing a small AB, denoted $(H_2, v)$, to $(G_1, u_1)$.
Then, by Lemma~\ref{lemm:fusereplace},
\[
    \gamma(G) = \gamma(H) + \gamma(G_2, u_2) - \gamma(H_2, v).
\]
The small AB satisfies $\phi(H_2, v) = 3.5 \cdot \gamma(H_2, v) - 1$.
By Corollary~\ref{coro:skeletalbound}, $\phi(G_2, u_2) \ge 3.5 \cdot \gamma(G_2, u_2) - 1$. Therefore,
\begin{align*}
    \Phi(G) - \Phi(H) &\ge \phi(G_1, u_1) + \phi(G_2, u_2) + 1 - \Big(\phi(G_1, u_1) + \phi(H_2, v) + 1 \Big)\\
    &= \phi(G_2, u_2) - \phi(H_2, v)\\
    &\ge 3.5 \Big(\gamma(G_2, u_2) -  \gamma(H_2, v) \Big) = 3.5 \Big(\gamma(G) - \gamma(H) \Big)
\end{align*}
In particular, if $H$ satisfies Theorem~\ref{theo:main}, i.e. if $\Phi(H) \ge 3.5 \gamma(H)$,
then so does $G$.

In the actual proof, some care has to be taken to avoid circular arguments inside the induction step.
For example, in order to use Corollary~\ref{coro:skeletalbound}, $G_2$ should not be a small AB / LR / Nope.

\subsection{Dealing with the 3-connected case} \label{sect:threecontricks}
Once $G$ is 3-connected, we manually pick specific (high-degree) vertices to be in the dominating set and then delete
the picked vertices and sufficiently many of their neighbors. Intuitively, this should always be possible
by looking at a large enough section of the graph, given that
we expect 3-connected near-triangulations to satisfy $\gamma \le 3n / 11$, which is a bit stronger than the $2n/7$ bound we are aiming for.

The main difficulty is that deleting a high-degree vertex may yield many problematic configurations, increasing $\Phi$. Even worse, the graph might get separated into many small components, for which Theorem~\ref{theo:main} on longer holds. We deal with these issues in three different ways: (1) delete edges instead of vertices. Deleting an edge only affects the two incident vertices, which is much easier to handle than a vertex deletion. For example, a (non-bridge) boundary edge between two vertices of degree $\ge 5$ may always be deleted,
as this never creates any problematic configurations. (2) whenever we delete vertices, fuse a small LR to any vertex that is already dominated. This gets rid of any problematic configurations caused by that vertex. (3) when picking a vertex to be in the dominating set, instead of deleting that vertex, fuse a small AB to it. Fusing a small AB increases $\Phi$ by $2.5$ while (often) increasing $\gamma$ by $1$. This has essentially the same effect as decreasing $\Phi$ by $1$ by deleting the vertex, but avoids the aforementioned issues around vertex deletions.

Nevertheless, this part of our proof contains many cases. This is likely unavoidable:
Since Theorem~\ref{theo:main} does not hold for $n=10$, our proof needs
to look at a large enough piece of the graph to avoid a specific 10-vertex example,
see Theorem~\ref{theo:mainprecise}.

\section{A sketch of the full proof}\label{sec:sketch}

The precise version of Theorem~\ref{theo:main} is the following.
\begin{theorem}\label{theo:mainprecise}
    Let $G$ by a skeletal triangulation that is not the 3-bifan, octahedron or the special 4343434-heptagon (see Figure~\ref{fig:sporadic}). 
    Then $\gamma(H) \le \lfloor \frac{\Phi(G)}{3.5} \rfloor$.
\end{theorem}
We prove this via induction, using a carefully chosen partial ordering on skeletal triangulations.
\begin{definition}
    Let $G$ and $H$ be skeletal triangulations. We say $G$ is \emph{smaller} than $H$ if, in decreasing order of importance: (1) $G$ has fewer interior vertices (than $H$), (2) $G$ has fewer bridges, (3) $G$ has smaller $\Phi$, (4) $G$ has fewer blocks (2-connected components), (5) $G$ has fewer vertices, (6) $G$ has fewer degree-2 vertices.
\end{definition}

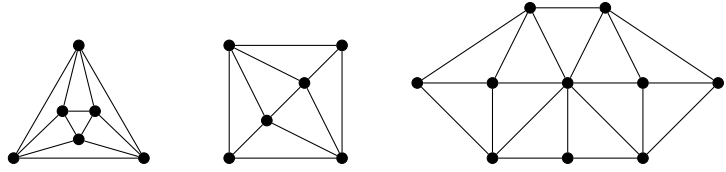
\begin{figure}
	\begin{center}
		\begin{tikzpicture}
			\begin{scope}[scale=0.5]
				\begin{scope}[shift={(0, 1)}]
					\node (A) at (210:2) {};
					\node (B) at (-30:2) {};
					\node (C) at  (90:2) {};
					\node (X) at ( 30:0.5) {};
					\node (Y) at (150:0.5) {};
					\node (Z) at (-90:0.5) {};
					\draw (A) -- (B) -- (C) -- (A);
					\draw (A) -- (Z) -- (B) -- (X) -- (C) -- (Y) -- (A);
					\draw (X) -- (Y) -- (Z) -- (X);
				\end{scope}
				\begin{scope}[shift={(4,0)}]
					\node (A) at (0, 0) {};
					\node (B) at (3, 0) {};
					\node (C) at (3, 3) {};
					\node (D) at (0, 3) {};
					\node (X) at (1, 1) {};
					\node (Y) at (2, 2) {};
					\draw (A) -- (B) -- (C) -- (D) -- (A);
					\draw (A) -- (X) -- (Y) -- (C);
					\draw (B) -- (X) -- (D) -- (Y) -- (B);
				\end{scope}
				\begin{scope}[shift={(13, 2)}]
					\node (A) at (-1, 2) {};
					\node (B) at (-4, 0) {};
					\node (C) at (-2, -2) {};
					\node (D) at (0, -2) {};
					\node (E) at (2, -2) {};
					\node (F) at (4, 0) {};
					\node (G) at (1, 2) {};
					\node (X) at (-2, 0) {};
					\node (Y) at (0, 0) {};
					\node (Z) at (2, 0) {};
					\draw (A) -- (B) -- (C) -- (D) -- (E) -- (F) -- (G) -- (A) {};
					\draw (B) -- (X) -- (Y) -- (Z) -- (F) {};
					\draw (C) -- (X) -- (A) -- (Y) -- (G) -- (Z) -- (E) -- (Y) -- (C);
					\draw (D) -- (Y);
				\end{scope}
			\end{scope}
		\end{tikzpicture}
	\end{center}
	\caption{From left to right: octahedron, $3$-bifan, special 4343434-heptagon.}
	\label{fig:sporadic}
\end{figure}

\noindent
Here is a rough sketch of how we prove the theorem.
\begin{enumerate}[itemsep=1pt, parsep=0pt, topsep=4pt]
    \item If $G$ has a bridge, apply Corollary~\ref{coro:skeletalbound} to both sides of the bridge, then check the 9 combinations of AB / LR / Nope. The corollary may be used as both sides have fewer bridges (and at most as many interior vertices) as $G$, hence Theorem~\ref{theo:main} holds for these graphs by induction. Conclusion: $G$ has no bridges.
    \item If $G$ has a cut vertex, use Section~\ref{sect:divcong} to replace one side by a small AB / LR / Nope. Then, replace the AB by an A attachment, delete the Nope and LR. In the LR case, delete the cut vertex too if there is a problematic configuration. This is justified by Lemma~\ref{lemm:forcecover}. Conclusion: $G$ has no cut vertices.
    \item If $G$ has a chord and one side acts as AND, L+R, OCTA, L OR R, L, R, None, use Corollary~\ref{coro:nearbound} to bound that side and delete it, possibly together with one of the endpoints of the chord.
    Some of the bounds in Corollary~\ref{coro:nearbound} are not tight on any small example, so we cannot just replace these attachments by small ones.
    \item If $G$ has a chord, then one side acts as A+B, OR, A, B. Replace that side by a small OR, OR, A, B.
    (Here, Corollary~\ref{coro:nearbound} is tight.)
    \item Handle small As and Bs by (a) deleting boundary edges leading to high-degree vertices and (b) deleting neighboring low-degree vertices that are dominated by the ``forced'' vertex in the A / B.
    \item Handle small ORs by considering many cases. After this step, we may conclude: $G$ is 3-connected as there are no chords.
    \item Try deleting any boundary edge without creating problematic configurations. If this does not work, then the boundary of $G$ consists of problematic configurations that are ``covered'' by a single edge. After this step, we conclude: $G$ has many degree-3 boundary vertices. Moreover, the degrees on the boundary of $G$ follow one of the following patterns: $345^{+}43$, $345^{+}3$, $35^{+}43$, $35^{+}3$, $34443$, $3443$, $343$, $33$.
    \item Handle the degree patterns $345^{+}43$, $345^{+}3$, $35^{+}43$ and $34443$, followed by $33$, followed by $35^{+}3$ and $3443$. This involves checking many cases by hand. Using the techniques from Section~\ref{sect:threecontricks}, this is not difficult, but it is a bit tedious. Conclusion: only the degree patterns $3443$ and $343$ remain.
    \item Handle the remaining cases while avoiding the 3-bifan, octahedron and special 4343434-heptagon. If $G$ has many boundary vertices, this is easy, but if $G$ has few vertices, we have to be careful to avoid these examples.
\end{enumerate}



\section{Full Proof in Detail}\label{sect:actasrefine}

In the following, we will provide the details sketched in Sections~\ref{sec:techniques} and \ref{sec:sketch}. 

\tored{Make this connect with the previous chapters. E.g. the definition of ears is different.}

\subsection{Attaching}

``Attaching'' small near-triangulations to a boundary edge of a given planar graph turns out to be
a useful tool for manipulating dominating sets. The natural way of doing
this is by creating a 2-cut.

\toyellow{Add some references to the later section to the first 10 pages?}
\toyellow{Replace $s$ by $\gamma$?}
\begin{definition}[Rooted near-triangulation]
    A \emph{rooted near-triangulation} $G = (G, u, v) = (V, E, u, v)$ \emph{with base} $u, v$ is
    a near-triangulation $G$ with a boundary edge $u, v$. The \emph{base vertices} are $u, v$ and the
    \emph{base edge} is $\{u, v\}$.
\end{definition}

\begin{definition}[Attaching] \label{def:attaching}
    Let $G_1$ be a skeletal triangulation with boundary edge $u_1, v_1$.
    Let $G_2 = (G_2, u_2, v_2)$ be a rooted near-triangulation.
    We can \emph{attach} $G_2$ to $u_1, v_1$ as follows: Consider the disjoint union $G_1 \sqcup G_2$
    and identify $u := u_1 = u_2$ and $v := v_1 = v_2$.
\end{definition}

\begin{figure}
    \begin{center} \begin{tikzpicture} \begin{scope}[scale=0.7]
        \newcommand{\doit}[2]{
            \node (AA) at (-3, 1.7) {};
            \node (A) at (-2, 0) {};
            \node (B) at (-1, 1.7) {};
            \node (C) at (0, 0) {};
            \node (D) at (1, 1.7) [#1] {};
            \node (E) at (2, 0) [#2] {};
            \node (F) at (1, 0.6) {};
            \draw (AA) -- (A) -- (B) -- (C) -- (D) -- (E) -- (C) (B) -- (AA);
            \draw (F) -- (C) (F) -- (D) (F) -- (E);
        }
        \begin{scope}
            \doit{label=above right:$u_1$}{label=above right:$v_1$}
        \end{scope}
        \begin{scope}[shift={(3.5, 0)}]
            \node (T) at (0, 0) [label=below:$u_2$] {};
            \node (U) at (2, 0) [label=below:$v_2$] {};
            \node (V) at (2, 1.3) {};
            \node (W) at (0, 1.3) {};
            \draw (T) -- (U) -- (V) -- (W) -- (T);
            \draw (U) -- (W);
        \end{scope}
        \begin{scope}[shift={(10, 0)}]
            \doit{label=above left:$u$}{label=below right:$v$}
            \node (V) at (3, 0.8) {};
            \node (W) at (2, 2.45) {};
            \draw (D) -- (W) -- (V) -- (E);
            \draw (E) -- (W);
        \end{scope}

    \end{scope} \end{tikzpicture} \end{center}
    \caption{A skeletal triangulation with boundary edge $\{u_1, v_1\}$, a rooted near triangulation with base $u_2, v_2$ and the result of attaching the later to $u_1, v_1$.}
    \label{fig:attaching}
\end{figure}
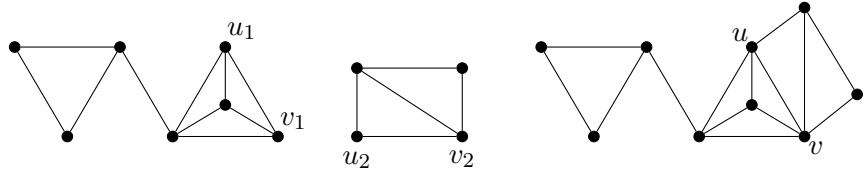

See Figure~\ref{fig:attaching} for an example.
The resulting graph $G$ is a (unrooted) skeletal triangulation with chord $u, v$,
with one ``side'' (including $u, v$) being isomorphic to $G_1$ and the other side isomorphic to $G_2$.
If $G_1$ was a near-triangulation, then so is $G$.
The following generalization of dominating sets behaves well with regards to
attaching.

\begin{definition}[Rooted dominating set]
    A \emph{rooted dominating set} of a rooted near-triangulation with base $u, v$
    is a subset $S \subseteq V$ such that $N[S] \supseteq V \setminus \{u, v\}$,
    i.e. such that every vertex except maybe $u, v$ is either in $S$ or has a neighbor in $S$.
    We denote the size of a minimum rooted dominating set by $s(G, u, v)$.
\end{definition}
\begin{lemma} \label{lemm:attachdom}
    In the setting of Definition~\ref{def:attaching},
    let $S_1$ be a dominating set in $G_1$ and let $S_2$ be a rooted dominating set in $G_2$.
    Then $S_1 \cup S_2$ is a dominating set in the graph $G$ obtained by attaching $G_2$ to $u_1, v_1$.
    In particular, $s(G) \le s(G_1) + s(G_2, u, v)$.
\end{lemma}
\begin{proof} Straight-forward. \end{proof}

\subsection{Fusing}

``Fusing'' is the analog of attaching, but now both graphs are skeletal triangulations
and we identify only a single vertex.
The natural way of doing this is by creating a 1-cut. 
We want to entertain the possibility of creating a bridge or degree-2 cut vertex this way,
so we allow the ``fused'' vertex to have degree one.

\begin{definition}[Rooted skeletal triangulation]
    A \emph{rooted skeletal triangulation} $G = (G, u) = (V, E, u)$ \emph{with root $u$}
    is a connected planar graph in which every bounded face is a triangle and every vertex
    except maybe $u$ has degree $\ge 2$.
\end{definition}

\begin{definition}[Fusing] \label{def:fusing}
Let $G_1$ be a skeletal triangulation with boundary vertex $u_1$ or a rooted
skeletal triangulation with root $u_1$.
Let $G_2$ be a rooted skeletal triangulation with root $u_2$.
We can \emph{fuse} $G_2$ to $u_1$ as follows: Consider the disjoint union $G_1 \sqcup G_2$
and identify $u := u_1 = u_2$.
\end{definition}

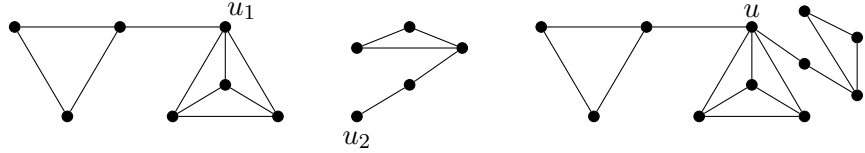
\begin{figure}
    \begin{center} \begin{tikzpicture} \begin{scope}[scale=0.7]
        \newcommand{\doit}[1]{
            \node (AA) at (-3, 1.7) {};
            \node (A) at (-2, 0) {};
            \node (B) at (-1, 1.7) {};
            \node (C) at (0, 0) {};
            \node (D) at (1, 1.7) [#1] {};
            \node (E) at (2, 0) {};
            \node (F) at (1, 0.6) {};
            \draw (AA) -- (A) -- (B) -- (D) (C) -- (D) -- (E) -- (C) (B) -- (AA);
            \draw (F) -- (C) (F) -- (D) (F) -- (E);
        }
        \begin{scope}
            \doit{label=above right:$u_1$}
        \end{scope}
        \begin{scope}[shift={(3.5, 0)}]
            \node (T) at (0, 0) [label=below:$u_2$] {};
            \node (U) at (1, 0.6) {};
            \node (V) at (2, 1.3) {};
            \node (W) at (0, 1.3) {};
            \node (X) at (1, 1.7) {};
            \draw (T) -- (U) -- (V) -- (W) -- (X) -- (V);
        \end{scope}
        \begin{scope}[shift={(10, 0)}]
            \doit{label=above:$u$}
            \node (W) at (3, 0.4) {};
            \node (V) at (2, 1) {};
            \node (X) at (3, 1.5) {};
            \node (Y) at (2, 2) {};
            \draw (D) -- (V) -- (W) -- (X) -- (Y) -- (W);
        \end{scope}

    \end{scope} \end{tikzpicture} \end{center}
    \caption{A skeletal triangulation with boundary vertex $\{u_1\}$, a rooted skeletal triangulation with base $u_2$ and the result of fusing the later to $u_1$.}
    \label{fig:fusing}
\end{figure}

See Figure~\ref{fig:fusing} for an example.
The resulting graph $G$ is a skeletal triangulation
with cut vertex $u$, with each block in $G$ corresponding to a block in exactly one of $G_1, G_2$.
The following generalization of dominating sets behaves well with regards to
fusing.

\begin{definition}[Rooted dominating set]
    A \emph{rooted dominating set} in a rooted skeletal triangulation $G$ with root $u$
    is a subset $S \subseteq V$ such that $N[S] \supseteq V \setminus \{u\}$,
    i.e. such that every vertex except maybe $u$ is either in $S$ or has a neighbor in $S$.
    We denote the size of a minimum rooted dominating set by $s(G, u)$.
\end{definition}
\begin{lemma}\label{lemm:fusedom}
    In the setting of Definition~\ref{def:fusing},
    let $S_1$ be a dominating set in $G_1$ and let $S_2$ be a rooted dominating set in $G_2$,
    then $S_1 \cup S_2$ is a dominating set in the graph $G$ obtained by fusing $G_2$ to $u_1$.
    In particular, $s(G) \le s(G_1) + s(G_2, u_2)$.
\end{lemma}
\begin{proof} Trivial. \end{proof}

\subsection{The Penalty Function}

Let $G$ be a skeletal triangulation, rooted skeletal triangulation or rooted near-triangulation.

\begin{definition}[Cluster]
    Let $G$ be a graph and let $P$ be some property that a vertex in $G$ may or may not have,
    e.g. being adjacent to a fixed vertex $u$ or having a certain degree.
    A \emph{cluster} in $G$ is a maximal connected subgraph consisting only of vertices that satisfy property $P$.
\end{definition}

\begin{definition}[Ears and Pivoting Triangles]
    An \emph{Ear} in $G$ is a facial triangle with exactly one vertex of degree two,
    called the \emph{ear tip}.
    Equivalently, an ear tip is a size-1 cluster of degree-2 non-cut vertices.

    A \emph{pivoting triangle} in $G$ is a facial triangle with exactly two vertices of degree two.
    Equivalently, a pivoting triangle is a size-2 cluster of degree-2 non-cut vertices + their shared neighbor.

    An \emph{isolated triangle} in $G$ is a facial triangle with exactly three vertices of degree two.
    This implies the whole graph is a triangle. (See Figure \ref{fig:degtwoexample}.)
\end{definition}
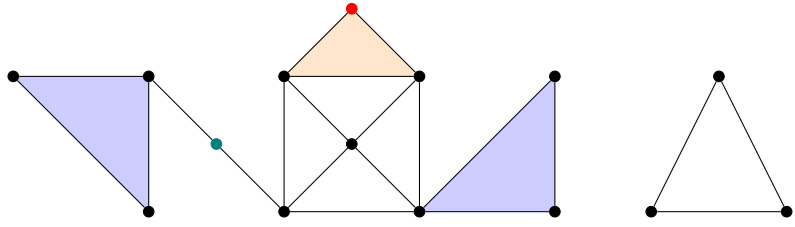
\begin{figure}
    \begin{center} \begin{tikzpicture}[scale=0.9]
        \node (A) at (0, 0) {};
        \node (B) at (2, 0) {};
        \node (C) at (2, 2) {};
        \node (D) at (0, 2) {};
        \node (E) at (1, 1) {};
        \node (T) at (1, 3) [red] {};
        \node (L) at (-1, 1) [teal] {};
        \node (M) at (-2, 2) {};
        \node (N) at (-4, 2) {};
        \node (O) at (-2, 0) {};
        \node (R) at (4, 0) {};
        \node (S) at (4, 2) {};

        \draw (A) -- (B) -- (C) -- (D) -- (A);
        \draw (A) -- (E) -- (C) (B) -- (E) -- (D);
        \draw (C) -- (T) -- (D);
        \draw (A) -- (L) -- (M);
        \draw (M) -- (N) -- (O) -- (M);
        \draw (B) -- (R) -- (S) -- (B);
        \begin{scope}[on background layer]
            \path [fill=blue, opacity=0.2] (M.center) to (N.center) to (O.center) to cycle;
            \path [fill=orange, opacity=0.2] (C.center) to (T.center) to (D.center) to cycle;
            \path [fill=blue, opacity=0.2] (B.center) to (R.center) to (S.center) to cycle;
        \end{scope}
    \end{tikzpicture}\hspace{1.0cm} \begin{tikzpicture}[scale=0.9]
        \node (A) at (0, 0) {};
        \node (B) at (2, 0) {};
        \node (C) at (1, 2) {};
        \draw (A) -- (B) -- (C) -- (A);
    \end{tikzpicture} \end{center}
    \caption{On the left: A skeletal triangulation with an ear (orange) with ear tip (red), two pivoting triangles (blue)
    and a degree-2 cut vertex (teal). On the right: An isolated triangle.}
    \label{fig:degtwoexample}
\end{figure}
\begin{definition}[Bad 5-wheel]
    A \emph{bad 5-wheel} in $G$ is a subgraph $H \subseteq G$ isomorphic to the 5-wheel
    such that the outer 4-cycle in $H$ contains two consecutive $G$-boundary vertices of degree $3$, called a 3-pair.
    A bad 5-wheel contains $1-4$ such 3-pairs. (See Figure~\ref{fig:badwheels}.)
\end{definition}

\begin{figure}
    \begin{center} \begin{tikzpicture}[scale=0.9]
        \node (A) at (0, 0) {};
        \node (B) at (2, 0) {};
        \node (C) at (2, 2) {};
        \node (D) at (0, 2) {};
        \node (E) at (1, 1) [label=above:$y$] {};
        \node (c) at (1, 0.3) {};
        \node (b) at (0.8, 0.5) {};
        \node (a) at (1.2, 0.5) {};
        \node (W) at (3, -1) [label=above:$z$] {};
        \node (X) at (2, -2) {};
        \node (Y) at (4, -2) {};
        \node (Z) at (4, 0) {};
        \node (I) at (0, -2) {};
        \node (J) at (-2, -2) {};
        \node (P) at (-2, 0) {};
        \node (Q) at (-2, 2) {};
        \node (R) at (-4, 2) {};
        \node (S) at (-4, 0) {};
        \node (O) at (-3, 1) [label=above:$x$] {};
        \node (M) at (6, -2) {};
        \node (N) at (6, 0) {};

        \draw[red,thick] (A) -- (B) -- (C) -- (D) -- (A) -- (E) -- (C);
        \draw[red,thick] (B) -- (E) -- (D);
        \draw (B) -- (X) -- (Y) -- (Z) -- (B) -- (W) -- (Y);
        \draw (X) -- (W) -- (Z);
        \draw (A) -- (I) -- (J) -- (P) -- (I);
        \draw[blue,thick] (P) -- (Q) -- (R) -- (S) -- (P);
        \draw[blue,thick] (P) -- (O) -- (R) (Q) -- (O) -- (S);
        \draw (Y) -- (M) -- (N) -- (Y);
        \draw (A) -- (c) -- (B) -- (a) -- (E) -- (b) -- (A);
        \draw (a) -- (b) -- (c) -- (a);
    \end{tikzpicture} \end{center}
    \caption{A skeletal triangulation with two bad 5-wheels: the blue one centered at $x$ and the
    red one centered at $y$. Note that there is no bad 5-wheel centered at $z$,
    as there are no two consecutive boundary vertices of degree $3$ on that 5-wheel.}
    \label{fig:badwheels}
\end{figure}
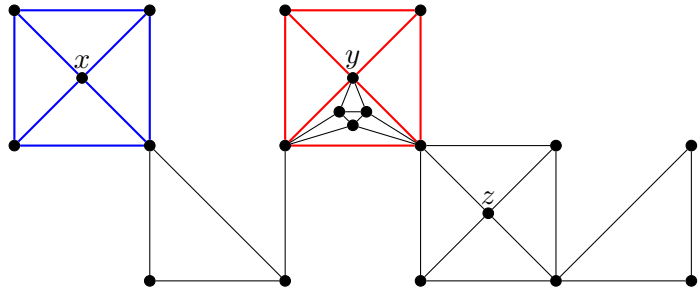
\begin{figure}
    \begin{center} \begin{tikzpicture}[scale=1.5]
        \node (O) at (90:0.5) [label=above:$x$] {};
        \node (A) at (0:1) {};
        \node (B) at (60:1) [label=above:$v_2$] {};
        \node (C) at (120:1) [label=above:$v_1$] {};
        \node (D) at (180:1) {};
        \node (E) at (240:1) {};
        \node (F) at (300:1) [label=below:$u$] {};
        \draw (O) -- (A) -- (B) -- (C) -- (D) -- (E) -- (F) -- (A) -- (E);
        \draw (O) -- (B) (O) -- (C) (O) -- (D) -- (A);
        \draw (A) -- (E);
    \end{tikzpicture}\end{center}
    \caption{A near-triangulation $G$ with an ear tip $u$ and a bad 5-wheel $N[x]$, in which $v_1, v_2$ are consecutive boundary vertices of degree $3$,
    i.e. a 3-pair.
    Note that $\{v_1, v_2\}$ is a cluster of $x$-adjacent degree-3 vertices.
    In terms of the penalty function, $n = 7$, $e=1$, $t=1$ and $\Phi=9$.
    If we instead consider this as a rooted near-triangulation with base $v_1, v_2$, then $n = 5$, $e=1$, $t = 0$ and
    $\phi(G, v_1, v_2) = 6$. If we consider it as a rooted skeletal triangulation with root $u$,
    then $n=6$, $e=0$, $t=1$ and $\phi(G, u) = 7$.}
    \label{fig:examplepenalty}
\end{figure}
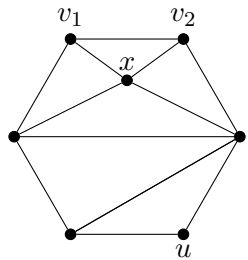

\begin{definition}[Penalty Function]
    Let $G$ be a skeletal triangulation. The penalty function is $\Phi = \Phi(G) = n + e/2 + f/2$
    where $n$ is the number of vertices, $e$ is the number of ears, pivoting triangles,
    isolated triangles and degree-2 cut vertices and $f$ is the number of bad 5-wheels.

    Let $(G, u, v)$ be a rooted near-triangulation. The penalty function
    is $\phi = \phi(G, u, v) = n + e/2 + f/2$ where $n$ is the number of non-$u, v$ vertices, $e$ is the number of
    ears with tip $\neq u, v$, and $f$ is the number of bad 5-wheels that contain 3-pair
    disjoint from $\{u, v\}$.
    
    Let $(G, u)$ be a rooted skeletal triangulation. The penalty function is
    $\phi = \phi(G, u) = n + e/2 + f/2 + r/2$
    where $n$ is the number of non-$u$ vertices, $e$ is the number of ears with tips $\neq u$,
    pivoting triangles (which may include $u$), isolated triangles
    and non-$u$ degree-2 cut vertices, $f$ is the number of bad 5-wheels that contain a 3-pair 
    disjoint from $\{u\}$
    and $r$ is $1$ if $\deg(u) = 1$ and $0$ otherwise.
\end{definition}
\begin{definition}
    A \emph{low-degree problem} in a (rooted) near triangulation or (rooted) skeletal triangulation
    is anything that contributes to $\Phi$ or $\phi$ other than the $n$ term.
    A vertex is \emph{involved in} a low-degree problem if it is a degree-2 vertex in
    an ear / pivoting triangle / isolated triangle / degree-2 cut vertex or if it is in a 3-pair in a bad 5-wheel.
\end{definition}

\begin{remark}\label{rem:noninvolved}
    In a skeletal triangulation, a vertex that is a cut vertex is never involved in any low-degree problems.
    In a near-triangulation a vertex that is incident to a chord is never involved in any low-degree problems.
\end{remark}

Note that (rooted) near-triangulations do not contain pivoting triangles;
all their low-degree problems are ears or bad 5-wheels.
See Figure~\ref{fig:examplepenalty} for an example involving one of each. 
The three penalty functions are closely related.
\begin{lemma}(De-rooting) \label{lemm:unroot}
    If $G$ is a skeletal triangulation with boundary vertex $u$ then
    \[
        \phi(G, u)+1 \leq \Phi(G) \leq \phi(G, u) + 1.5.
    \]
    Moreover, $\Phi(G) = \phi(G, u) + 1.5$ if and only if
    $u \in G$ is an ear tip, a degree-2 cut vertex, or is contained in every 3-pair of a bad 5-wheel in $G$.
    Otherwise, $\Phi(G) = \phi(G) + 1$.

    If $(G, u, v)$ is a rooted near-triangulation,
    then
    \[
        \phi(G, u, v)+2 \leq \Phi(G) \leq \phi(G, u, v) + 2.5
    \]
    Moreover, $\Phi(G) = \phi(G, u, v) + 2.5$ if and only if one or both of $u, v \in G$
    is an ear tip or $u, v$ is a 3-pair of a bad 5-wheel and $G$ is not an isolated 5-wheel.
    Otherwise, i.e. if neither case happens, then $\Phi(G) = \phi(G, u, v) + 2$.
\end{lemma}
\begin{proof}
    Trivial, but this lemma is very important, so you should check it.
\end{proof}
At first, it might seem a bit weird to have the $+1$ and $+2$ here, but this leads to nicer formulas when fusing and attaching:
\begin{lemma}(Detaching and defusing) \label{lemm:defuse}
    Let $G_1$ be a skeletal triangulation with boundary vertex $u_1$ and let $(G_2, u_2)$ be a rooted skeletal triangulation.
    Let $G$ be the graph obtained by fusing $G_2$ to $u_1$. Then
    \[
        \Phi(G) \leq \Phi(G_1) + \phi(G_2, u_2) \le \Phi(G) + 0.5.
    \]
    If moreover $\deg_{G_2}(u_2) \ne 1$, i.e. if $G_2$ is a skeletal triangulation, then
    \[
        \Phi(G) = \phi(G_1, u_1) + \phi(G_2, u_2) + 1
    \]

    Let $H_1$ be a skeletal triangulation with boundary edge $u_1, v_1$ and let $(H_2, u_2, v_2)$ be a rooted near-triangulation.
    Let $H$ be the graph obtained by attaching $H_2$ to $u_1, v_1$. Then
    \[
        \Phi(H) \le \Phi(H_1) + \phi(H_2, u_2, v_2) \le \Phi(H) + 0.5.
    \]
    If moreover $H_1$ is a near-triangulation, then
    \[
        \Phi(H) = \phi(H_1, u_1, v_1) + \phi(H_2, u_2, v_2) + 2.
    \]
\end{lemma}
\begin{proof}
    We make the following crucial observation: If there is a low-degree problem in $G$ (or $H$), then
    the same low-degree problem occurs in exactly one of $G_1$ and $G_2$. This needs exactly the right definition of a bad 5-wheel.
    The rest is straight-forward.
\end{proof}

\section{Setting up}

\subsection{The Main Result}

The main result in this document is the following;
\begin{theorem}\label{theo:mainresult}
    Let $G$ be a skeletal triangulation,
    then $G$ has a dominating set of size
    $\left\lfloor\frac{\Phi(G)}{3.5}\right\rfloor$ unless $G$ is one of the following:
    \begin{itemize}
        \item octahedron
        \item $3$-bifan (= octahedron minus one edge)
        \item special 4343434 heptagon
    \end{itemize}
    We call these the \emph{sporadic examples}. They are depicted in Figure~\ref{fig:sporadic}.
\end{theorem}

In Section~\ref{sect:mainproof}, we prove Theorem~\ref{theo:mainresult} via induction. This requires a very particular ordering on skeletal triangulations.

\begin{definition}[Smaller]
    Let $G_1, G_2$ be skeletal triangulations. We say $G_1$ is \emph{smaller than} $G_2$ if
    \begin{itemize}
        \item $G_1$ has fewer interior vertices than $G_2$, or the same number and
        \item $G_1$ has fewer bridges than $G_2$, or the same number and
        \item $G_1$ has smaller $\Phi$ than $G_2$, or the same number and
        \item $G_1$ has fewer blocks (i.e. 2-connected components) than $G_2$, or the same number and
        \item $G_1$ has fewer vertices than $G_2$, or the same number and
        \item $G_1$ has fewer degree-2 vertices than $G_2$.
    \end{itemize}
\end{definition}

\noindent
Formally, we show the following
\begin{proposition}[Induction Step] \label{prop:inductionstep}
    Let $G$ be a skeletal triangulation that is not one of the sporadic examples.
    Suppose every $G'$ that is smaller than $G$ satisfies the following:
    \begin{itemize}
        \item (Induction Hypothesis) If $G'$ is not one of the sporadic examples, then it has a dominating set of size
            $\left\lfloor \frac{\Phi(G')}{3.5} \right\rfloor$.
    \end{itemize}
    Then $G$ has a dominating set of size $\left\lfloor \frac{\Phi(G)}{3.5} \right\rfloor$.
\end{proposition}

\begin{remark}[Pitfalls]
    The conditions in the induction hypothesis might look innocuous, but we have to be very careful when applying the induction hypothesis
    to some graph $G'$ we constructed. Here are some common pitfalls and how we might deal with them:
    \begin{itemize}
        \item If $G'$ is disconnected, then it is not a skeletal triangulation. Solution: Handle cut vertices,
            chords or shared interior neighbors in earlier cases. This allows for stronger connectivity assumptions in later cases.
        \item If $G'$ contains a leaf, then it is not a skeletal triangulation. Solution: When deleting things, pay
            special attention to vertices that loose two or more neighbors. A vertex of degree $\ge 3$ can only turn into
            a leaf if it looses at least two neighbors.
        \item $G'$ might be a sporadic example. Solution: The sporadic examples are all 3-connected. If $G'$ is the result
            of an attaching or fusing operation, then $G'$ is not 3-connected and hence not a sporadic example.
    \end{itemize}
\end{remark}

\subsection{Acts as}

Intuitively speaking, if a rooted skeletal triangulation (or rooted near triangulation)
has a minimum rooted dominating set that contains the root (or base), this makes it ``easier''
to find small dominating sets in the graph obtained by fusing (or attaching). Our goal is to establish
a precise relation between this ``easier'' and the minimum possible penalty $\phi$.

\begin{definition}[Acts as]
    Let $(G, u)$ be a rooted skeletal triangulation. We say $(G, u)$ \emph{acts as}
    \begin{itemize}
        \item[AB] if $G$ has a minimum rooted dominating set that contains $u$,
        \item[LR] if $G$ has a minimum rooted dominating set that dominates $u$ (and $G$ does not act as A+B), and
        \item[Nope] otherwise.
    \end{itemize}
\end{definition}

\noindent
The sporadic examples, rooted at any boundary vertex, all act as Nope.
If we know what $G$ acts as, we can make Theorem~\ref{theo:mainresult} more specific.

\begin{theorem}\label{theo:skeletalact}
    Let $(G, u)$ be a rooted skeletal triangulation.
    Let $s = s(G, u)$ and let $\phi = \phi(G, u)$. If $(G, u)$ acts as
    \begin{itemize}
        \item [AB] then $\phi \geq 3.5 s - 1$.
        \item [LR] then $\phi \geq 3.5s$.
        \item [Nope] then $\phi \geq 3.5s + 1.5$.
    \end{itemize}
\end{theorem}
\begin{proof}
    Follows from Theorem~\ref{theo:mainresult}, see Section~\ref{sect:actasrefine}.
\end{proof}

\begin{figure}
    \begin{center} \begin{tikzpicture} \begin{scope}
        \begin{scope}
            \node (A) at (0, 0) [label=below:$u$] {};
            \node (B) at (-0.7, 1) {};
            \node (C) at (0.7, 1) {};
            \draw (A) -- (B) -- (C) -- (A);
        \end{scope}
        \begin{scope}[shift={(3.3, 0)}]
            \node (A) at (0, 0) [label=below:$u$] {};
            \node (B) at (-0.7, 1) {};
            \node (C) at (0.7, 1) {};
            \node (D) at (0, 2) {};
            \draw (A) -- (B) -- (C) -- (A);
            \draw (B) -- (D) -- (C);
        \end{scope}
        \begin{scope}[shift={(7, 0)},scale=0.7 ]
            \node (A) at (0, 0) [label=below:$u$] {};
            \node (B) at (0, 1) {};
            \node (C) at (-0.7, 2) {};
            \node (D) at (0.7, 2) {};
            \node (E) at (0, 3) {};
            \draw (A) -- (B) -- (C) -- (D) -- (E) -- (C);
            \draw (B) -- (D);
        \end{scope}

    \end{scope} \end{tikzpicture} \end{center}
    \caption{From left to right: small AB, small LR, small Nope, each with root $u$.}
    \label{fig:smallfuse}
\end{figure}
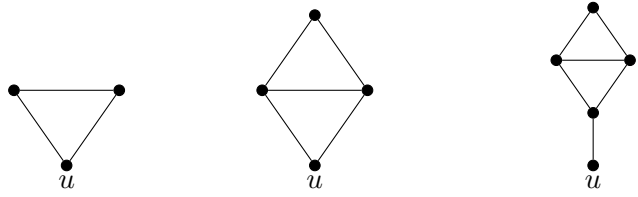

Figure~\ref{fig:smallfuse} depicts a small example for each case. The bounds in Theorem~\ref{theo:skeletalact} are tight in those examples.
Next, we consider a similar notion and bound for rooted near-triangulations.

\begin{definition}[Acts as]
    Let $(G, u, v)$ be a rooted near-triangulation. Let $s = s(G, u, v)$.
    In the following, each case excludes the preceding ones.
    We say $(G, u, v)$ \emph{acts as}
    \begin{itemize}
        \item [A+B] if $G$ has a rooted dominating set of size $s$ that contains both $u$ and $v$.
        \item [OR] if $G$ has a two rooted dominating sets of size $s$ with one containing $u$ and one containing $v$.
        \item [A] if $G$ has a rooted dominating set of size $s$ that contains $u$.
        \item [B] if $G$ has a rooted dominating set of size $s$ that contains $v$.
        \item [AND] if $G$ has a dominating set of size $s$ and a rooted dominating set of size $s+1$ that contains both $u$ and $v$.
        \item [L+R] if $G$ has a dominating set of size $s$.
        \item [OCTA] if $G$ has two rooted dominating sets of size $s$ with one dominating $u$ and one dominating $v$, plus a rooted dominating set of size $s+1$ that contains both $u$ and $v$.
        \item [L OR R] if $G$ has two rooted dominating sets of size $s$ with one dominating $u$ and one dominating $v$.
        \item [L] if $G$ has a rooted dominating set of size $s$ that dominates $u$.
        \item [R] if $G$ has a rooted dominating set of size $s$ that dominates $v$.
        \item [None] if otherwise.
    \end{itemize}
\end{definition}

\begin{theorem} \label{theo:nearact}
    Let $(G, u, v)$ be a rooted near-triangulation.
    Let $s = s(G, u, v)$ and let $\phi = \phi(G, u, v)$. If $(G, u, v)$ acts as
    \begin{itemize}[itemindent=80pt]
        \item [A+B, OR] then $\phi \geq 3.5s - 2$.
        \item [A, B] then $\phi \geq 3.5s - 1$.
        \item [AND, L+R, OCTA, L OR R] then $\phi \geq 3.5s$.
        \item [L, R] then $\phi \geq 3.5s + 0.5$.
        \item [None] then $\phi \geq 3.5s + 1.5$.
    \end{itemize}
\end{theorem}
\begin{proof}
    Follows from Theorem~\ref{theo:mainresult}, see Section~\ref{sect:actasrefine}.
\end{proof}

\begin{figure}
    \begin{center} \begin{tikzpicture} \begin{scope}
        \begin{scope} 
            \node(A) at (0, 0) {};
            \node(B) at (1, 0) {};
            \node(C) at (-1, 1) {};
            \node(D) at (0, 1) {};
            \node(E) at (1, 1) {};
            \node(F) at (2, 1) {};
            \draw (A) -- (C) -- (D) -- (E) -- (F) -- (B) -- (E) -- (A) -- (D);
            \draw (A) -- (C);
            \draw[based] (A) -- (B);
        \end{scope} 
        \begin{scope}[shift={(3, 0)}]
            \node (A) at (0, 0) {};
            \node (B) at (1, 0) {};
            \node (C) at (0.5, 1) {};
            \draw (A) -- (C) -- (B);
            \draw[based] (A) -- (B);
        \end{scope} 
        \begin{scope}[shift={((5, 0)}]
            \node (A) [red] at (0, 0) {};
            \node (B) at (1, 0) {};
            \node (C) at (0, 1) {};
            \node (D) at (1, 1) {};
            \draw (A) -- (C) -- (D) -- (B);
            \draw (A) -- (D);
            \draw [based] (A) -- (B);
        \end{scope} 
        \begin{scope}[shift={(7, 0)}]
            \node (A) at (0, 0) {};
            \node (B) [red] at (1, 0) {};
            \node (C) at (0, 1) {};
            \node (D) at (1, 1) {};
            \draw (A) -- (C) -- (D) -- (B);
            \draw (B) -- (C);
            \draw [based] (A) -- (B);
        \end{scope}
    \end{scope} \end{tikzpicture}\end{center}
    
    \begin{center} \begin{tikzpicture} \begin{scope}
        \begin{scope} 
            \node (A) at (0, 0) {};
            \node (B) at (1, 0) {};
            \node (C) at (-0.5, 1) {};
            \node (D) at (0.5, 1) {};
            \node (E) at (1.5, 1) {};
            \draw (A) -- (C) -- (D) -- (E) -- (B);
            \draw (A) -- (D) -- (B);
            \draw[based] (A) -- (B);
        \end{scope}
        \begin{scope}[shift={(2.5, 0)}] 
            \node (A) at (0, 0) {};
            \node (B) at (1, 0) {};
            \node (C) at (0, 1) {};
            \node (D) at (1, 1) {};
            \node (E) at (0.5, 1.7) {};
            \draw (A) -- (C) -- (E) -- (D) -- (B);
            \draw (A) -- (D) -- (C);
            \draw[based] (A) -- (B);
        \end{scope}
        \begin{scope}[shift={(5.5,0.48)},scale=0.5]
            \node (A) at (210:2) {};
            \node (B) at (-30:2) {};
            \node (C) at  (90:2) {};
            \node (X) at ( 30:0.5) {};
            \node (Y) at (150:0.5) {};
            \node (Z) at (-90:0.5) {};
            \draw (A)  (B) -- (C) -- (A);
            \draw (A) -- (Z) -- (B) -- (X) -- (C) -- (Y) -- (A);
            \draw (X) -- (Y) -- (Z) -- (X);
            \draw[based] (A) -- (B);
        \end{scope}
        \begin{scope}[shift={(7.5,0)}] 
            \node(A) at (0, 0) {};
            \node (B) at (1, 0) {};
            \node (C) at (0.5, 0.5) {};
            \node (D) at (0, 1) {};
            \node (E) at (1, 1) {};
            \node (F) at (0.5, 1.7) {};
            \draw (A) -- (D) -- (F) -- (E) -- (B) -- (C) -- (A);
            \draw (D) -- (C) -- (E) -- (D);
            \draw[based] (A) -- (B);
        \end{scope}
    \end{scope} \end{tikzpicture}\end{center}
    
    \begin{center} \begin{tikzpicture} \begin{scope}
        \newcommand{\drawL}[1]{
            \begin{scope}[#1] 
                \node (A) at (0, 0) {};
                \node (B) at (1, 0) {};
                \node (C) at (0, 1) {};
                \node (D) at (1, 1) {};
                \node (E) at (0, 2) {};
                \node (F) at (1, 2) {};
                \draw (A) -- (C) -- (E) -- (F) -- (D) -- (B);
                \draw (A) -- (D) -- (C) -- (F);
                \draw[based] (A) -- (B);
            \end{scope}
        }
        \drawL{} 
        \drawL{shift={(3, 0)}, xscale=-1} 
        \begin{scope}[shift={(5, 0)}] 
            \node (A) at (-0.5, 0) {};
            \node (B) at (1.5, 0) {};
            \node (C) at (-0.1, 0.8) {};
            \node (D) at (0.5, 0.5) {};
            \node (E) at (1.1, 0.8) {};
            \node (F) at (0.5, 1.1) {};
            \node (G) at (0.5, 2) {};
            \draw (A) -- (C) -- (D) -- (E) -- (B);
            \draw (A) -- (D) -- (B);
            \draw (C) -- (F) -- (E) -- (G) -- (C);
            \draw (G) -- (F) -- (D);
            \draw[based] (A) -- (B);
        \end{scope}
    \end{scope} \end{tikzpicture} \end{center}
    \caption{Small examples, with base edge in blue and bold. Each row to be read from left to right.
    Top row: A+B, OR, A, B, the later two with their red vertex. Middle row: AND, L+R, OCTA, L OR R. Bottom row: L, R, None.
    It is worth noting that the only non-outerplanar examples here
    are OCTA, L OR R and None.}
    \label{fig:smallattach}
\end{figure}
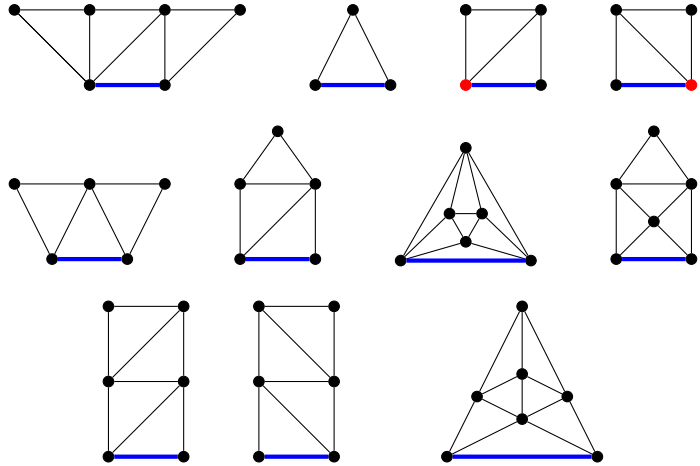

Figure~\ref{fig:smallattach} depicts a small example for each case. The most important ones are small A, small B and small OR.
Take special note of the \emph{red vertex} in the small A, B.
Note in the OCTA, L OR R, L, R cases, the bound in proposition~\ref{prop:nearact} is not tight.
This is illustrated in Table~\ref{tab:actbound}. The loose bounds are sufficient for our proof.

\begin{table}\centering
    \ra{1.2}
    \begin{tabular}{@{}lllll@{}}\toprule
        Act as & $s$ & $\phi$ & formula & lower bound\\ \midrule
        A+B & $2$ & $5$ & $\phi = 3.5s - 2$ & $\phi \geq 3.5s - 2$ \\
        OR & $1$ & $1.5$ & $\phi = 3.5s - 2$ & $\phi \geq 3.5s - 2$ \\
        A, B & $1$ & $2.5$ & $\phi = 3.5s - 1$ & $\phi \geq 3.5s - 1$ \\
        AND & $1$ & $3.5$ & $\phi = 3.5s$ & $\phi \geq 3.5s$ \\
        L+R & $1$ & $3.5$ & $\phi = 3.5s$ & $\phi \geq 3.5s$ \\
        OCTA & $1$ & $4$ & $\phi = 3.5s+0.5$ & $\phi \geq 3.5s$ \\
        L OR R & $1$ & $4.5$ & $\phi = 3.5s+1$ & $\phi \geq 3.5s$ \\
        L, R & $1$ & $4.5$ & $\phi = 3.5s+1$ & $\phi \geq 3.5s + 0.5$ \\
        None & $1$ & $5$ & $\phi = 3.5s+1.5$ & $\phi \geq 3.5s + 1.5$ \\
        \bottomrule
    \end{tabular}
    \caption{The relation between $\phi$ and the size $s$ of a minimum rooted dominating set.
    For each act-as type, we compare the small example in Figure~\ref{fig:smallattach}
    to the lower bound in Proposition~\ref{prop:nearact}.}
    \label{tab:actbound}
\end{table}

\subsection{Toolbox: Replacing attachments}

\begin{lemma}\label{lemm:replaceattach}
    Let $G_0$ be a near triangulation with boundary edge $u_0, v_0$.
    Let $(G_1, u_1, v_1)$ and $(G_2, u_2, v_2)$ be rooted near-triangulations that act as the same type.
    Let $H_i$ be the result of attaching $G_i$ to $u_0, v_0$.
    Then
    \begin{align*}  
        s(H_2) - s(H_1) &= s(G_2, u_2, v_2) - s(G_1, u_1, v_1)\\
        \Phi(H_2) - \Phi(H_1) &= \phi(G_2, u_2, v_2) - \phi(G_1, u_1, v_1)
    \end{align*}
\end{lemma}
\begin{proof}
    For the statement on $s$, note that a $H_i$-dominating set is the disjoint union of a set in $G_0$
    and a rooted dominating set in $G_i$. Since $G_1$ and $G_2$ act as the same type,
    we can switch between rooted dominating sets of the two that both contain $u$ and/or $v$ or
    both do not contain $u$ and/or $v$.
    The statement on $\Phi$ follows from the equality cases in Lemmas~\ref{lemm:unroot} and~\ref{lemm:defuse} (unrooting and detaching).
\end{proof}

\subsection{Deleting one problem creates at most one new one}
Sometimes, we want to delete a low degree vertex that is already dominated for one reason or another.
It is crucial that this does not increase $\Phi$ by too much.

\begin{lemma}[Problems are not adjacent]
    Let $G$ be a skeletal triangulation. Suppose $u, v$ are each involved in distinct low-degree problems that are not degree-2 cut vertices,
    i.e. each a degree-3 vertex in a 3-pair of a distinct bad 5-wheel or each a degree-2 vertex in a distinct isolated triangle / pivoting triangle / ear.
    Then $u$ is not adjacent to $v$.
\end{lemma}
\begin{proof}
    A tedious but straight-forward case analysis.
\end{proof}

\begin{corollary}[Degree bound on problems] \label{coro:fewproblems}
    Let $G$ be a near-triangulation with boundary vertex $u$. Suppose $G - u$ is a skeletal triangulation.
    Then $\Phi(G - u) \le \Phi(G) - 1 + \lfloor \frac{\deg(u)+1}{2}\rfloor$.
    In other words, deleting $u$ creates at most $\lfloor \frac{\deg(u)+1}{2} \rfloor$ new low-degree problems.
\end{corollary}
\begin{proof}
    Every newly created low-degree problem contains at least one vertex adjacent to $u$. Pick one such vertex
    for each problem, then by the previous lemma,
    these vertices are not adjacent to each other, so they form an independent set in $N(u)$. $N(u)$ is a path on $\deg(u)$ vertices.
\end{proof}

\begin{lemma}[Deleting Problems]\label{lemm:deleteprob}
    Let $G$ be a skeletal triangulation with boundary vertex $u$.
    Suppose that $\Phi(G) = \phi(G, u) + 1.5$ and that $u$ is not a cut vertex.
    Then $u \in G$ is an ear tip or in every 3-pair of a bad 5-wheel
    and $H := G - u$ is a skeletal triangulation with
    \[
        \Phi(H) \le \phi(G, u) + 0.5 = \Phi(G) - 1.
    \]
    Let $G$ be a near-triangulation with boundary edge $u, v$. Suppose that $\Phi(G) = \phi(G, u, v) + 2.5$.
    Then at least one of $u, v$ is an ear tip or in the (unique) 3-pair of a bad 5-wheel, suppose it is $u$.
    Then $H := G - u$ is a near-triangulation with
    \[
        \Phi(H) \le \phi(G, u, v) + 1 + 0.5 = \Phi(G) - 1.
    \]
\end{lemma}
\begin{proof}
    Lemma~\ref{lemm:unroot} shows that $u$ an ear tip or in a 3-pair.
    If $\deg(u) = 2$, then Corollary~\ref{coro:fewproblems} gives the result.
    If $u$ is in a 3-pair, then deleting $u$ creates exactly one ear and no new bad 5-wheels.
\end{proof}

\subsection{Toolbox: Covering by Fusing a small LR}

In the main proof, we sometimes modify the graph and would like to ``remember'' that some vertex $u$ is already
dominated in the original graph, meaning it does not have to be dominated again.
Fusing a small LR to $u$ achieves exactly this.

\begin{lemma}
    Let $G$ be a skeletal triangulation with boundary vertex $u$. Let $H$ be the graph resulting
    from fusing a small LR (Figure~\ref{fig:smallfuse}) to $u$. Then:
    \begin{itemize}
        \item $\Phi(H) \leq \Phi(G) + 3.5$, with equality if $\deg_G(u) \geq 4$.
        \item If $H$ has a dominating set of size $s$, then $G$ has a set of size $\leq s-1$ that
            dominates every vertex except maybe $u$.
    \end{itemize}
\end{lemma}
\begin{proof}
    Let $H' \subseteq H$ be the small LR that was fused, including $u$. Then,
    \[
        \Phi(H) \leq \Phi(G) + \phi(H', u) = \Phi(G) + 3.5.
    \]
    If $\deg(u) \geq 4$, then $u$ is not involved in any low-degree problems, hence equality holds.
    For the second statement, let $D$ be a $H$-dominating set
    of size $s$. Put $D' = D \setminus (H' - u)$,
    then clearly $D'$ dominates $G - H'$ as $u$ is a cut vertex in $H$.
    Moreover $|D'| < |D| = s$ as $D$ contains at least one vertex of $H' - u$.
\end{proof}

This lemma generalizes to many vertices $u_i$:
\begin{lemma}[Covering with LRs]
    Let $G$ be a skeletal triangulation with boundary vertices $u_1, \dots, u_k$. Let $H$ be the graph resulting
    from fusing a small LR to each $u_i$. Then:
    \begin{itemize}
        \item $\Phi(H) \leq \Phi(G) + 3.5 k$.
        \item If $H$ has a dominating set of size $s$, then $G$ has a set of size $s-k$ that
            dominates every vertex except maybe some of the $u_i$.
    \end{itemize}
\end{lemma}
\begin{proof}
    Similar to the previous proof.
\end{proof}

\subsection{Toolbox: Neat dominating sets}
In some graphs, there are vertices that
appear ``weakly suboptimal'' to include in a dominating set.

\begin{definition}[Neat]
    Let $G$ be a skeletal triangulation. A dominating set $S \subseteq G$ is \emph{neat}
    if for every $u \in S$, there is no $v \in V(G)$ with $N[u] \subsetneq N[v]$.
\end{definition}

\begin{lemma}
    Every skeletal triangulation has a minimum dominating set that is neat.
\end{lemma}
\begin{proof}
    Take a minimum dominating set $S$ that maximizes
    \[
        \sum_{u \in S} \big| N[u] \big|,
    \]
    then $S$ is neat.
\end{proof}

In many cases, neat dominating sets allow us to
assume that some vertices are not contained or have to be contained
in a dominating set. Let $S \subseteq G$ be a near dominating set. Here are some examples,
many of which occur in Figure~\ref{fig:neatdom}.
\begin{itemize}
    \item If $u$ is the tip of an ear or in a 3-pair of a bad 5-wheel, then $u \notin S$.
    \item If $\deg(u) = 3$ and $G$ is a 3-connected near-triangulation, then $u \notin S$.
    \item If $G$ is the result of fusing a small AB to $u$, then $u \in S$.
    \item If $G$ is the result of attaching a small A (or B) to $u, v$, then the red vertex (see Figure~\ref{fig:smallattach}) of the A (or B) is in $S$.
\end{itemize}

\begin{figure}
    \begin{center} \begin{tikzpicture}[scale=0.8]
        \node (A) at (0, 0) {};
        \node (B) at (2, 0) {};
        \node (E) at (1, 1) {};
        \node (C) at (2, 2) [red] {};
        \node (D) at (0, 2) {};
        \node (X) at (1, 4) [blue] {};
        \node (Y) at (3, 4) [blue] {};
        \draw (A) -- (B) -- (C) -- (D) -- (A);
        \draw (A) -- (E) -- (C) (B) -- (E) -- (D);
        \draw (D) -- (X) -- (C) -- (Y) -- (X);
    \end{tikzpicture} \hspace{2cm} \begin{tikzpicture}[scale=0.8]
        \node (A) at (0, 0) [blue] {};
        \node (B) at (2, 0) [blue] {};
        \node (E) at (1, 1) [orange] {};
        \node (C) at (2, 2) {};
        \node (D) at (0, 2) {};
        \node (X) at (1, 4) {};
        \node (Y) at (3, 4) {};
        \draw (A) -- (B) -- (C) -- (D) -- (A);
        \draw (A) -- (E) -- (C) (B) -- (E) -- (D);
        \draw (D) -- (X) -- (C) -- (Y) -- (X);
    \end{tikzpicture} \end{center}
    \caption{On the left: In the attached B, the red vertex is contained in every neat dominating set, whereas the blue vertices are not.
    On the right: The blue vertices, which form a 3-pair for the bad 5-wheel centered at the orange vertex,
    are not contained in any neat dominating set, as the orange vertex has strictly larger closed neighborhood.
    Note however that the orange vertex is not contained in every neat dominating set.}
    \label{fig:neatdom}
\end{figure}
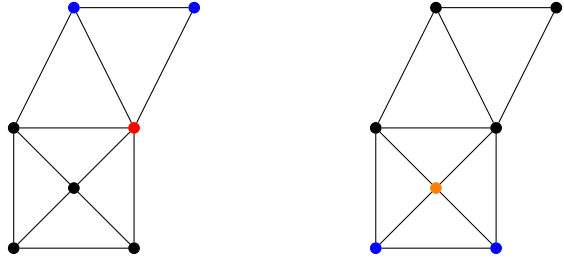

\subsection{Toolbox: Forcing a Vertex}

We have just seen that fusing a small AB allows us to force a vertex into every neat dominating set.
This turns out to be extremely useful.

\begin{lemma}[Forcing a vertex]
    Let $G$ be a skeletal triangulation with boundary vertex $u$. Let $H$ be the graph resulting
    from fusing a small AB to $u$. Then:
    \begin{itemize}
        \item $\Phi(H) \leq \Phi(G) + 2.5$.
        \item Any neat dominating set in $H$ is a dominating set in $G$ that contains $u$.
    \end{itemize}
\end{lemma}
\begin{proof}
    Let $H' \subseteq H$ be the small AB that was fused, including $u$. Then,
    \[
        \Phi(H) \leq \Phi(G) + \phi(H', u, v) = \Phi(G) + 2.5.
    \]
    Any dominating set in $H$ contains at least one vertex of the AB (including $u$),
    so a neat one has to contain $u$.
\end{proof}

Due to the following reason, forcing a vertex is a lot more versatile than simply ``picking'' a vertex $u$ and deleting
$u$ together with some of its neighbors: If we delete $u$, this by necessity turns all interior neighbors of $u$ into boundary vertices,
which might create many low-degree problems. If we instead force $u$ and are somewhat picky with the other neighbors we delete,
we can keep most interior neighbors in the interior, which avoids having to discuss them in detail.
Intuitively, forcing is ``efficient'' in the sense that $\Phi$ increases by $2.5$ and keeping $u$ around instead of deleting it is another $+1$,
so we get an ``increase'' of $3.5$ while potentially having $1$ extra vertex in a minimum dominating set.

\subsection{Proof of Theorem~\ref{theo:skeletalact}}

In this section, we show that Theorem~\ref{theo:mainresult} implies Theorem~\ref{theo:skeletalact},
in a way that can be used inside the induction step of the main proof.

\begin{proposition} \label{prop:skeletalact}
    Let $(G, u)$ be a rooted skeletal triangulation. Let $s = s(G, u)$ and let $\phi = \phi(G, u)$.
    If $(G, u)$ acts as
    \begin{itemize}
        \item[AB] then $\phi \ge 3.5s - 1$ under the following assumption: 
            Let $H$ be the graph resulting from
            fusing a small LR to $u$. Assume Theorem~\ref{theo:mainresult} holds for $H$.
        \item[LR] then $\phi \ge 3.5s$ under the following assumption: 
            Let $\{u, v\}$ be a boundary edge incident to $u$.
            Let $H$ be the graph resulting from attaching an A with red vertex $u$ to $u, v$. Assume
            Theorem~\ref{theo:mainresult} holds for $H$.
        \item[Nope] then $\phi \ge 3.5s + 1.5$ under the following assumptions: If $G$ is a sporadic example, assume nothing.
            If $\deg(u) \ge 2$, assume Theorem~\ref{theo:mainresult} holds for $G$. If $\deg(u) = 1$,
            assume $H := G - u$ satisfies the attaching-an-A assumption of the LR case.
    \end{itemize}
\end{proposition}
\begin{proof}
    If $G$ is a sporadic example, then $G$ acts as Nope and we check $\phi \ge 3.5s + 1.5$ by hand.
    Otherwise, suppose $G$ acts as
    \begin{itemize}
        \item[AB] Fuse a small LR to $u$. By assumption, the resulting graph $H$ satisfies Theorem~\ref{theo:mainresult},
            hence $\Phi(H) \geq 3.5 s(H)$. As $G$ acts as AB and the small LR acts as LR,
            \[
                s(H) = s(G, u) + s(\mathrm{small\ LR}, u) = s(G, u) + 1.
            \]
            The small LR has $\phi = 3.5$, hence by Lemma~\ref{lemm:defuse}
            \[
                \Phi(H) = \phi(G, u) + 3.5 + 1 = \phi(G, u) + 4.5.
            \]
            Chaining inequalities yields
            \[
                \phi(G, u) = \Phi(H) - 4.5 \ge 3.5s(H) - 4.5 = 3.5 s(G, u) - 1.
            \]
        \item[LR] Let $\{u, v\}$ be a boundary edge incident to $u$. Attach a small A with red vertex $u$ to $u, v$.
            By assumption, the resulting graph $H$ satisfies Theorem~\ref{theo:mainresult},
            hence $\Phi(H) \geq 3.5 s(H)$. As $G$ acts as LR and the small A acts as A,
            \[
                s(H) = s(G, u) + s(\mathrm{small\ A}, u) = s(G) + 1.
            \]
            The small A has $\phi = 2.5$, hence by a slight modification of Lemma~\ref{lemm:defuse}
            \[
                \Phi(H) \le \phi(G, u) + \phi(\mathrm{small\ A}, u, v) + 1 = \phi(G, u) + 3.5.
            \]
            Chaining inequalities yields
            \[
                \phi(G, u) \ge \phi(H) - 3.5 \ge 3.5s(H) - 3.5 = 3.5 s(G, u).
            \]
        \item[Nope] We distinguish whether $u \in G$ is a degree-1 root.
            \begin{itemize}
                \item If $\deg(u) \geq 2$, then $G$ is a skeletal triangulation. By Lemma~\ref{lemm:unroot}, $\Phi(G) \leq \phi(G, u) + 1.5$. As $(G, u)$ acts as Nope,
                    $s(G) = s(G, u) + 1$.  By assumption, $G$ satisfies Theorem~\ref{theo:mainresult}, hence $\Phi(G) \geq 3.5 s(G)$. Chaining inequalities
                    yields
                    \[
                        \phi(G, u) \ge \Phi(G) - 1.5 \ge 3.5 s(G) - 1.5 = 3.5 s(G, u) + 2.
                    \]
                \item If $\deg(u) = 1$, then let $w$ be the unique neighbor of $u$ and let $H = G - w$. $H$ is a rooted skeletal triangulation with root $w$
                    and $\phi(H, w) = \phi(G, u) - 1.5$.
                    As $G$ acts as Nope, $H$ acts as LR and $s(G, u) = s(H, w)$. By assumption, $H$ satisfies the assumptions needed for
                    the LR case of this proposition. Therefore, by the LR case, $\phi(H, w) \geq 3.5 s(H, w)$. Chaining inequalities
                    yields
                    \[
                        \phi(G, u) = \phi(H, w) + 1.5 \ge 3.5 s(H, w) + 1.5 = 3.5 s(G, u) + 1.5.
                    \]
            \end{itemize}
    \end{itemize}
\end{proof}

\subsection{Proof of Theorem~\ref{theo:nearact}}

In this section, we show that Theorem~\ref{theo:mainresult} implies Theorem~\ref{theo:nearact},
in a way that can be used inside the induction step of the main proof.

\begin{proposition} \label{prop:nearact}
    Let $(G, u, v)$ be a rooted near-triangulation.
    Let $s = s(G, u, v)$ and let $\phi = \phi(G, u, v)$. 
    Assume Theorem~\ref{theo:mainresult} holds for any near-triangulation with the same number of interior vertices as $G$.
    Then: If $(G, u, v)$ acts as
    \begin{itemize}[itemindent=80pt]
        \item [A+B, OR] then $\phi \geq 3.5s - 2$.
        \item [A, B] then $\phi \geq 3.5s - 1$.
        \item [AND, L+R, OCTA, L OR R] then $\phi \geq 3.5s$.
        \item [L, R] then $\phi \geq 3.5s + 0.5$.
        \item [None] then $\phi \geq 3.5s + 1.5$.
    \end{itemize}
\end{proposition}
\begin{proof}
    If $G$ is a sporadic example, then $G$ acts as OCTA and $\phi \ge 3.5 s + 0.5$. Otherwise, suppose $G$ acts as
    \begin{description}
        \item [A+B, OR] Let $H$ be the graph resulting from attaching a small L+R to $u, v$. Then 
            $s(H) = s(G, u, v) + 1$, and $\Phi(H) = \phi(G, u, v) + 2 + 3.5$ by Lemma~\ref{lemm:defuse}.
            By assumption, Theorem~\ref{theo:mainresult} holds for $H$, therefore $\Phi(H) \geq 3.5 s(H)$.
            Chaining inequalities yields $\phi(G, u, v) \geq 3.5 s(G, u, v) - 2$.
        \item [A, B] Let $H$ be the graph resulting from attaching a small B / A to $u, v$, in a way that forces the other base vertex compared to $G$.
            Then $s(H) = s(G, u, v) + 1$ and $\Phi(H) = \phi(G, u, v) + 2 + 2.5$.
            By assumption, Theorem~\ref{theo:mainresult} holds for $H$, therefore $\Phi(H) \geq 3.5 s(H)$.
            Chaining inequalities yields $\phi(G, u, v) \geq 3.5 s(G, u, v) - 1$.
        \item [AND, L+R, L OR R, OCTA] Let $H$ be the graph resulting from attaching a small $OR$ to $u, v$. Then
            $s(H) = s(G, u, v) + 1$ as $G$ acts as AND, L+R, L OR R, OCTA, and $\Phi(H) = \phi(G, u, v) + 2 + 1.5$.
            By assumption, Theorem~\ref{theo:mainresult} holds for $H$, therefore $\Phi(H) \ge 3.5 s(H)$.
            Chaining inequalities yields $\phi(G, u, v) \geq 3.5 s(G, u, v)$
        \item [L, R] Let $H$ be the graph resulting from attaching a small L / R to $u, v$ in a way that dominates the same base vertex compared to
            $G$. Then $\Phi(H) = \phi(G, u, v) + 2 + 4.5$ and $s(H) = s(G, u, v) + 2$ (one +1 from the small L / R and one +1
            from a $u$ or $v$ that is not dominated by either side.) Chaining inequalities yields $\phi(G, u, v) \geq 3.5 s(G, u, v) + 0.5$.
        \item [None] Consider the two cases in the attachment part of Lemma~\ref{lemm:unroot}. Suppose first that $\Phi(G) = \phi(G, u, v) + 2$. As $G$ acts as None,
            $s(G) = s(G, u, v) + 1$. By assumption, Theorem~\ref{theo:mainresult} holds for $G$, therefore $\Phi(G) \ge 3.5 s(G)$.
            Chaining inequalities yields
            \[
                \phi(G, u, v) \ge \Phi(G) - 2 \ge 3.5 s(G) - 2 \ge 3.5 s(G, u, v) + 1.5.
            \]
            In all remaining cases, $\Phi(G) = \phi(G, u, v) + 2.5$.
            Suppose that $u$ (or $v$) is not involved in a low-degree problem. Then $\phi(G, v) = \phi(G, u, v) + 1$. As $(G, u, v)$ acts
            as None, $s(G, v) = s(G, u, v) + 1$. By Proposition~\ref{prop:skeletalact},
            $\phi(G, v) \ge 3.5 s(G, v)  - 1$, independent of what $(G, v)$ acts as. Chaining inequalities yields
            \[
                \phi(G, u, v) = \phi(G, v) - 1 \ge 3.5 s(G, v) - 2 = 3.5s(G) + 1.5.
            \]
            In the remaining case, $u, v$ is a 3-pair in a bad 5-wheel. Let $w$ be the interior vertex adjacent to $u, v$.
            Then $H = G / \{u, v\}$ is a skeletal triangulation. Let $x \in H$ be the vertex
            corresponding to $\{u, v\}$. Then $(H, x)$ acts as Nope with $s(H, x) = s(G, u, v)$ and $\phi(H, x) = \phi(G, u, v)$.
            By Proposition~\ref{prop:skeletalact}, $\phi(H, x) \ge 3.5 s(H, x) + 1.5$. Chaining inequalities yields
            \[
                \phi(G, u, v)  = \phi(H, x) \ge 3.5 s(H, x) + 1.5 = 3.5 s(G, u, v) + 1.5.
            \]
    \end{description}
\end{proof}

\section{The Proof} \label{sect:mainproof}

In this section, we prove Proposition~\ref{prop:inductionstep}. The proof consists of many cases. In each case, we assume
none of the previously discussed cases apply. In particular, for later cases, we get to make stronger and stronger assumptions on the given graph.

\subsection{Bridge} \label{sect:bridge}
Suppose $G$ has a bridge $u, v$. Let $G_1, G_2$ be the two components resulting from deleting $\{u, v\}$.

\begin{claim}
    Then
    \[
        \Phi(G) = \phi(G_1, u) + \phi(G_2, v) + 2. 
    \]
\end{claim}
\begin{proof}
    If $u$ (or $v$) is a degree-1 root in $(G_1, u)$ (or $(G_2, v)$), then it is a degree-2 cut vertex in $G$ and vice versa.
    Otherwise, $u$ and $v$ are not involved in any low-degree problems, as they are cut vertices in $G$ and roots in $G_1, G_2$.
    All other low-degree problems are the same in $G$ and $G_1, G_2$,
    but $u, v$ are not counted in the later (as they are roots), so we get a $+2$.
\end{proof}

\begin{claim}\label{clai:bridgered}
    Theorem~\ref{theo:skeletalact} holds for $(G_1, u)$ and $(G_2, v)$. 
\end{claim}
\begin{proof}
    We check that Proposition~\ref{prop:skeletalact} applies: 
    The graphs resulting from attaching an A / B, or fusing a small LR to $G_1$ or to $G_2$ all have
    fewer bridges than $G$\footnote{And the same number of interior vertices and at most the same $\Phi$.}, hence these satisfy Theorem~\ref{theo:mainresult} by the induction hypothesis.
\end{proof}

With the claim, we conclude as follows: If one of $G_1, G_2$ acts as AB, then
$s(G) = s(G_1, u_1) + s(G_2, u_2)$ and
\begin{align*}
    \Phi(G) &= \phi(G_1, u_1) +\phi(G_2, u_2) + 2\\
    &\geq 3.5 s(G_1, u_1) - 1 + 3.5 s(G_2, u_2) - 1 + 2 = 3.5 s(G).
\end{align*}
Otherwise, if one of $G_1, G_2$ acts as Nope, say $G_1$, then $s(G) = s(G_1, u_1) + s(G_2, u_2) + 1$ and
\begin{align*}
    \Phi(G) &= \phi(G_1, u_1) + \phi(G_2, u_2) + 2\\
    &\geq 3.5 s(G_1, u_1) + 1.5 + 3.5 s(G_2, u_2)  + 2 = 3.5 s(G).
\end{align*}
Otherwise, both $G_1$ and $G_2$ act as LR, then $s(G) = s(G_1, u_1) + s(G_2, u_2)$ and
\begin{align*}
    \Phi(G) &= \phi(G_1, u_1) + \Phi(G_2, u_1) + 2\\
    &\geq 3.5 s(G_1, u_1) + 3.5 s(G_2, u_2) + 2 = 3.5 s(G) + 2.
\end{align*}

\paragraph{Conclusion} From now on, we assume that $G$ does not contain any bridges.

\subsection{Cut vertex} \label{sect:cutvertex}

Suppose $G$ has a cut vertex $u$. Let us split $G$ at $u$ into two pieces in the obvious way: Let $G_1, G_2$ each be subgraphs induced by $u$ together with one or more components in $G - u$
such that each component occurs in exactly one of $G_1$ and $G_2$. Then, $(G_1, u)$ and $(G_2, u)$ are rooted skeletal triangulations
and the result of fusing them is $G$. As $G$ has no bridges, neither $(G_1, u)$ nor $(G_2, u)$ has a degree-1 root.
Therefore
\[
    \Phi(G) = \phi(G_1, u) + \phi(G_2, u) + 1
\]
by Lemma~\ref{lemm:defuse}.
Similar to the bridge case, we would like to conclude via Theorem~\ref{theo:skeletalact},
but some care has to be taken to avoid circular arguments.

\subsubsection{One side acts as Nope}

Suppose that $G_1$ (or $G_2$) acts as Nope.
The graphs in the Nope case of Proposition~\ref{prop:skeletalact} have fewer blocks that $G$,
hence $G_1$ satisfies Theorem~\ref{theo:mainresult},
therefore $\phi(G_1, u) \ge 3.5 s(G_1, u) + 1.5$.
If $G_2$ is a sporadic example, then $\Phi(G_2) \ge 3.5 s(G_2) - 1$. Otherwise, 
by the induction hypothesis, $\Phi(G_2) \geq 3.5 s(G_2)$.
By Lemma~\ref{lemm:defuse}, $\Phi(G) \ge \phi(G_1, u) + \Phi(G_2) - 0.5$. By Lemma~\ref{lemm:fusedom}, $s(G) \le s(G_2) + s(G_1, u)$. Combining everything yields
\begin{align*}
    \Phi(G) &\ge \phi(G_1, u) + \Phi(G_2, u) - 0.5\\
    &\ge 3.5 s(G_1, u) + 1.5 + 3.5 s(G_2) - 1 - 0.5 \ge 3.5 s(G).
\end{align*}

\subsubsection{Both sides act as AB}
Suppose both $G_1$ and $G_2$ act as AB, then $s(G) = S(G_1, u) + S(G_2, u) - 1$.
As $G_1, G_2$ are both smaller than $G$ and not a sporadic example, by the induction hypothesis, $\Phi(G_i) \ge 3.5 s(G_i)$.
By Lemma~\ref{lemm:unroot}, $\Phi(G_i) \le \phi(G_i, u) + 1.5$. Chaining inequalities yields
\begin{align*}
    \Phi(G) &= \phi(G_1, u) + \phi(G_2, u) + 1 \ge \Phi(G_1) + \Phi(G_2) - 2\\
    &\ge 3.5 s(G_1) + 3.5 s(G_2) - 2 = 3.5 s(G) + 1.
\end{align*}

\subsubsection{One side is a small LR}
Suppose $G_1$ is a small LR, then $s(G) = s(G_2, u) + 1$.
By Lemma~\ref{lemm:defuse}, $\Phi(G) = \phi(G_1, u) + \phi(G_2, u) + 1$.
As $G_2$ does not act as Nope, $s(G_2) = s(G_2, u)$.
As $G_2$ is smaller than $G$ and not a sporadic example, by the induction hypothesis, $\Phi(G_2) \ge 3.5 s(G_2)$.
If $\Phi(G_2) = \phi(G_2, u) + 1$, then chaining inequalities yields
\begin{align*}
    \Phi(G) &= \phi(G_1, u) + \phi(G_2, u) + 1 = 3.5 + \Phi(G_2)\\
    &\ge 3.5 + 3.5 s(G_2) = 3.5 s(G).
\end{align*}
Otherwise, $\Phi(G_2) = \phi(G_2, u) + 1.5$, then by Lemma~\ref{lemm:deleteprob},
$u \in G_2$ is an ear tip or part of a 3-pair in a bad 5-wheel.
Let $H = G_2 - u$, then $s(H) = s(G_2)$ as a neat rooted dominating set in $G_2$ does not contain $u$
and $s(G_2) = s(G_2, u)$ as $G_2$ does not act as Nope.
By Lemma~\ref{lemm:deleteprob}, $\Phi(H) \le \phi(G_2, u) + 0.5$.
Chaining inequalities yields
\begin{align*}
    \Phi(G) &= \phi(G_1, u) + \phi(G_2, u) + 1 \ge 3.5 + \Phi(H) + 0.5\\
    &\ge 3.5 + 3.5 s(H) + 0.5 = 3.5 s(G) + 0.5.
\end{align*}

\subsubsection{One side acts as LR}

In all remaining cases, $G_1$ (or $G_2$) acts as LR, then $s(G) = s(G_1, u) + s(G_2, u)$.
By Lemma~\ref{lemm:defuse}, $\Phi(G) = \phi(G_1, u) + \phi(G_2) + 1$.
As $G_1$ is not a small LR, both $G_1$ and $G_2$ satisfy the assumptions of Proposition~\ref{prop:skeletalact},
as the involved graphs have fewer blocks and/or have smaller $\Phi$ than $G$.
$G_1$ acts as LR, so this yields $\phi(G_1, u) \ge 3.5 s(G_1, u)$.
$G_2$ acts as LR or AB, so this yields $\phi(G_2, u) \ge 3.5 s(G_2, u) - 1$.
Chaining inequalities yields
\[
    \Phi(G) = \phi(G_1, u) + \phi(G_2, u) + 1 \ge 3.5 s(G_1, u) + 3.5 s(G_2, u) = 3.5 s(G).
\]

%

\paragraph{Conclusion} From now on, we may assume that $G$ does not contain any cut vertices.
In particular, from now on, $G$ is a near-triangulation.

\subsection{Outerplanar}

Suppose $G$ is an outerplanar near-triangulation with $k$ vertices of degree 2.
There are two known bounds:
\begin{enumerate}
    \item $s(G) \le n / 3$, see \cite{matheson_dominating_1996}.
    \item $s(G) \le (n + k) / 4$, see \cite{campos_dominating_2013}.
\end{enumerate}
The linear combination $\frac{3}{7} \cdot (1) + \frac{4}{7} \cdot (2)$ yields
\[
    s(G) \le \frac{3}{7} \cdot \frac{n}{3} + \frac{4}{7} \cdot \frac{n+k}{4} = \frac{2n + k}{7} \le \frac{2}{7} \cdot \Phi(G).
\]
In other words, we ignore bad 3-pairs and just use known bounds.

\begin{remark}
    There is a also a direct proof based on finding $7$ dominating sets such that every ear tip is contained in $3$ of them
    and every other vertex in $2$ of them.
\end{remark}

\paragraph{Conclusion} From now on, $G$ is not outerplanar. In particular, it has at least one interior vertex.

\subsection{Non-trivial chord} \label{sect:nontrivialchord}

Suppose $G$ has a chord $u, v$. Let $G_1, G_2$ be the two sides, with $G_1$ having at least one interior vertex\footnote{In particular, $G_1$ is not a small A, B, OR, L+R, L, R}.

\begin{claim}\label{clai:nearbound}
    Then, $G_2$ satisfies Proposition~\ref{prop:nearact}.
\end{claim}
\begin{proof}
    $G_1$ has at least one interior vertex, hence any skeletal triangulation with the same number of interior vertices as $G_2$
    is smaller than $G$ due to having fewer interior vertices.
    In particular, any such graph satisfies Theorem~\ref{theo:mainresult} by the induction hypothesis.
\end{proof}

Intuitively speaking, the claim allows us to replace $G_2$ by a small OR, A, B that acts as the same type
or delete $G_2$ and argue about the low-degree problems we create. Formally, suppose $G_2$ acts as
\begin{itemize}[itemindent=60pt]
    \item [A+B, OR] Let $H$ be the result of attaching a small OR to $u, v \in G_2$. Then $H$ has smaller $\Phi$ than $G$
        as the small OR is the unique smallest rooted near triangulation that acts as OR,
        so by the induction hypothesis, $\Phi(H) \geq 3.5 s(H)$. By Claim~\ref{clai:nearbound} and Lemma~\ref{lemm:replaceattach}\footnote{%
            Technically speaking, the lemma as stated only applies to the OR case, but as A+B is ``strictly stronger'' than OR, a more careful%
            look shows that the A+B case also works.},
        also $\Phi(G) \geq 3.5 s(G)$.
    \item [A, B] Let $H$ be the result of attaching a small A, B to $u, v$. Then $H$ has smaller $\Phi$ than $G$, so by the induction hypothesis,
        $\Phi(H) \geq 3.5 s(H)$. By Claim~\ref{clai:nearbound} and Lemma~\ref{lemm:replaceattach},
        also $\Phi(G) \geq 3.5 s(G)$.
    \item [L, R, None] Let $H = G_1$, then $H$ has smaller $\Phi$ than $G$, so by the induction hypothesis, $\Phi(H) \geq 3.5 s(H)$ unless $H$ is
        a sporadic example. By
        Lemma~\ref{lemm:defuse}, $\Phi(G) + 0.5 \ge \Phi(H) + \phi(G_2, u_2, v_2)$. By Lemma~\ref{lemm:attachdom}, $s(G) \leq s(H) + s(G_2, u, v)$.
        By Claim~\ref{clai:nearbound}, $\phi(G_2, u_2, v_2) \ge 3.5 s(G_2, u_2, v_2) + 0.5$. Chaining inequalities yields
        \begin{align*}
            \Phi(G) &\ge \Phi(H) + \phi(G_2, u_2, v_2) - 0.5\\
            &\ge 3.5 s(H) + 3.5 s(G_2, u_2, v_2) + 0.5 - 0.5 \ge 3.5 s(G)
        \end{align*}
        If $H$ is a sporadic example, and $G_2$ acts as None, then $\Phi(H) \ge 3.5 s(H) - 1$ while $\phi(G_2, u, v) \ge 3.5s + 1.5$,
        so a similar chain of inequalities works. If $H$ is sporadic and $G_2$ acts as L, R, then $G_2$ can be used to cover one vertex of $H$:
        Then $\Phi(H) \ge 3.5s (H, u, v) + 2.5$ and, since $H$ acts as OCTA and $G_2$ as L / R, $s(G) \le s(H, u, v) + s(G_2, u, v)$. Chaining inequalities yields
        \begin{align*}
            \Phi(G) &\ge \Phi(H) + \phi(G_2, u_2, v_2) - 0.5\\
            &\ge 3.5 s(H, u, v) + 2.5 + 3.5 s(G_2, u_2,v_2) + 0.5 - 0.5 \ge 3.5 s(G) + 2.5
        \end{align*}
    \item [AND, L+R, L OR R, OCTA] By Claim~\ref{clai:nearbound}, $\phi(G_2, u_2, v_2) \ge 3.5 s(G_2, u_2, v_2) + 0.5$.
        Similar to lemma~\ref{lemm:unroot}, if there is no low-degree problem in $G_1$ at $u, v$, then $\Phi(G) = \Phi(G_1) + \phi(G_2, u_2, v_2)$
        and we conclude $\phi(G) \ge 3.5 s(G)$ as in the L, R case. This includes the case where $G_2$ is a sporadic example.
        Suppose now there is a low-degree problem in $G_1$ that involves $u$. Let $H = G_1 - u$, then by Lemma~\ref{lemm:deleteprob} $\Phi(H) \le \Phi(G_1) - 0.5$.
        As $G_2$ acts as AND, L+R, L OR R, OCTA, there is a rooted dominating set in $G_2$ that dominates $u$. Therefore, $s(G) \leq s(H) + s(G_2, u, v)$.
        Chaining inequalities similar to the L, R case yields
        \begin{align*}
            \Phi(G) &\ge \Phi(G_1) + \phi(G_2, u_2, v_2) - 0.5 \ge \Phi(H) + \phi(G_2, u_2, v_2)\\
            &\ge 3.5 s(H) + 3.5 s(G_2, u_2, v_2) \ge 3.5 s(G)
        \end{align*}
\end{itemize}

\subsection{Conclusion}
From now on, for every chord, one side is a small OR, A, B attachment.
More precisely, $G$ has exactly one 3-connected component with interior vertices. All other 3-connected components are copies of a small OR / A / B.
\begin{definition}
    The 3-connected component with interior vertices is called the \emph{polygon}. All other ones are
    the \emph{A, B, OR-attachments of $G$}.
\end{definition}
In fact, $G$ is the result of attaching its A, B, OR-attachments to its polygon.

\subsection{Notation for further cases}
\subsubsection{Polygon vertices}
As concluded in the previous case, $G$ now consists of a \emph{polygon} with \emph{attachments}.
\begin{definition}
    A \emph{polygon vertex} is a boundary vertex of $G$ that is part of the polygon.
\end{definition}

In the remaining cases, we (implicitly) use $s, t, u, v, w, x, y, z$ to denote a range of consecutive polygon vertices,
either in clockwise or counter-clockwise order. (These may not be distinct if the polygon is small.)
For example, if we say ``Suppose $\deg(v) = 3$ and $\deg(x) = 4$.'',
we really mean: Suppose there is a polygon vertex $v$ with $\deg(v) = 3$ and a polygon vertex $x$ with $\deg(x) = 4$
with exactly one polygon vertex $w$ in between.

\subsubsection{Simplified framework}
Explicitly arguing with inequalities and the induction hypothesis gets very tedious and distracts from more important parts of the proof. 
For the remaining cases, we use the following simplified framework:
Let $G$ be the graph we consider. By modifying $G$ slightly, we construct a new skeletal triangulation $H$.
We require $H$ to be smaller than $G$ and not one of the sporadic examples.
Then, by the induction hypothesis, $\Phi(H) \ge 3.5 s(H)$.
Let $\Delta \Phi := \Phi(G) - \Phi(H)$ and $\Delta s := s(G) - s(H)$.
We show that the decrease in $\Phi$ satisfies
\[
    \Delta \Phi = \Phi(G) - \Phi(H) \ge 3.5 (s(G) - s(H)) = 3.5 \Delta s.
\]
Together with the previous equation, this implies $\Phi(G) \ge 3.5 s(G)$,
as desired. To show this, we usually state a bound $\Delta s \le C$.
We then prove the bound $-\Delta \Phi \le -3.5 C$\footnote{Note the minus here. In general, $\Delta \Phi \ge 0$ describes the decrease in $\Phi$
whereas $-\Delta \Phi \le 0$ describes the net change in $\Phi$. We prefer this minus here, as it avoid a bunch of minuses in the next sentence.
Working with $\Delta s$ and $-\Delta \Phi$ also has the advantage of having to prove an \emph{upper} bound for both.}.
Our argument will be phrased as ``$\Phi$ decreases by $2$, decreases by $1.5$ and increases $0.5$'',
which really means $\Phi(H) = \Phi(G) - 2 - 1.5 + 0.5$, or equivalently, $-\Delta \Phi = -2 -1.5 + 0.5$.

\subsection{A, B attachments}
Suppose $G$ has one or more A, B attachments.

\subsubsection{Same red vertex}
Suppose $G$ has an A and a B attachment with the same red vertex $u$. See Figure~\ref{fig:samered}.
\begin{case}
\item[Construction] Construct $H$ by deleting the A attachment from $G$.
\item[Domination] A minimum neat dominating set $S \subseteq H$ contains $u$, due to the B attachment. In $G$, $u$ dominates the deleted A,
hence $S$ dominates $G$. Therefore $\Delta s \le 0$.
\item[Penalty] Deleting the $A$ decreases $\Phi$ by 2.5. This may create up to one low-degree problem, namely one involving $v$,
    increasing $\Phi$ by $\le 0.5$. Overall, $-\Delta \Phi \le -2.0$%
\item[Smaller] $H$ has smaller $\Phi$ than $G$\footnote{And the same number of interior vertices. We will only write down the highest ``priority'' difference between $H$ and $G$.}
\end{case}

\begin{figure}
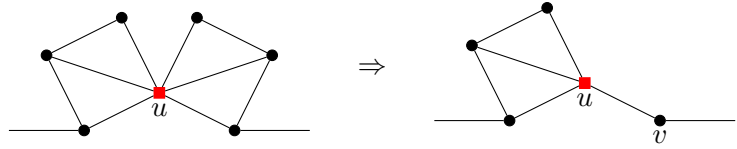

    \beforeafter{
        \node (A) at (0, 0) {};
        \node (B) [red, rectangle, label=below:$u$] at (1, 0.5) {};
        \node (C) at (2, 0) {};
        \node (w) at (-0.5, 1) {};
        \node (x) at (0.5, 1.5) {};
        \node (y) at (1.5, 1.5) {};
        \node (z) at (2.5, 1) {};
        \draw (A) -- (B) -- (C);
        \draw (A) -- (w) -- (x) -- (B) -- (w);
        \draw (C) -- (z) -- (y) -- (B) -- (z);
        \draw (A) -- +(-1, 0);
        \draw (C) -- +(1, 0);
    }{
        \node (A) at (0, 0) {};
        \node (B) [red, rectangle, label=below:$u$] at (1, 0.5) {};
        \node (C) [label=below:$v$] at (2, 0) {};
        \node (w) at (-0.5, 1) {};
        \node (x) at (0.5, 1.5) {};
        \draw (A) -- (B) -- (C);
        \draw (A) -- (w) -- (x) -- (B) -- (w);
        \draw (A) -- +(-1, 0);
        \draw (C) -- +(1, 0);
    }
    \caption{A and B attachment with same red vertex.}
    \label{fig:samered}
\end{figure}

\subsubsection{Consecutive red vertex}
Suppose $G$ has an A / B attachment with red vertex $v$ and another one with red vertex $w$.
Then, by the previous case, there are the only attachments on $u, v, w, x$. See Figure~\ref{fig:consecred}.
\begin{case}
    \item [Construction] Temporarily remove both A/ B attachments. Delete the edge $\{v, w\}$.
        Add the A / B attachments back, with the same red vertex as before, but possibly different base.
    \item[Domination] A minimum neat dominating set $S \subseteq H$ contains both $v$ and $w$, hence $S$ dominates $G$. Therefore $\Delta s \le 0$.
    \item[Penalty] Removing the A / B attachments and adding them back does not change $\Phi$.
        Deleting the edge does not create any low-degree problems.
        Therefore $- \Delta \Phi \le 0$.
    \item[Smaller] $H$ has one fewer interior vertex than $G$.
\end{case}
\begin{figure}
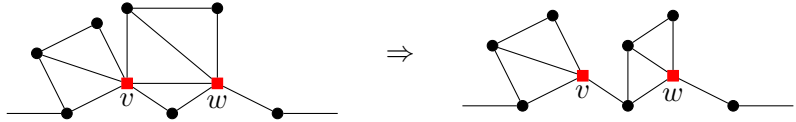

    \beforeafter[0.8]{
        \node (A) at (0, 0) {};
        \node (B) [red, rectangle, label=below:$v$] at (1, 0.5) {};
        \node (BB) [red, rectangle, label=below:$w$] at (2.5, 0.5) {};
        \node (I) at (1.75, 0) {};
        \node (C) at (3.5, 0) {};
        \node (w) at (-0.5, 1) {};
        \node (x) at (0.5, 1.5) {};
        \node (y) at (1, 1.75) {};
        \node (z) at (2.5, 1.75) {};
        \draw (A) -- (B) -- (BB) -- (C);
        \draw (B) -- (I) -- (BB);
        \draw (A) -- (w) -- (x) -- (B) -- (w);
        \draw (y) -- (BB) -- (z) -- (y) -- (B);
        \draw (A) -- +(-1, 0);
        \draw (C) -- +(1, 0);
    }{
        \node (A) at (0, 0) {};
        \node (B) [red, rectangle, label=below:$v$] at (1, 0.5) {};
        \node (BB) [red, rectangle, label=below:$w$] at (2.5, 0.5) {};
        \node (I) at (1.75, 0) {};
        \node (C) at (3.5, 0) {};
        \node (w) at (-0.5, 1) {};
        \node (x) at (0.5, 1.5) {};
        \node (y) at (1.75, 1) {};
        \node (z) at (2.5, 1.5) {};
        \draw (A) -- (B)  (BB) -- (C);
        \draw (B) -- (I) -- (BB);
        \draw (A) -- (w) -- (x) -- (B) -- (w);
        \draw (z) -- (y) -- (BB) -- (z) (I) -- (y);
        \draw (A) -- +(-1, 0);
        \draw (C) -- +(1, 0);
    }
    \caption{Two B attachments with consecutive red vertices.}
    \label{fig:consecred}
\end{figure}

\subsubsection{Red vertex with OR}

Suppose $G$ has an A / B attachment with red vertex $u$ and an OR attachment with base $u, v$. See Figure~\ref{fig:redwithor}.
\begin{case}
\item[Construction] Delete the OR.
\item[Domination] A minimum neat dominating set $S \subseteq H$ contains $u$ due to the A / B. In $G$, $u$ dominates the deleted OR,
hence $S$ dominates $G$. Therefore $\Delta s \le 0$.
\item[Penalty] Deleting the OR decreases $\Phi$ by 1.5. This may create up to one low-degree problem at $v$, increasing $\Phi$ by $\le 0.5$. Overall, $-\Delta \Phi \le -1.0$.
\item[Smaller] $H$ has smaller $\Phi$ than $G$.
\end{case}
\begin{figure}
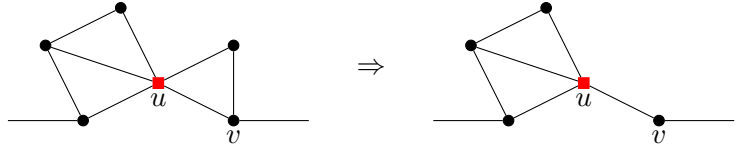

    \beforeafter{
        \node (A) at (0, 0) {};
        \node (B) [red, rectangle, label=below:$u$] at (1, 0.5) {};
        \node (C) [label=below:$v$] at (2, 0) {};
        \node (w) at (-0.5, 1) {};
        \node (x) at (0.5, 1.5) {};
        \node (y) at (2, 1) {};
        \draw (A) -- (B) -- (C);
        \draw (A) -- (w) -- (x) -- (B) -- (w);
        \draw (C) -- (y) -- (B);
        \draw (A) -- +(-1, 0);
        \draw (C) -- +(1, 0);
    }{
        \node (A) at (0, 0) {};
        \node (B) [red, rectangle, label=below:$u$] at (1, 0.5) {};
        \node (C) [label=below:$v$] at (2, 0) {};
        \node (w) at (-0.5, 1) {};
        \node (x) at (0.5, 1.5) {};
        \draw (A) -- (B) -- (C);
        \draw (A) -- (w) -- (x) -- (B) -- (w);
        \draw (A) -- +(-1, 0);
        \draw (C) -- +(1, 0);
    }
    \caption{B attachment with OR.}
    \label{fig:redwithor}
\end{figure}

\subsubsection{Red vertex next to $5^{+}$ vertex or next to vertex with attachment}
Suppose $G$ has an A / B attachment with red vertex $v$ and that $\deg(w) \geq 5$ or that there is an attachment with base $w, x$.
See Figure~\ref{fig:redfive}.
\begin{case}
\item[Construction] Delete the edge $v, w$.
\item[Domination] $H$ is a spanning subgraph of $G$, hence $\Delta s \le 0$.
\item[Penalty] After the deletion, $\deg(w) \geq 4$ or there is an attachment with base $w, x$.
    In either case, $w$ is not involved in a low-degree problem. Thus $\Delta \Phi = 0$.
\item[Smaller] $H$ has one fewer interior vertex than $G$.
\end{case}
\begin{figure}
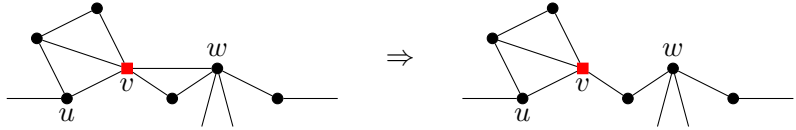

    \beforeafter[0.8]{
        \node (A) [label=below:$u$] at (0, 0) {};
        \node (B) [red, rectangle, label=below:$v$] at (1, 0.5) {};
        \node (BB) [label=above:$w$] at (2.5, 0.5) {};
        \node (I) at (1.75, 0) {};
        \node (C) at (3.5, 0) {};
        \node (w) at (-0.5, 1) {};
        \node (x) at (0.5, 1.5) {};
        \draw (A) -- (B) -- (BB) -- (C);
        \draw (B) -- (I) -- (BB);
        \draw (A) -- (w) -- (x) -- (B) -- (w);
        \draw (A) -- +(-1, 0);
        \draw (C) -- +(1, 0);
        \draw (BB) --+ (-105:1) (BB) --+ (-75:1);
    }{
        \node (A) [label=below:$u$] at (0, 0) {};
        \node (B) [red, rectangle, label=below:$v$] at (1, 0.5) {};
        \node (BB) [label=above:$w$] at (2.5, 0.5) {};
        \node (I) at (1.75, 0) {};
        \node (C) at (3.5, 0) {};
        \node (w) at (-0.5, 1) {};
        \node (x) at (0.5, 1.5) {};
        \draw (A) -- (B)  (BB) -- (C);
        \draw (B) -- (I) -- (BB);
        \draw (A) -- (w) -- (x) -- (B) -- (w);
        \draw (A) -- +(-1, 0);
        \draw (C) -- +(1, 0);
        \draw (BB) --+ (-105:1) (BB) --+ (-75:1);
    }
    \caption{Red vertex next to $5^{+}$ vertex.}
    \label{fig:redfive}
\end{figure}

\subsubsection{Red vertex next to $3,4$ vertex}
Suppose $G$ has an A / B attachment with base $u, v$ and red vertex $v$ and that $\deg(w) \in \{3, 4\}$.
By the previous cases, there is no attachment with base containing $w$. See Figure~\ref{fig:redthreefour}.
\begin{case}
\item[Construction] Delete $w$.
\item[Domination] A minimum neat dominating set $S \subseteq H$ contains $v$, which dominates $w$. Therefore $\Delta s \le 0$.
\item[Penalty] Deleting $w$ decreases $\Phi$ by $\ge 1.0$. As $\deg(w) \le 4$, this creates at most two new low-degree problems, increasing $\Phi$ by $\le 1.0$.
Overall, $-\Delta \Phi \le 0.0$.
\item[Smaller] $H$ has at least one fewer interior vertex than $G$.
\end{case}
\begin{figure}
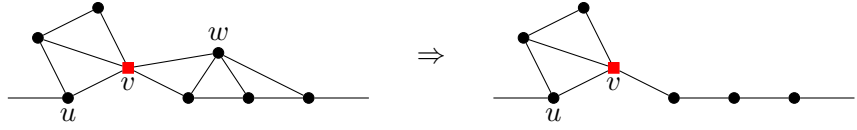

    \beforeafter[0.8]{
        \node (A) [label=below:$u$] at (0, 0) {};
        \node (B) [red, rectangle, label=below:$v$] at (1, 0.5) {};
        \node (BB) [label=above:$w$] at (2.5, 0.75) {};
        \node (I) at (2, 0) {};
        \node (J) at (3, 0) {};
        \node (C) at (4, 0) {};
        \node (w) at (-0.5, 1) {};
        \node (x) at (0.5, 1.5) {};
        \draw (A) -- (B) -- (BB) -- (C);
        \draw (B) -- (I) -- (BB);
        \draw (A) -- (w) -- (x) -- (B) -- (w);
        \draw (A) -- +(-1, 0);
        \draw (C) -- +(1, 0);
        \draw (BB) -- (J);
        \draw (I) -- (J) -- (C);
    }{
        \node (A) [label=below:$u$] at (0, 0) {};
        \node (B) [red, rectangle, label=below:$v$] at (1, 0.5) {};
        \node (I) at (2, 0) {};
        \node (J) at (3, 0) {};
        \node (C) at (4, 0) {};
        \node (w) at (-0.5, 1) {};
        \node (x) at (0.5, 1.5) {};
        \draw (A) -- (B);
        \draw (B) -- (I);
        \draw (A) -- (w) -- (x) -- (B) -- (w);
        \draw (A) -- +(-1, 0);
        \draw (C) -- +(1, 0);
        \draw (I) -- (J) -- (C);
    }
    \caption{Red vertex next to $3,4$ vertex.}
    \label{fig:redthreefour}
\end{figure}

\subsubsection{Conclusion}
From now on, $G$ has no A / B attachments. Thus, $G$ has only OR attachments. In particular, every boundary vertex of degree $\ge 3$ is a polygon
vertex and is hence adjacent to an interior vertex.

\subsection{Consecutive low-degree vertices.}

\subsubsection{Bad 5-wheel} \label{sect:badfive}
Suppose $v, w$ is a 3-pair in a bad 5-wheel. Then, as $u, x$ is an edge between boundary vertices and $p$ is an interior vertex, the polygon is just $u, v, w, x$, i.e. $y = u$.
There may or may not be at OR attachment at $x, u$ and there may or may not be vertices inside the triangle $u, p, x$. See Figure~\ref{fig:threepair}.
\begin{case}
\item[Construction] Let $p$ be the interior vertex adjacent to $v$ and $w$. Delete $v, w$. Force $p$ by attaching an B to $u, p$. Cover $x$ by fusing a small LR to $x$.
\item[Domination] A minimum neat dominating set $S \subseteq H$ contains $p$ and contains exactly one vertex, say $s \neq x$ in the small LR. Then $s$ dominates only the small LR (including $x$).
Then $S - s$ dominates $G$, as $x$ is dominated by $p$. $\Delta s \le -1$.
\item[Penalty] Deleting $v, w$ decreases $\Phi$ by $2.5$. Attaching the A increases $\Phi$ by 2.5. Fusing the LR increases $\Phi$ by 3.5. Overall, $-\Delta \Phi \le 3.5$ (no minus here as $\Delta s$ is negative).
\item[Smaller] $H$ has one fewer interior vertex than $G$, so $H$ is smaller despite $\Phi(H) > \Phi(G)$.
\end{case}
\begin{figure}
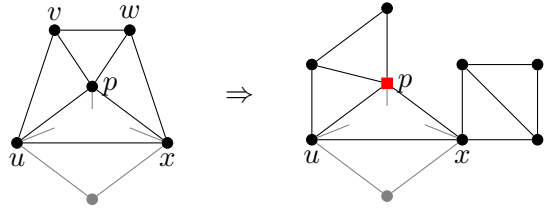

    \beforeafter[1.0]{
        \node (A) at (-1, 0)[label=below:$u$] {};
        \node (B) at (0, 0.75) [label=right:$p$] {};
        \node (C) at (1, 0) [label=below:$x$] {};
        \node (X) at (-0.5, 1.5) [label=above:$v$] {};
        \node (Y) at (0.5, 1.5) [label=above:$w$] {};
        \node (M) at (0, -0.75) [gray] {};
        \draw (A) -- (B) -- (C) -- (A);
        \draw (A) -- (X) -- (Y) -- (C);
        \draw (X) -- (B) -- (Y);
        \draw[gray] (A) -- (M) -- (C);
        \draw[gray] (A) --+ (0.5, 0.2) (C) --+ (-0.5, 0.2) (B) --+ (0, -0.3);
    }{
        \node (A) at (-1, 0) [label=below:$u$] {};
        \node (B) at (0, 0.75) [red,rectangle,label=right:$p$] {};
        \node (C) at (1, 0) [label=below:$x$] {};
        \node (X) at (0, 1.75) {};
        \node (Y) at (-1, 1) {};
        \node (U) at (2, 0) {};
        \node (V) at (2, 1) {};
        \node (W) at (1, 1) {};
        \node (M) at (0, -0.75) [gray] {};
        \draw (A) -- (B) -- (C) -- (A);
        \draw (A) -- (Y) -- (X);
        \draw (X) -- (B) -- (Y);
        \draw (C) -- (U) -- (V) -- (W) -- (C);
        \draw (U) -- (W);
        \draw[gray] (A) -- (M) -- (C);
        \draw[gray] (A) --+ (0.5, 0.2) (C) --+ (-0.5, 0.2) (B) --+ (0, -0.3);
    }
    \caption{Bad 5-wheel. The gray things may or may not exist.}
    \label{fig:threepair}
\end{figure}

\subsubsection{Consecutive OR attachments}
Suppose there is an OR attachment at $u, v$ and another one at $v, w$. See Figure~\ref{fig:consecor}.
\begin{case}
\item[Construction] Delete both ORs. Force $v$ by attaching a B to $u, v$.
\item[Domination] A minimum neat dominating set $S \subseteq H$ contains $v$. In $G$, $v$ dominates both ORs, hence $S$ dominates $G$.
    $\Delta s \le 0$.
\item[Penalty] Deleting both ORs decreases $\Phi$ by $3$. Attaching the A increases $\Phi$ by 2.5. The deletion might create a low-degree problem involving $w$,
    increasing $\Phi$ by 0.5. Overall, $-\Delta \Phi \le 0$.
\item[Smaller] $H$ has fewer degree-2 vertices than $G$.
\end{case}
\begin{figure}
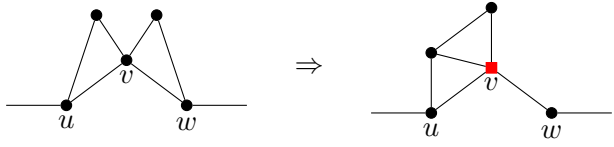

    \beforeafter[0.8]{
        \node (A) at (-1, 0) [label=below:$u$] {};
        \node (B) at (0, 0.75) [label=below:$v$] {};
        \node (C) at (1, 0) [label=below:$w$] {};
        \node (X) at (-0.5, 1.5) {};
        \node (Y) at (0.5, 1.5) {};
        \draw (A) -- (B) -- (C);
        \draw (A) -- (X) (Y) -- (C);
        \draw (X) -- (B) -- (Y);
        \draw (A) -- +(-1, 0);
        \draw (C) -- +(1, 0);
    }{
        \node (A) at (-1, 0) [label=below:$u$] {};
        \node (B) at (0, 0.75) [red,rectangle,label=below:$v$] {};
        \node (C) at (1, 0) [label=below:$w$] {};
        \node (X) at (0, 1.75) {};
        \node (Y) at (-1, 1) {};
        \draw (A) -- (B) -- (C);
        \draw (A) -- (Y) -- (X);
        \draw (X) -- (B) -- (Y);
        \draw (A) -- +(-1, 0);
        \draw (C) -- +(1, 0);
    }
    \caption{Consecutive OR attachments.}
    \label{fig:consecor}
\end{figure}

\subsubsection{Degree-3 triple}
Suppose that $\deg(v) = \deg(w) = \deg(x) = 3$.
We allow $u = x$, i.e. $G = K_4$.)
Let $p$ be the interior vertex adjacent to $v, w, x$. See Figure~\ref{fig:consecthreetriple}.
\begin{case}
\item[Construction] Delete $v, w, x$. Force $p$ by attaching an A to $p, y$. Fuse a small LR to $u$. 
\item[Domination] A minimum neat dominating set $S \subseteq H$ contains $p$ and contains exactly one vertex $s \neq u$ in the small LR.
    Then $S - s$ dominates $G$, as $p$ dominates $u, v, w, x, y$. $\Delta s \le -1$.
\item[Penalty] Deleting $v, w, x$ decreases $\Phi$ by $3$. Attaching the A increases $\Phi$ by 2.5. Fusing an LR increases $\Phi$ by 3.5. The deletions do not create any
    low-degree problems, as $u, p, y$ each get something fused / attached to them. Overall, $-\Delta \Phi \le 3$.
\item[Smaller] $H$ has fewer interior vertices.
\end{case}
\begin{figure}
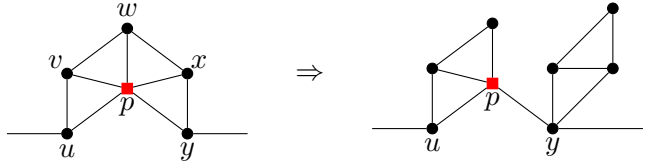

    \beforeafter[0.8]{
        \node (A) at (-1, 0) [label=below:$u$] {};
        \node (B) at (0, 0.75) [red,rectangle,label=below:$p$] {};
        \node (C) at (1, 0) [label=below:$y$] {};
        \node (X) at (0, 1.75) [label=above:$w$] {};
        \node (Y) at (-1, 1) [label=above left:$v$] {};
        \node (Z) at (1, 1) [label=above right:$x$] {};
        \draw (A) -- (B) -- (C);
        \draw (A) -- (Y) -- (X) -- (Z) -- (C);
        \draw (X) -- (B) -- (Y);
        \draw (B) -- (Z);
        \draw (A) -- +(-1, 0);
        \draw (C) -- +(1, 0);
    }{
        \node (A) at (-1, 0) [label=below:$u$] {};
        \node (B) at (0, 0.75) [red,rectangle,label=below:$p$] {};
        \node (C) at (1, 0) [label=below:$y$] {};
        \node (X) at (0, 1.75) {};
        \node (Y) at (-1, 1) {};
        \node (N) at (1, 1) {};
        \node (M) at (2, 1) {};
        \node (O) at (2, 2) {};
        \draw (A) -- (B) -- (C);
        \draw (A) -- (Y) -- (X);
        \draw (X) -- (B) -- (Y);
        \draw (C) -- (N) -- (M) -- (C);
        \draw (N) -- (O) -- (M);
        \draw (A) -- +(-1, 0);
        \draw (C) -- +(1.5, 0);
    }
    \caption{Degree-3 triple.}
    \label{fig:consecthreetriple}
\end{figure}

\subsubsection{Conclusion}
Now $G$ has no bad 5-wheels and all OR attachments have disjoint base vertices.
Moreover, no three consecutive boundary vertices all have degree $3$.

\subsection{Unproblematic ORs}

\subsubsection{OR next to \boldmath${5^{+}}$ vertex}
Suppose there is an OR attachment at $u, v$ and that $\deg(w) \geq 5$. See Figure~\ref{fig:orfive}.
\begin{case}
\item[Construction] Delete the edge $v, w$.
\item[Domination] $H$ is a spanning subgraph of $G$, hence $\Delta s \le 0$.
\item[Penalty] After the deletion, $\deg(w) \geq 4$, so this does not create low-degree problems. $-\Delta \Phi = 0$.
\item[Smaller] $H$ has one fewer interior vertex.
\end{case}
\begin{figure}
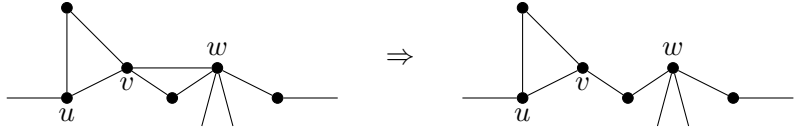

    \beforeafter[0.8]{
        \node (A) [label=below:$u$] at (0, 0) {};
        \node (B) [label=below:$v$] at (1, 0.5) {};
        \node (BB) [label=above:$w$] at (2.5, 0.5) {};
        \node (I) at (1.75, 0) {};
        \node (C) at (3.5, 0) {};
        \node (x) at (0, 1.5) {};
        \draw (A) -- (B) -- (BB) -- (C);
        \draw (B) -- (I) -- (BB);
        \draw (A) -- (x) -- (B);
        \draw (A) -- +(-1, 0);
        \draw (C) -- +(1, 0);
        \draw (BB) --+ (-105:1) (BB) --+ (-75:1);
    }{
        \node (A) [label=below:$u$] at (0, 0) {};
        \node (B) [label=below:$v$] at (1, 0.5) {};
        \node (BB) [label=above:$w$] at (2.5, 0.5) {};
        \node (I) at (1.75, 0) {};
        \node (C) at (3.5, 0) {};
        \node (x) at (0, 1.5) {};
        \draw (A) -- (B)  (BB) -- (C);
        \draw (B) -- (I) -- (BB);
        \draw (A) -- (x) -- (B);
        \draw (A) -- +(-1, 0);
        \draw (C) -- +(1, 0);
        \draw (BB) --+ (-105:1) (BB) --+ (-75:1);
    }
    \caption{OR next to $5^{+}$ vertex.}
    \label{fig:orfive}
\end{figure}

\subsubsection{Two ORs one edge apart}
Suppose there is an OR attachment at $u, v$ and another OR attachment attachment at $w, x$. See Figure~\ref{fig:orappart}.
\begin{case}
\item [Construction] Delete the edge $v, w$.
\item [Domination] $H$ is a spanning subgraph of $G$, hence $\Delta s \le 0$.
\item [Penalty] Only the degrees of $v, w$ are affected. Both are incident to a chord, so this does not create low-degree problems by Remark~\ref{rem:noninvolved}.
\item [Smaller] $H$ has one fewer interior vertex.
\end{case}
\begin{figure}
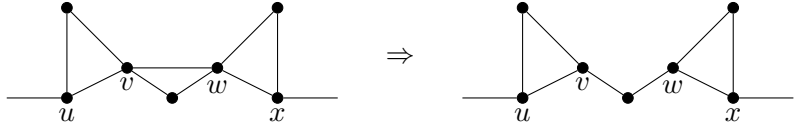

    \beforeafter[0.8]{
        \node (A) [label=below:$u$] at (0, 0) {};
        \node (B) [label=below:$v$] at (1, 0.5) {};
        \node (BB) [label=below:$w$] at (2.5, 0.5) {};
        \node (I) at (1.75, 0) {};
        \node (C) [label=below:$x$] at (3.5, 0) {};
        \node (y) at (3.5, 1.5) {};
        \node (x) at (0, 1.5) {};
        \draw (A) -- (B) -- (BB) -- (C);
        \draw (B) -- (I) -- (BB);
        \draw (A) -- (x) -- (B);
        \draw (A) -- +(-1, 0);
        \draw (C) -- +(1, 0);
        \draw (BB) -- (y) -- (C);
    }{
        \node (A) [label=below:$u$] at (0, 0) {};
        \node (B) [label=below:$v$] at (1, 0.5) {};
        \node (BB) [label=below:$w$] at (2.5, 0.5) {};
        \node (I) at (1.75, 0) {};
        \node (C) [label=below:$x$] at (3.5, 0) {};
        \node (x) at (0, 1.5) {};
        \node (y) at (3.5, 1.5) {};
        \draw (A) -- (B)  (BB) -- (C);
        \draw (B) -- (I) -- (BB);
        \draw (A) -- (x) -- (B);
        \draw (A) -- +(-1, 0);
        \draw (C) -- +(1, 0);
        \draw (BB) -- (y) -- (C);
    }
    \caption{Two ORs one edge apart.}
    \label{fig:orappart}
\end{figure}

\subsubsection{OR on triangle boundary}
Suppose there is an OR attachment at $v, w$ and that $x = u$, i.e. the polygon is a triangle. See Figure~\ref{fig:ortriangle}.
\begin{case}
\item[Construction] Delete $w$.
\item[Domination] A neat dominating set in $H$ contains either $u$ or $v$. That vertex then dominates $w$. $\Delta s = 0$.
\item[Penalty] Deleting $w$ decreases $\Phi$ by $1$. By the previous case, $\deg(w) \le 4$, hence the deletion creates at most
    one low-degree problem, increasing $\Phi$ by $\le 0.5$. Overall, $-\Delta \Phi \le -0.5$.
\item[Smaller] $H$ has at least one fewer interior vertex.
\end{case}
\begin{figure}
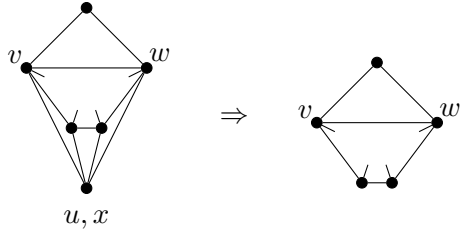

    \beforeafter[0.8]{
        \node (A) at (-1, 0.5) [label=above left:$v$] {};
        \node (B) at (1, 0.5) [label=above right:$w$] {};
        \node (C) at (0, 1.5) {};
        \node (D) at (0, -1.5) [label=below:{$u,x$}] {};
        \node (x) at (-0.25, -0.5) {};
        \node (y) at (0.25, -0.5) {};
        \draw (A) -- (B) -- (C) -- (A);
        \draw (A) -- (D) -- (B);
        \draw (A) -- (x) -- (D) -- (y) -- (B);
        \draw (x) -- (y);
        \draw (x) --+ (0.1, 0.3);
        \draw (y) --+ (-0.1, 0.3);
        \draw (A) --+ (0.3, -0.15);
        \draw (B) --+ (-0.3, -0.15);
    }{
        \node (A) at (-1, 0.5) [label=above left:$v$] {};
        \node (B) at (1, 0.5) [label=above right:$w$] {};
        \node (C) at (0, 1.5) {};
        \node (x) at (-0.25, -0.5) {};
        \node (y) at (0.25, -0.5) {};
        \draw (A) -- (B) -- (C) -- (A);
        \draw (A) -- (x) (y) -- (B);
        \draw (x) -- (y);
        \draw (x) --+ (0.1, 0.3);
        \draw (y) --+ (-0.1, 0.3);
        \draw (A) --+ (0.3, -0.15);
        \draw (B) --+ (-0.3, -0.15);
    }
    \caption{OR on triangle polygon.}
    \label{fig:ortriangle}
\end{figure}

\subsubsection{Conclusion}
Now every OR has a base that lies between two distinct polygon vertices, both of degree $\{3, 4\}$,
both not part of any attachment.

\subsection{Deleting ORs}
Suppose there is an OR attachment at $x, y$. Then $\deg(w), \deg(z) \in \{3, 4\}$ and $w \neq z$.
(It could happen that $v = z$.) Let $p$ be the interior vertex adjacent to $x, y$ s.t. $\{p, x, y\}$ is a facial triangle.

\subsubsection{Interior degree-3 neighbor}
Suppose there is an interior vertex $r$ with $\deg(r) = 3$ adjacent to $x$ (or $y$).
See Figure~\ref{fig:ordegtree}.

\begin{case}
\item[Construction] Delete $r$. Replace the OR by an A with red vertex $x$.
    Fuse a small LR to $w$.
\item[Domination] If $S$ is a minim neat dominating set in $H$ and $L \subseteq S$ is the vertex in the fused LR,
    then $S \setminus L$ is a dominating set in $G$ containing $x$, which dominates $w, r$. $\Delta s = -1$.
\item[Penalty] Deleting $r$ and the OR decreases $\Phi$ by $2.5$. Attaching the A increases $\Phi$ by $2.5$.
    Fusing the LR increases $\Phi$ by $3.5$. Deleting $r$ does not create any low-degree problem,
    as the only boundary vertices possibly adjacent to $r$ are $w, x, y$. Overall, $-\Delta \Phi \le 3.5$. 
\item[Smaller] $H$ has fewer interior vertices.
\end{case}
\begin{figure}
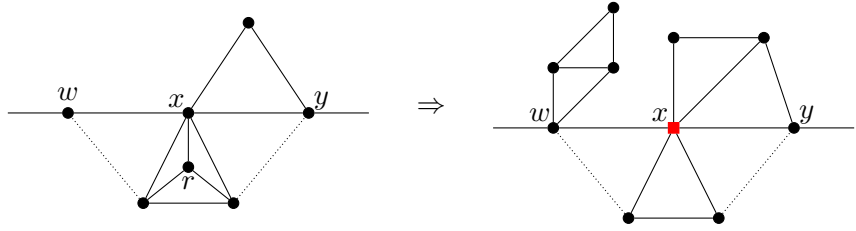

    \beforeafter[0.8]{
        \node (A) at (0, 0) [label=above left:$x$] {};
        \node (B) at (2, 0) [label=above right:$y$] {};
        \node (L) at (-2, 0) [label=above:$w$] {};
        \node (C) at (1, 1.5)  {};
        \node (q) at (-0.75, -1.5) {};
        \node (r) at (0.75, -1.5) {};
        \node (I) at (0, -0.9) [label=below:$r$] {};
        \draw (A) -- (B) -- (C) -- (A) -- (L);
        \draw (A) -- (q) -- (r) -- (A);
        \draw (I) -- (q) (I) -- (r) (I) -- (A);
        \draw (L) --+ (-1, 0);
        \draw (B) --+ (1, 0);
        \draw[densely dotted] (q) -- (L) (r) -- (B);
    }{
        \node (A) at (0, 0) [label=above left:$x$, red,rectangle] {};
        \node (B) at (2, 0) [label=above right:$y$] {};
        \node (L) at (-2, 0) [label=above left:$w$] {};
        \node (x) at (-2, 1) {};
        \node (y) at (-1, 1) {};
        \node (z) at (-1, 2) {};
        \node (C) at (1.5, 1.5) {};
        \node (D) at (0, 1.5) {};
        \node (q) at (-0.75, -1.5) {};
        \node (r) at (0.75, -1.5) {};
        \node (I) at (0, -0.9) [transparent] {};
        \draw (C) -- (D) -- (A) -- (B) -- (C) -- (A) -- (L);
        \draw (A) -- (q) -- (r) -- (A);
        \draw (L) -- (x) -- (y) -- (L) (x) -- (z) -- (y);
        \draw (L) --+ (-1, 0);
        \draw (B) --+ (1, 0);
        \draw[densely dotted] (q) -- (L) (r) -- (B);
    }
    \caption{OR with interior degree-3 neighbor.}
    \label{fig:ordegtree}
\end{figure}

\subsubsection{Antipodal 5-wheels}
Suppose there is an interior vertex $q \neq w$ adjacent to both $x$ and $y$ with $\deg(q) = 4$.
Then $N[q]$ is a 5-wheel with $x, y$ on antipodal sides of the wheel.
Let $r_1, r_2$ be the other neighbors of $q$, i.e. $N(q) = \{x, y, r_1, r_2\}$. See Figure~\ref{fig:orantiwheel}.
\begin{case}
\item[Construction] Delete $w$. Add the edge $r_1, r_2$.
\item[Domination] A neat dominating set in $H$ contains either $x$ or $y$. That vertex then dominates $w, r_1, r_2$. In particular,
    the added edge has no effect on $s$ and $\Delta s \le 0$.
\item[Penalty] Deleting $w$ decreases $\Phi$ by $1$. The deletion only decreases the degrees of $x, y$ so it does not create low degree problems.
    $-\Delta \Phi \le -1$.
\item[Smaller] $H$ has one fewer interior vertex.
\end{case}
\begin{figure}
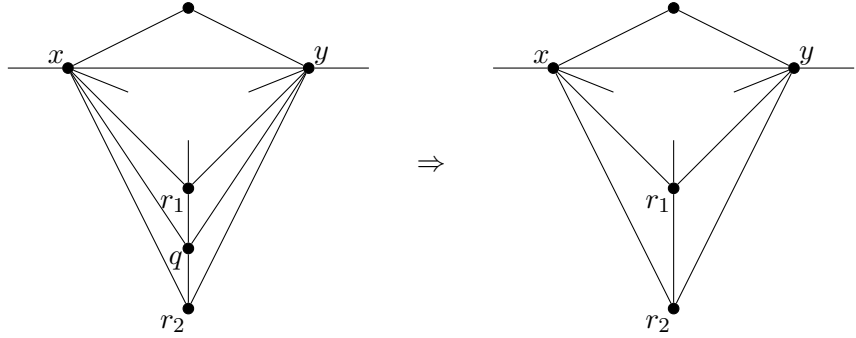

    \beforeafter[0.8]{
        \node (A) at (-2, 0.5) [label=above left:$x$] {};
        \node (B) at (2, 0.5) [label=above right:$y$] {};
        \node (C) at (0, 1.5)  {};
        \node (D) at (0, -1.5) [label=below left:$r_1$] {};
        \node (E) at (0, -2.5) [label=below left:$q$] {};
        \node (F) at (0, -3.5) [label=below left:$r_2$] {};
        \draw (A) -- (B) -- (C) -- (A);
        \draw (A) -- (D) -- (B);
        \draw (A) -- (E) -- (B);
        \draw (A) -- (F) -- (B);
        \draw (D) -- (E) -- (F);
        \draw (A) --+ (1, -0.4);
        \draw (B) --+ (-1, -0.4);
        \draw (D) --+ (0, 0.8);
        \draw (A) --+ (-1, 0);
        \draw (B) --+ (1, 0);
    }{
        \node (A) at (-2, 0.5) [label=above left:$x$] {};
        \node (B) at (2, 0.5) [label=above right:$y$] {};
        \node (C) at (0, 1.5) {};
        \node (D) at (0, -1.5) [label=below left:$r_1$] {};
        \node (F) at (0, -3.5) [label=below left:$r_2$] {};
        \draw (A) -- (B) -- (C) -- (A);
        \draw (A) -- (D) -- (B);
        \draw (A) -- (F) -- (B);
        \draw (D) -- (F);
        \draw (A) --+ (1, -0.4);
        \draw (B) --+ (-1, -0.4);
        \draw (D) --+ (0, 0.8);
        \draw (A) --+ (-1, 0);
        \draw (B) --+ (1, 0);
    }
    \caption{OR with Antipodal 5-wheel.}
    \label{fig:orantiwheel}
\end{figure}

\subsubsection{Octahedral interior 4-pairs}
Let $c(x)$ denote the number of octahedral interior 4-pairs adjacent to $x$.
Suppose $c(x) \geq 1$. See Figure~\ref{fig:orocta}.
\begin{case}
\item [Construction] Delete the OR and $x$, but keep $y$. Delete both vertices of every interior 4-pair adjacent to $x$.
    Fuse a small LR to every remaining neighbor of $x$. Suppose we fuse $k$ LRs this way.
\item [Domination] If $S$ is a minimum dominating set in $H$ and $L \subseteq S$ are the $k$ vertices in the fused LRs,
    then $S \cup \{x\} \setminus L$ is a dominating set in $G$. $\Delta s \le k-1$.
\item [Penalty] Deleting the OR and $x$ decreases $\Phi$ by 2.5. Deleting interior 4-pairs decreases $\Phi$ by $2 \cdot c(x)$.
    Any low-degree problem created by deleting the OR and $x$ get deleted or covered by an LR.
    Deleting interior 4-pairs creates exactly one ear per pair, increasing $\Phi$ by $0.5 \cdot c(x)$.
    Fusing the LRs increases $\Phi$ by $3.5 k$. Overall,
    \[
        -\Delta \Phi \le -2.5 -1.5 c(x) + 3.5k \le -4 + 3.5k.
    \]
\item [Smaller] $H$ has fewer interior vertices.
\end{case}
\begin{figure}
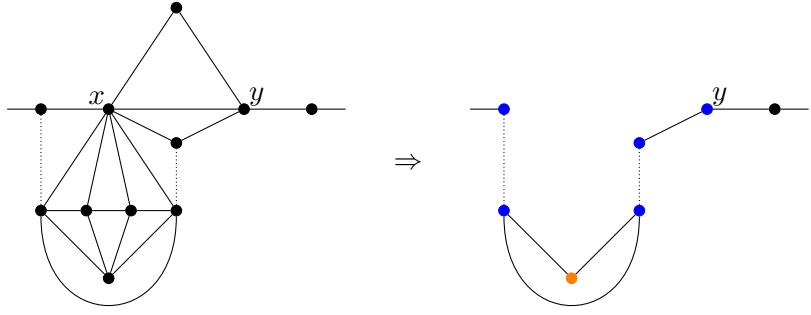

    \beforeafter[0.9]{
        \node (L) at (-1, 0) {};
        \node (R) at (3, 0) {};
        \node (A) at (0, 0) [label=above left:$x$] {};
        \node (B) at (2, 0) [label=above right:$y$] {};
        \node (C) at (1, 1.5)  {};
        \node (D) at (1, -0.5) {};
        \node (p) at (-1, -1.5) {};
        \node (q) at (-0.33, -1.5) {};
        \node (r) at (0.33, -1.5) {};
        \node (s) at (1, -1.5) {};
        \node (I) at (0, -2.5) {};
        \draw (A) -- (p) -- (I) -- (q) -- (A) -- (r) -- (I) -- (s) -- (A);
        \draw (p) -- (q) -- (r) -- (s);
        \draw (A) -- (B) -- (C) -- (A);
        \draw (L) --+ (-0.5, 0);
        \draw (R) --+ (0.5, 0);
        \draw (A) -- (D) -- (B);
        \draw (L) -- (A) (B) -- (R);
        \draw [densely dotted] (s) -- (D) (L) -- (p);
        \draw (p) to [in=-90, out=-90, distance=50] (s);
    }{
        \node[blue] (L) at (-1, 0) {};
        \node (R) at (3, 0) {};
        \node[blue] (B) at (2, 0) [label=above right:$y$] {};
        \node[transparent] (C) at (1, 1.5)  {};
        \node[blue] (D) at (1, -0.5) {};
        \node[blue] (p) at (-1, -1.5) {};
        \node[blue] (s) at (1, -1.5) {};
        \node[orange] (I) at (0, -2.5) {};
        \draw (p) -- (I) (I) -- (s);
        \draw (L) --+ (-0.5, 0);
        \draw (R) --+ (0.5, 0);
        \draw (D) -- (B);
        \draw (B) -- (R);
        \draw [densely dotted] (s) -- (D) (L) -- (p);
        \draw (p) to [in=-90, out=-90, distance=50] (s);
    }
    \caption{Octahedral interior 4-pairs. Here, $c(x)$. The dashed lines represent a path of vertices all adjacent to $x$.
    The blue vertices each have an LR fused to them (not drawn here, to keep the picture clean). The orange vertices mark locations where
    a low-degree problem might be created.}
    \label{fig:orocta}
\end{figure}

\subsubsection{No interior problems}
The previous cases now allow us to delete both $x$ and $y$, without creating too many low-degree problems.
If $v = z$, then $\deg(w) = \deg(z) = 3$  makes $w, z$ a 3-pair in the bad 5-wheel $w, x, y, z, p$,
which is covered by Case~\ref{sect:badfive}. Therefore,
suppose that $v \ne z$ or that WLOG $\deg(z) \ge 4$. See Figure~\ref{fig:ornofourpairs}.
\begin{case}
\item [Construction] Delete the OR, $w, x, y$. Fuse a small LR to every remaining neighbor of $x$. Suppose we fuse $k$ LRs this way.
\item [Domination] If $S$ is a minimum dominating set in $H$ and $L \subseteq S$ are the $k$ vertices in the LRs,
    then $S \cup \{x\} \setminus L$ is a dominating set in $G$. $\Delta s \le k-1$.
\item [Penalty] The deletions decrease $\Phi$ by 4.5. Fusing LRs increases $\Phi$ by $3.5 k$.
    The deletions may create one low-degree problem involving $z$ and,
    as $\deg(w) \le 4$, at most one involving a former neighbor of $w$.
    Overall, these increase $\Phi$ by $\le 1$.
    If $v = z$, then $\deg(z)$ decreases by two, so we need $\deg(z) \ge 4$ to avoid creating a leaf.  
    There are no other low-degree problems: All former neighbors of $x$
    get covered by an LR. Former neighbors of $y$ cannot be involved in a low degree problem in which $w, z$ are not involved,
    as that would be require an interior vertex of degree $3$ (for ears or pivoting triangles), an interior octahedral 4-pair (for 3-pairs in a bad 5-wheel)
    or a chord from $y$ to a degree-3 boundary vertex (for a degree-2 cut vertex)\footnote{%
        There is a bit of subtlety if $w$ and $y$ share interior neighbors of degree $4$.}.
    All of these are covered by previous cases.
    Overall,
    \[
        -\Delta \Phi \le -4.5 + 1 + 3.5k \le -3.5 (k-1).
    \]
\item [Smaller] $H$ has fewer interior vertices.
\end{case}
\begin{figure}
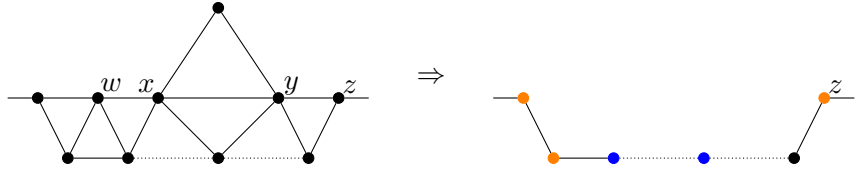

    \beforeafter[0.8]{
        \node (A) at (0, 0) [label=above left:$x$] {};
        \node (B) at (2, 0) [label=above right:$y$] {};
        \node (C) at (1, 1.5) {};
        \node (D) at (1, -1.0) {};
        \node (R) at (3, 0) [label=above right:$z$] {};
        \node (S) at (2.5, -1) {};
        \node (L) at (-1, 0) [label=above right:$w$] {};
        \node (M) at (-2, 0) {};
        \node (N) at (-1.5, -1) {};
        \node (O) at (-0.5, -1) {};
        \draw (B) -- (S) -- (R) {};
        \draw (M) --+ (-0.5, 0);
        \draw (R) --+ (0.5, 0);
        \draw (M) -- (L) -- (A) -- (B) -- (R);
        \draw (M) -- (N) -- (O) -- (A);
        \draw (N) -- (L) -- (O);
        \draw (A) -- (C) -- (B);
        \draw (A) -- (D) -- (B);
        \draw [densely dotted] (O) -- (D) (D) -- (S);
    }{
        \node[transparent] (C) at (1, 1.5) {};
        \node[blue] (D) at (1, -1.0) {};
        \node[orange] (R) at (3, 0) [label=above right:$z$] {};
        \node (S) at (2.5, -1) {};
        \node[orange] (M) at (-2, 0) {};
        \node[orange] (N) at (-1.5, -1) {};
        \node[blue] (O) at (-0.5, -1) {};
        \draw (S) -- (R) {};
        \draw (M) --+ (-0.5, 0);
        \draw (R) --+ (0.5, 0);
        \draw (M) -- (N) -- (O);
        \draw [densely dotted] (O) -- (D) (D) -- (S);
    }
    \caption{OR with no interior 4-pairs The dashed lines represent a path of vertices all adjacent to $x$ (or $y$).
    The blue vertices each have an LR fused to them (not drawn here). The orange vertices mark locations where a low-degree problem might arise.}
    \label{fig:ornofourpairs}
\end{figure}

\subsubsection{Conclusion}
Now $G$ has no OR attachments. If $G$ is a single triangle, then $\Phi = 3.5$ and $s=1$.
Otherwise, and $G$ is a 3-connected near-triangulation.
This is great for deleting boundary vertices:

\begin{definition}[Interior graph] \label{def:interiorgraph}
    Let $G$ be a skeletal triangulation. The \emph{interior graph} $\Interior(G)$
    is the graph induced by all interior vertices in $G$.
    An \emph{interior leaf} is a vertex $u \in \Interior(G)$ with $\deg_{\Interior(G)}(u) = 1$.
\end{definition}
\begin{lemma}
    If $G$ is a 3-connected near-triangulation $G$, then $\Interior(G)$ is connected.
\end{lemma}
\begin{proof}
    Let $s, t \in \Interior(G)$ be arbitrary.
    By Menger's theorem, there are 3 vertex-disjoint s-t-paths.
    If all of these contain a boundary vertex, then adding a vertex connected to all boundary vertices to the unbounded face creates a planar $K_{3, 3}$-subdivision, contradiction.
    Hence, at least one of the s-t-paths lies fully in $\Interior(G)$.
\end{proof}
\begin{corollary}
    Let $G$ be a 3-connected near-triangulation and let $B \subseteq G$ be any set of boundary vertices.
    Then $G - B$ is connected.
\end{corollary}
\begin{proof}
    By the Lemma, $\Interior(G)$ is connected.
    Every boundary vertex of $G$ has a neighbor in $\Interior(G)$.
\end{proof}

\subsection{Deletable boundary edges} \label{sect:cutable}
\subsubsection{Sporadic examples}
Suppose deleting the (boundary) edge $\{v, w\}$ yields a sporadic example $H$.
In other words, $G$ is the result of adding an edge to a sporadic example in a way
that creates a facial triangle. 

$H$ contains at least four boundary vertices, so $H$ is not the octahedron.
If $H$ is the 3-bifan, then $n=6$ and either $G$ is the octahedron, or $G$
has two vertices of degree $5$ and hence $s = 1$. See Figure~\ref{fig:deledgebifan}.
Finally, if $H$ is the special 4343434-heptagon, then $G$ has six boundary vertices and $n=10$.
There are seven ways of adding an edge to the special 4343434-heptagon and one
can check that $s=2$ in each case. See Figure~\ref{fig:deledgeheptagon}.

\begin{figure}
    \begin{center}
    \begin{tikzpicture}[scale=0.6]
        \node (A) at (0, 0) {};
        \node (B) at (3, 0) {};
        \node (C) at (3, 3) {};
        \node (D) at (0, 3) {};
        \node (x) at (1, 1) {};
        \node (y) at (2, 2) {};
        \draw (A) -- (B) -- (C) -- (D) -- (A);
        \draw (A) -- (x) -- (y) -- (C);
        \draw (B) -- (x) -- (D) -- (y) -- (B);
        \begin{scope}[on background layer]
            \draw[blue] (A.center) to[out=135, in=135, distance=100] (C.center);
        \end{scope}
    \end{tikzpicture}%
    \hspace{1.5cm}%
    \begin{tikzpicture}[scale=0.6]
        \node (A) at (0, 0) {};
        \node (B) at (3, 0) {};
        \node (C) at (3, 3) {};
        \node (D) at (0, 3) {};
        \node (x) at (1, 1) {};
        \node (y) at (2, 2) {};
        \draw (A) -- (B) -- (C) -- (D) -- (A);
        \draw (A) -- (x) -- (y) -- (C);
        \draw (B) -- (x) -- (D) -- (y) -- (B);
        \begin{scope}[on background layer]
            \draw[blue] (B.center) to[out=45, in=45, distance=100] (D.center);
        \end{scope}
    \end{tikzpicture}
    \end{center}

    \caption{The two ways of adding one edge to the 3-bifan. On the left: the octahedron. On the right: a graph with $s=1$.}
    \label{fig:deledgebifan}
\end{figure}
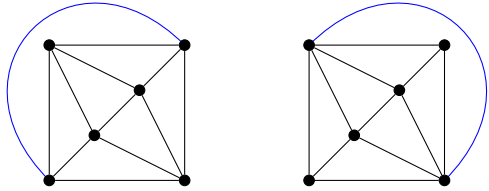

\begin{figure}
    \begin{center}
        \begin{tikzpicture}[scale=0.4]
            \begin{scope}[shift={(0, 0)}]
                \node (A) at (-1, 2) [red,rectangle] {};
                \node (B) at (-4, 0) {};
                \node (C) at (-2, -2) {};
                \node (D) at (0, -2) [red,rectangle] {};
                \node (E) at (2, -2) {};
                \node (F) at (4, 0) [orange] {};
                \node (G) at (1, 2) {};
                \node (X) at (-2, 0) {};
                \node (Y) at (0, 0) {};
                \node (Z) at (2, 0) {};
                \draw (A) -- (B) -- (C) -- (D) -- (E) -- (F) -- (G) -- (A) {};
                \draw (B) -- (X) -- (Y) -- (Z) -- (F) {};
                \draw (C) -- (X) -- (A) -- (Y) -- (G) -- (Z) -- (E) -- (Y) -- (C);
                \draw (D) -- (Y);
                \begin{scope}[on background layer]
                    \draw[blue] (A.center) to[out=45, in=90, distance=80] (F.center) {};
                    \draw[blue] (D.center) to[out=-45, in=-90, distance=90] (F.center) {};
                \end{scope}
            \end{scope}
            \begin{scope}[shift={(12, 0)}]
                \node (A) at (-1, 2) {};
                \node (B) at (-4, 0) [orange] {};
                \node (C) at (-2, -2) {};
                \node (D) at (0, -2) [red,rectangle] {};
                \node (E) at (2, -2) {};
                \node (F) at (4, 0) {};
                \node (G) at (1, 2) [red,rectangle] {};
                \node (X) at (-2, 0) {};
                \node (Y) at (0, 0) {};
                \node (Z) at (2, 0) {};
                \draw (A) -- (B) -- (C) -- (D) -- (E) -- (F) -- (G) -- (A) {};
                \draw (B) -- (X) -- (Y) -- (Z) -- (F) {};
                \draw (C) -- (X) -- (A) -- (Y) -- (G) -- (Z) -- (E) -- (Y) -- (C);
                \draw (D) -- (Y);
                \begin{scope}[on background layer]
                    \draw[blue] (G.center) to[out=135, in=90, distance=80] (B.center) {};
                    \draw[blue] (D.center) to[out=-135, in=-90, distance=90] (B.center) {};
                \end{scope}
            \end{scope}
            \begin{scope}[shift={(0, 8)}]
                \node (A) at (-1, 2) [red,rectangle] {};
                \node (B) at (-4, 0) {};
                \node (C) at (-2, -2) [orange] {};
                \node (D) at (0, -2) {};
                \node (E) at (2, -2) [red,rectangle] {};
                \node (F) at (4, 0) {};
                \node (G) at (1, 2) {};
                \node (X) at (-2, 0) {};
                \node (Y) at (0, 0) {};
                \node (Z) at (2, 0) {};
                \draw (A) -- (B) -- (C) -- (D) -- (E) -- (F) -- (G) -- (A) {};
                \draw (B) -- (X) -- (Y) -- (Z) -- (F) {};
                \draw (C) -- (X) -- (A) -- (Y) -- (G) -- (Z) -- (E) -- (Y) -- (C);
                \draw (D) -- (Y);
                \begin{scope}[on background layer]
                    \draw[blue] (A.center) to[out=180, in=-180, distance=120] (C.center) {};
                    \draw[blue] (E.center) to[out=-135, in=-45, distance=60] (C.center) {};
                \end{scope}
            \end{scope}
            \begin{scope}[shift={(12, 8)}]
                \node (A) at (-1, 2) {};
                \node (B) at (-4, 0) {};
                \node (C) at (-2, -2) [red,rectangle] {};
                \node (D) at (0, -2) {};
                \node (E) at (2, -2) [orange] {};
                \node (F) at (4, 0) {};
                \node (G) at (1, 2) [red,rectangle] {};
                \node (X) at (-2, 0) {};
                \node (Y) at (0, 0) {};
                \node (Z) at (2, 0) {};
                \draw (A) -- (B) -- (C) -- (D) -- (E) -- (F) -- (G) -- (A) {};
                \draw (B) -- (X) -- (Y) -- (Z) -- (F) {};
                \draw (C) -- (X) -- (A) -- (Y) -- (G) -- (Z) -- (E) -- (Y) -- (C);
                \draw (D) -- (Y);
                \begin{scope}[on background layer]
                    \draw[blue] (G.center) to[out=0, in=0, distance=120] (E.center) {};
                \end{scope}
            \end{scope}
        \end{tikzpicture}
    \end{center}

    \caption{The special 4343434-heptagon with four sets of size $2$ (red) that each dominate all vertices except one (orange).
    In each case, adding any one of the blue edges makes the red set dominating.}
    \label{fig:deledgeheptagon}
\end{figure}
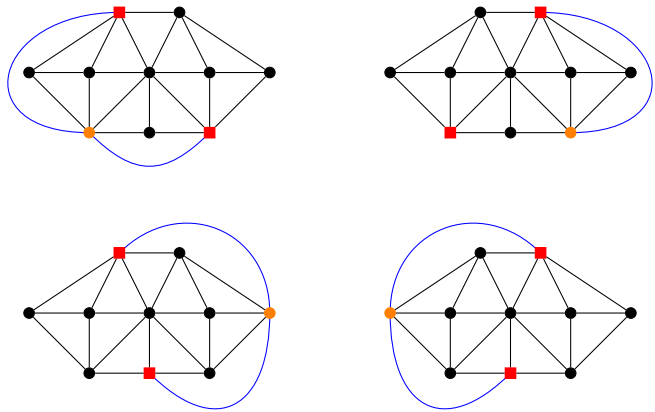

\subsubsection{No problems}\label{sect:noprob}
Suppose deleting the (boundary) edge $\{v, w\}$ does not create any low-degree problems.
\begin{case}
\item[Construction] Delete the edge $\{v, w\}$.
\item[Domination] $H \subseteq G$ is a spanning subgraph. $\Delta s \le 0$.
\item[Penalty] By assumption, no low-degree problems are created. $-\Delta \Phi \le 0$.
    By the previous case, $H$ is not a sporadic example.
\item[Smaller] $H$ has fewer interior vertices. 
\end{case}

\subsubsection{Interior degree-3 vertex.} \label{sect:surfacekfour}
Suppose there is an interior degree-3 vertex $p$ such that $v, p, w$ is a facial triangle.
By Case~\ref{sect:noprob}, deleting $\{v, w\}$ creates at least one low-degree problem.

If the problem involves $p$, then $v, p$ (or $p, w$) is a 3-pair into some bad 5-wheel, centered around some vertex $q$.
Then, $\{u, w\}$ is an edge, hence the polygon of $G$ is a triangle.
See Figure~\ref{fig:boundarydegthree}.
If $\deg(u) = 3$, then $G$ is a 5-wheel with one extra edge,
with $\Phi = 5$ and $s = 1$. Otherwise, deleting the boundary edge $\{u, v\}$
does not create any low-degree problems and Case~\ref{sect:noprob} applies.

Otherwise, if the problem is an ear, say at $v$, then $\deg_G(v) = 3$ so $G = K_4$.
If the problem is a 3-pair, say $u, v$, in some bad 5-wheel. Then $G$ is the 5-wheel with one extra edge.

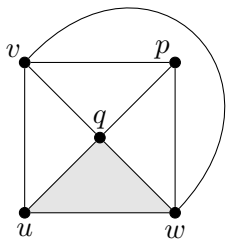
\begin{figure}
    \begin{center} \begin{tikzpicture}
        \begin{scope}[rotate=-90]
            \node (A) at (0, 0) [label=above left:$v$] {};
            \node (B) at (2, 0) [label=below:$u$] {};
            \node (C) at (2, 2) [label=below:$w$] {};
            \node (D) at (0, 2) [label=above left:$p$] {};
            \node (O) at (1, 1) [label=above:$q$] {};
            \draw (A) -- (B) -- (C) -- (D) -- (A);
            \draw (A) -- (O) -- (C) (B) -- (O) -- (D);
            \draw (A) to[out=135, in=135, distance=70] (C);
            \begin{scope}[on background layer]
                \path [fill=gray, opacity=0.2] (C.center) to (B.center) to (O.center) to cycle;
            \end{scope}
        \end{scope}
    \end{tikzpicture} \end{center}
    \caption{Boundary edge with interior degree-3 vertex. Any additional vertices lie in the shaded area.}
    \label{fig:boundarydegthree}
\end{figure}

\paragraph{Observation} A direct consequence this case is the following: If $p$ is an interior degree-3 vertex adjacent
to some boundary vertex $v$, then $p$ is not adjacent to $u$ (nor $w$). Moreover, as $G$ is 3-connected, $v$ is the only boundary vertex adjacent to $p$.
In addition, this implies $\deg(v) \ge 5$.

\subsubsection{Conclusion} \label{sect:alldegpatterns}
Now, deleting any boundary edge $\{v, w\}$ creates at least one low-degree problem,
which involves only boundary vertices in $G$.
Let $H$ be resulting graph.
As $G$ is 3-connected, $H$ is 2-connected, so any newly created low-degree problem is an ear tip or a bad 5-wheel,
and in both cases, a 2-cut is created.
More precisely, at least one of the following is true.
\begin{itemize}
    \item $\deg_G(v) = 3$ and $v \in H$ is an ear tip.
    \item $\deg_G(w) = 3$ and $w \in H$ is an ear tip.
    \item $\deg_G(u) = 3$, $\deg_G(v) = 4$ and $u, v \in H$ is a 3-pair in a bad 5-wheel.
    \item $\deg_G(w) = 4$, $\deg_G(x) = 3$ and $w, x \in H$ is a 3-pair in a bad 5-wheel.
\end{itemize}

Therefore, in $G$, there is at least one boundary vertex of degree 3, and, in between two boundary vertices of degree $3$,
the degrees of boundary vertices form one of the following patterns:
\begin{itemize}
    \item $3 4 5^{+} 4 3$, i.e. $\deg(u) = 3$, $\deg(v) = 4$, $\deg(w) \ge 5$, $\deg(x) = 4$ and $\deg(y) = 3$.
    \item $3 4 5^{+} 3$ or $3 5^{+} 4 3$.
    \item $3 5^{+} 3$.
    \item $3 4 4 4 3$, $3 4 4 3$, or $3 4 3$.
    \item $3 3$.
\end{itemize}

\subsection{Triangular boundary}
Suppose the polygon is a triangle, i.e. $x = u$.
If at least two of $u, v, w$ have degree 3, then $G = K_4$ with $\Phi = 4$ and $s = 1$.
Otherwise, WLOG let $\deg(u) = 3$ and $\deg(v), \deg(w) \ge 4$.
By Case~\ref{sect:cutable}, deleting the edge $\{v, w\}$ creates a 3-pair $u, v$ (or $w, u$) in a bad 5-wheel,
hence $\deg_G(v) = 4$. See Figure~\ref{fig:triangboundary}.
\begin{case}
\item[Construction] Delete $v$.
\item[Domination] A dominating set in $H$ contains a vertex that dominates $u$. As $N[u] \subseteq N[v]$,
    that vertex also dominates $v$. $\Delta s \le 0$.
\item[Penalty] Deleting $v$ decreases $\Phi$ by 1. As $\deg(v) = 4$, this creates at most two new low-degree problems, increasing $\Phi$ by $\le 1$. $-\Delta \Phi \le 0$.
\end{case}
\begin{figure}
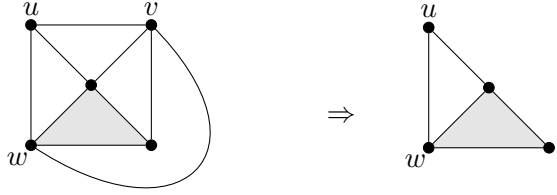

    \beforeafter[0.8]{
        \node (A) at (0, 0) [label=below left:$w$] {};
        \node (B) at (2, 0) {};
        \node (C) at (2, 2) [label=above:$v$] {};
        \node (D) at (0, 2) [label=above:$u$] {};
        \node (E) at (1, 1) {};
        \node (Z) at (0, -1.5) [transparent] {};
        \draw (A) -- (B) -- (C) -- (D) -- (A);
        \draw (A) -- (E) -- (C);
        \draw (B) -- (E) -- (D);
        \draw (A.center) to[out=-34, in=-45, distance=3cm] (C.center);
        \begin{scope}[on background layer]
            \path [fill=gray, opacity=0.2] (A.center) to (B.center) to (E.center) to cycle;
        \end{scope}
    }{
        \node (A) at (0, 0) [label=below left:$w$] {};
        \node (B) at (2, 0) {};
        \node (D) at (0, 2) [label=above:$u$] {};
        \node (E) at (1, 1) {};
        \node (Z) at (0, -1.5) [transparent] {};
        \draw (A) -- (B)  (D) -- (A);
        \draw (A) -- (E) ;
        \draw (B) -- (E) -- (D);
        \begin{scope}[on background layer]
            \path [fill=gray, opacity=0.2] (A.center) to (B.center) to (E.center) to cycle;
        \end{scope}
    }
    \caption{Triangular boundary. All additional vertices lie in the shaded area.}
    \label{fig:triangboundary}
\end{figure}

\paragraph{Conclusion} The polygon is not a triangle, hence there are at least four boundary vertices.

\subsection{Boundary vertex with interior octahedron} \label{sect:boundaryocta}
Suppose deleting $w$ creates a 3-pair $p, q$ in a bad 5-wheel with $p, q \notin \{v, x\}$.
Then, $p, q$ are interior vertices and $N[p, q]$ is the octahedron. Let $t \in N[p, q]$ be the vertex antipodal to $w$.
See Figure~\ref{fig:threeproblems}. If $v \in N[p, q]$ (or $x \in N[p, q]$), then deleting $\{v, w\}$ (or $\{w, x\}$) does not create a low-degree problem,
contradiction. In particular, $\deg(w) \ge 5$ and $\deg(v), \deg(x) \le 4$.

\begin{case}
\item[Construction] Delete $v, w, x, p, q$. Fuse a small LR to every remaining neighbor of $w$. Suppose we fuse $k$ LRs this way.
\item[Domination] If $S$ is a minimum dominating set in $H$ and $L \subseteq S$ are the $k$ vertices in the fused LRs, then $S \cup \{w\} \setminus L$
    is a dominating set in $G$. $\Delta s \le 1-k$.
\item[Penalty] The deletions decrease $\Phi$ by $5$. Fusing the LRs increases $\Phi$ by $3.5 k$.
    The deletions may create up to three low-degree problems, involving $u, y$ and $t$ respectively\footnote{One can show that it is actually $u$ and $y$, and not not their formerly-interior neighbor, but it is only important that low-degree problems are never adjacent.}
    This increases $\Phi$ by $\le 1.5$.
    Overall, $-\Delta \Phi \le -5 + 3.5 k + 1.5 = -3.5 (1-k)$.
\item[No Leaves] The only vertices than could end up as leaves are $u$ and $y$. If $u \ne y$,
    then $u$ and $y$ each loose only one neighbor and hence cannot end up as leaves. Therefore, $u = y$ and $\deg(u) = 3$.
    If $\deg(v) = 4$ and $\deg(x) = 3$, then deleting $\{v, w\}$ does not create a bad 5-wheel, contradicting Case~\ref{sect:alldegpatterns}.
    Finally, if $\deg(v) = \deg(x) = 4$, then deleting $\{v, w\}$ or $\{w, x\}$ cannot both create a bad 5-wheel,
    as that would force $\deg(w) = 3$.
\end{case}

\begin{figure}
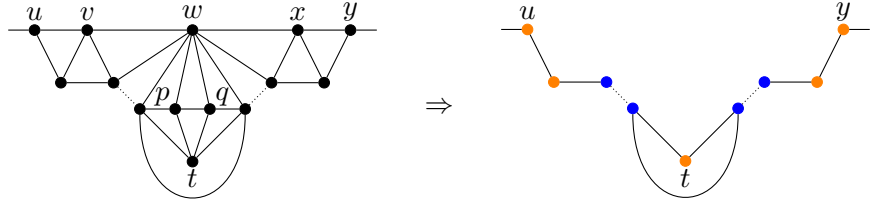

    \beforeafter[0.7]{
        \node (L) at (-3, 0) [label=above:$u$] {};
        \node (R) at (3, 0) [label=above:$y$] {};
        \node (A) at (0, 0) [label=above:$w$] {};
        \node (B) at (2, 0) [label=above:$x$] {};
        \node (C) at (1.5, -1) {};
        \node (D) at (2.5, -1) {};
        \node (U) at (-2, 0) [label=above:$v$] {};
        \node (V) at (-1.5, -1) {};
        \node (W) at (-2.5, -1) {};
        \node (p) at (-1, -1.5) {};
        \node (q) at (-0.33, -1.5) [label=above left:$p$] {};
        \node (r) at (0.33, -1.5) [label=above right:$q$] {};
        \node (s) at (1, -1.5) {};
        \node (I) at (0, -2.5) [label=below:$t$] {};
        \draw (A) -- (p) -- (I) -- (q) -- (A) -- (r) -- (I) -- (s) -- (A);
        \draw (p) -- (q) -- (r) -- (s);
        \draw (A) -- (B);
        \draw (A) -- (C) -- (D) -- (R);
        \draw (C) -- (B) -- (D);
        \draw (U) -- (V) -- (W) -- (L);
        \draw (V) -- (A) -- (U) -- (W);
        \draw (L) --+ (-0.5, 0);
        \draw (R) --+ (0.5, 0);
        \draw (L) -- (U) (B) -- (R);
        \draw [densely dotted] (s) -- (C) (V) -- (p);
        \draw (p) to [in=-90, out=-90, distance=60] (s);
    }{
        \node (L) at (-3, 0) [orange, label=above:$u$] {};
        \node (R) at (3, 0) [orange, label=above:$y$] {};
        \node (A) at (0, 0) [transparent] {};
        \node (B) at (2, 0) [transparent] {};
        \node (C) at (1.5, -1) [blue] {};
        \node (D) at (2.5, -1) [orange] {};
        \node (U) at (-2, 0) [transparent] {};
        \node (V) at (-1.5, -1) [blue] {};
        \node (W) at (-2.5, -1) [orange] {};
        \node (p) at (-1, -1.5) [blue] {};
        \node (q) at (-0.33, -1.5) [transparent] {};
        \node (r) at (0.33, -1.5) [transparent] {};
        \node (s) at (1, -1.5) [blue] {};
        \node (I) at (0, -2.5) [orange, label=below:$t$] {};
        \draw (p) -- (I) (I) -- (s);
        \draw (C) -- (D) -- (R);
        \draw (V) -- (W) -- (L);
        \draw (L) --+ (-0.5, 0);
        \draw (R) --+ (0.5, 0);
        \draw [densely dotted] (s) -- (C) (V) -- (p);
        \draw (p) to [in=-90, out=-90, distance=60] (s);
    }
    \caption{Boundary vertex with interior octahedron.}
    \label{fig:threeproblems}
\end{figure}

\subsubsection{Conclusion}
Now, deleting a boundary vertex never creates a 3-pair of former interior vertices in a bad 5-wheel.

\subsection{Boundary vertex with deletable \boldmath$K_4$} \label{sect:boundarykfour}
Suppose deleting $w$ creates an ear tip $q$ with $q \notin \{v, x\}$.
Then $\deg(w) \ge 5$, hence $\deg(v), \deg(x) \le 4$.
Suppose moreover that $\deg(v) = 3$ and that deleting $v$ does not
create a low-degree problem involving $u$. See Figure~\ref{fig:boundarykfour}.

\begin{case}
\item[Construction] Delete $v, w, x, q$. Fuse a small LR to every remaining neighbor of $w$. Suppose we fuse $k$ LRs this way.
\item[Domination] If $S$ is a minimum dominating set in $H$ and $L \subseteq S$ are the $k$ vertices in the fused LRs, then $S \cup \{w\} \setminus L$ is a dominating set in $G$.
    $\Delta s = 1-k$.
\item[Penalty] The deletions decrease $\Phi$ by $4$. Fusing the LRs increases $\Phi$ by $3.5 k$. The deletions may create a single low-degree problem, involving $y$.
    This increases $\Phi$ by $0.5$. Overall, $-\Delta \Phi \le -4 + 3.5k + 0.5 = -3.5(1-k)$.
\end{case}
\begin{figure}
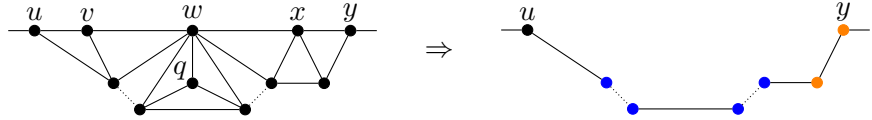

    \beforeafter[0.7]{
        \node (L) at (-3, 0) [label=above:$u$] {};
        \node (R) at (3, 0) [label=above:$y$] {};
        \node (A) at (0, 0) [label=above:$w$] {};
        \node (B) at (2, 0) [label=above:$x$] {};
        \node (C) at (1.5, -1) {};
        \node (D) at (2.5, -1) {};
        \node (U) at (-2, 0) [label=above:$v$] {};
        \node (V) at (-1.5, -1) {};
        \node (W) at (-2.5, -1) [transparent] {};
        \node (p) at (-1, -1.5) {};
        \node (q) at (0, -1) [label=above left:$q$] {};
        \node (s) at (1, -1.5) {};
        \draw (A) -- (p)  (q) -- (A)   (s) -- (A);
        \draw (p) -- (q) -- (s);
        \draw (A) -- (B);
        \draw (A) -- (C) -- (D) -- (R);
        \draw (C) -- (B) -- (D);
        \draw (U) -- (V) -- (L);
        \draw (V) -- (A) -- (U);
        \draw (L) --+ (-0.5, 0);
        \draw (R) --+ (0.5, 0);
        \draw (L) -- (U) (B) -- (R);
        \draw [densely dotted] (s) -- (C) (V) -- (p);
        \draw (p) -- (s);
    }{
        \node (L) at (-3, 0) [label=above:$u$] {};
        \node (R) at (3, 0) [label=above:$y$, orange] {};
        \node (A) at (0, 0) [transparent] {};
        \node (B) at (2, 0) [transparent] {};
        \node (C) at (1.5, -1) [blue] {};
        \node (D) at (2.5, -1) [orange] {};
        \node (U) at (-2, 0) [transparent] {};
        \node (V) at (-1.5, -1) [blue] {};
        \node (W) at (-2.5, -1) [transparent] {};
        \node (p) at (-1, -1.5) [blue] {};
        \node (q) at (0, -1) [transparent] {};
        \node (s) at (1, -1.5) [blue] {};
        \draw (C) -- (D) -- (R);
        \draw (V) -- (L);
        \draw (L) --+ (-0.5, 0);
        \draw (R) --+ (0.5, 0);
        \draw [densely dotted] (s) -- (C) (V) -- (p);
        \draw (p) -- (s);
    }
    \caption{Boundary vertex $w$ with interior $K_4$.}
    \label{fig:boundarykfour}
\end{figure}

\subsubsection{Conclusion}
If $w$ is a boundary vertex with an interior degree-3 neighbor, then
deleting either $v$ (or $x$) creates a low-degree problem involving $u$ (or $y$).
In particular, $\deg_G(u), \deg_G(y) \le 4$, hence $u$ and $y$ have no interior degree-3 neighbor.

\subsection{Interior vertex adjacent to two non-consecutive degree-3 boundary vertices} \label{sect:doublethree}
Suppose $\deg(v) \ge 4$, $\deg(w) = 3$ and $\deg(x) \ge 4$.
Let $p$ be the interior vertex adjacent to $v, w, x$.
Suppose $p$ is adjacent to another boundary vertex $s \ne w$ with $\deg(s) = 3$.
Let $r, t$ be the boundary neighbors of $s$. See Figure~\ref{fig:twosidethrees}.
If both $v$ and $x$ have an interior degree-3 neighbor, then Case~\ref{sect:boundarykfour} applies,
as deleting $w$ does not create any low-degree problems.
Hence, WLOG assume that $x$ does not have an interior degree-3 neighbor.
\begin{case}
\item[Construction] Delete $x$. Delete $w$ and $s$. Force $p$ by attaching a B to $v, p$. Fuse a small LRs to $r$ and $t$.
\item[Domination] If $S \subseteq H$ is a minimum neat dominating set and $L \subseteq S$ are the two vertices in the LRs,
    then $S \setminus L$ contains $p$, which dominates $H$. $\Delta s \le -2$.
\item[Penalty] Deleting three vertices decreases $\Phi$ by $3$. Attaching the B increases $\Phi$ by 2.5. Fusing two LRs
    increases $\Phi$ by $7$. The deletions may create a low-degree problem involving $y$ but no further ones, increasing $\Phi$ by $\le 0.5$.
    Overall, $-\Delta \Phi \le -3+2.5+7+0.5 = 7$.
\end{case}
\begin{figure}
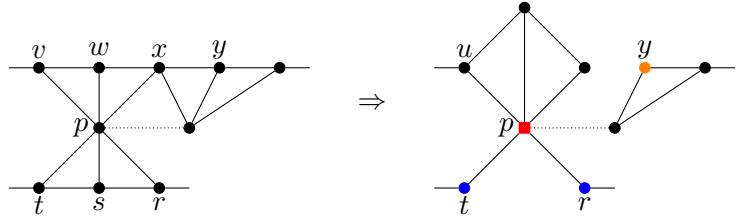

    \beforeafter[0.8]{
        \node (A) at (-1, 0) [label=above:$v$] {};
        \node (B) at (0, 0) [label=above:$w$] {};
        \node (AA) at (-1, 1) [transparent] {};
        \node (C) at (1, 0) [label=above:$x$] {};
        \node (D) at (2, 0) [label=above:$y$] {};
        \node (R) at (3, 0) {};
        \node (I) at (0, -1) [label=left:$p$] {};
        \node (X) at (1.5, -1) {};
        \node (x) at (-1, -2) [label=below:$t$] {};
        \node (y) at (0, -2) [label=below:$s$] {};
        \node (z) at (1, -2) [label=below:$r$] {};
        \draw (A) -- (B) -- (C) -- (D) -- (R);
        \draw (I) -- (A) (B) -- (I) -- (C);
        \draw (C) -- (X);
        \draw (X) -- (R);
        \draw (X) -- (D);
        \draw [densely dotted] (I) -- (X);
        \draw (x) -- (I) (y) -- (I) (z) -- (I);
        \draw (x) -- (y) -- (z);
        \draw (x) --+(-0.5, 0) (z) --+ (0.5, 0);
        \draw (A) --+ (-0.5, 0);
        \draw (R) --+ (0.5, 0);
    }{
        \node (L) at (-1, 0) [label=above:$u$] {};
        \node (A) at (0, 1) {};
        \node (B) at (1, 0) {};
        \node (D) at (2, 0) [orange, label=above:$y$] {};
        \node (R) at (3, 0) {};
        \node (I) at (-0, -1) [red,rectangle, label=left:$p$] {};
        \node (X) at (1.5, -1) {};
        \node (x) at (-1, -2) [blue, label=below:$t$] {};
        \node (y) at (0, -2) [transparent] {};
        \node (z) at (1, -2) [blue, label=below:$r$] {};
        \draw (L) -- (A) -- (B)  (D) -- (R);
        \draw (L) -- (I) -- (A) (B) -- (I);
        \draw (X) -- (R);
        \draw (X) -- (D);
        \draw [densely dotted] (I) -- (X);
        \draw (x) -- (I) (z) -- (I);
        \draw (x) --+(-0.5, 0) (z) --+ (0.5, 0);
        \draw (L) --+ (-0.5, 0);
        \draw (R) --+ (0.5, 0);
    }
    \caption{Degree-3 boundary vertices on two sides.}
    \label{fig:twosidethrees}
\end{figure}

\subsubsection{Conclusion}
Now any two non-consecutive degree-3 boundary vertices are adjacent to distinct interior vertices.

\subsection{Consecutive degree-3 boundary vertices}
Suppose $\deg(u) \ge 4$, $\deg(v) = \deg(w) = 3$ and $\deg(x) \ge 4$.
Let $p$ be the interior vertex adjacent to $u, v, w, x$.

\subsubsection{No interior degree-3 neighbor}

Suppose that $x$ (or $u$) has no interior degree-3 neighbor. See Figure~\ref{fig:consecthrees}.
\begin{case}
\item[Construction] Delete $x$. This turns $v, w$ into a B with red vertex $p$.
\item[Domination] A minimum neat dominating set in $H$ contains $p$, which dominates $x$. $\Delta s \le 0$.
\item[Penalty] Deleting one vertex decreases $\Phi$ by $1$.
    Deleting $x$ creates at most two low-degree problems. This increases $\Phi$ by $\le 1$.
    Overall, $-\Delta \Phi \le -1+1 = 0$.
\end{case}
\begin{figure}
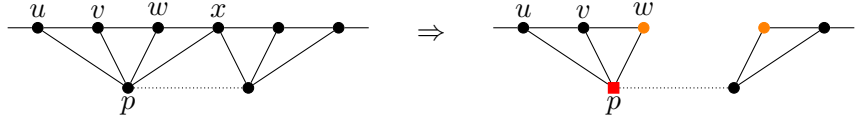

    \beforeafter[0.8]{
        \node (L) at (-2, 0) [label=above:$u$] {};
        \node (A) at (-1, 0) [label=above:$v$] {};
        \node (B) at (0, 0) [label=above:$w$] {};
        \node (AA) at (-1, 1) [transparent] {};
        \node (C) at (1, 0) [label=above:$x$] {};
        \node (D) at (2, 0) {};
        \node (R) at (3, 0) {};
        \node (I) at (-0.5, -1) [label=below:$p$] {};
        \node (X) at (1.5, -1) {};
        \draw (L) -- (A) -- (B) -- (C) -- (D) -- (R);
        \draw (L) -- (I) -- (A) (B) -- (I) -- (C);
        \draw (C) -- (X);
        \draw (X) -- (R);
        \draw (X) -- (D);
        \draw [densely dotted] (I) -- (X);
        \draw (L) --+ (-0.5, 0);
        \draw (R) --+ (0.5, 0);
    }{
        \node (L) at (-2, 0) [label=above:$u$] {};
        \node (A) at (-1, 0) [label=above:$v$] {};
        \node (B) at (0, 0) [label=above:$w$, orange] {};
        \node (AA) at (-1, 1) [transparent] {};
        \node (D) at (2, 0) [orange] {};
        \node (R) at (3, 0) {};
        \node (I) at (-0.5, -1) [label=below:$p$, red,rectangle] {};
        \node (X) at (1.5, -1) {};
        \draw (L) -- (A) -- (B)  (D) -- (R);
        \draw (L) -- (I) -- (A) (B) -- (I);
        \draw (X) -- (R);
        \draw (X) -- (D);
        \draw [densely dotted] (I) -- (X);
        \draw (L) --+ (-0.5, 0);
        \draw (R) --+ (0.5, 0);
    }
    \caption{Consecutive degree-3 boundary vertices with no interior $K_4$ at $x$.}
    \label{fig:consecthrees}
\end{figure}

\subsubsection{Both sides have an interior $K_4$.}
Suppose that both $u$ and $x$ have an interior degree-3 neighbor,
then $\deg(u), \deg(x) \ge 5$, and $t \ne y$ due to Case~\ref{sect:boundarykfour}.
See Figure~\ref{fig:doublekfour}.
\begin{case}
\item[Construction] For each of $u, x$, delete one interior degree-3 neighbor. Delete $t, v, w, y$.
    Force $u$ and $x$ by attaching a B with red vertex $u$ and an A with red vertex $x$. Fuse a small LR to $p$.
\item[Domination] If $S$ is a minimum neat dominating set in $H$ and $L$ are the vertices in the small LR, then $S \setminus L$
    is a dominating set in $G$, as it contains both $u$ and $x$. $\Delta s \le -1$.
\item[Penalty] Deleting six vertices decreases $\Phi$ by $6$. Attaching the A and B increases $\Phi$ by $5$.
    Fusing a small LR increases $\Phi$ by $3.5$. If $s \ne y$, then the deletion may
    create two low degree problems, involving $s$ and $z$ respectively, increasing $\Phi$ by $1$. If $s = z$,
    then Case~\ref{sect:boundarykfour} forces $\deg(y) = \deg(t) = 3$ and the deletions create at most one low-degree problem.
    No other low-degree problems
    are created, not even degree-2 cut vertices. Overall, $-\Delta \Phi \le -6 + 5 + 3.5 + 1  = 3.5$.
\item[No leaves] Me might create a leaf, but only if $s = z$, $\deg(s) = 3$ and $\deg(y) = \deg(t) = 4$. We treat this special case up next.
\end{case}

\begin{figure}
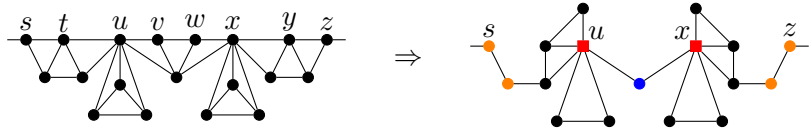

    \beforeafter[0.5]{
        \node (LL) at (-4, 0) [label=above:$s$] {};
        \node (L) at (-3, 0) [label=above:$t$] {};
        \node (A) at (-1.5, 0) [label=above:$u$] {};
        \node (B) at (-0.5, 0) [label=above:$v$] {};
        \node (D) at (0.5, 0) [label=above:$w$] {};
        \node (E) at (1.5, 0) [label=above:$x$] {};
        \node (R) at (3, 0) [label=above:$y$] {};
        \node (RR) at (4, 0) [label=above:$z$] {};
        \node (y) at (0, -1) {};
        \node (p) at (-2.2, -2) {};
        \node (q) at (-0.8, -2) {};
        \node (r) at (-1.5, -1.2) {};
        \node (s) at (0.8, -2) {};
        \node (t) at (2.2, -2) {};
        \node (u) at (1.5, -1.2) {};
        \node (a) at (-3.5, -1) {};
        \node (b) at (-2.5, -1) {};
        \node (d) at (2.5, -1) {};
        \node (e) at (3.5, -1) {};
        \draw (LL) -- (L) -- (A) -- (B) --  (D) -- (E) -- (R) -- (RR);
        \draw (B) -- (y) -- (D);
        \draw (A) -- (p) -- (q) -- (A);
        \draw (r) -- (p) (r) -- (q) (r) -- (A);
        \draw (E) -- (s) -- (t) -- (E);
        \draw (u) -- (E) (u) -- (s) (u) -- (t);
        \draw (A) -- (y) -- (E);
        \draw (LL) -- (a) -- (b) -- (A) (a) -- (L) -- (b);
        \draw (E) -- (d) -- (e) -- (RR) (d) -- (R) -- (e);
        \draw (LL) --+ (-0.5, 0) (RR) --+ (0.5, 0);
    }{
        \node (LL) at (-4, 0) [label=above:$s$, orange] {};
        \node (L) at (-3, 0) [transparent] {};
        \node (A) at (-1.5, 0) [label=above right:$u$, red,rectangle] {};
        \node (B) at (-0.5, 0) [transparent] {};
        \node (D) at (0.5, 0) [transparent] {};
        \node (E) at (1.5, 0) [label=above left:$x$, red,rectangle] {};
        \node (R) at (3, 0) [transparent] {};
        \node (RR) at (4, 0) [label=above:$z$, orange] {};
        \node (y) at (0, -1) [blue] {};
        \node (p) at (-2.2, -2) {};
        \node (q) at (-0.8, -2) {};
        \node (r) at (-1.5, -1.2) [transparent] {};
        \node (s) at (0.8, -2) {};
        \node (t) at (2.2, -2) {};
        \node (u) at (1.5, -1.2) [transparent] {};
        \node (a) at (-3.5, -1) [orange] {};
        \node (b) at (-2.5, -1) [] {};
        \node (d) at (2.5, -1) [] {};
        \node (e) at (3.5, -1) [orange] {};
        \node (l) at (-2.5, 0) {};
        \node (m) at (-1.5, 1) {};
        \node (n) at (1.5, 1) {};
        \node (o) at (2.5, 0) {};
        \draw (o) -- (E) -- (n) -- (o) -- (d);
        \draw (l) -- (A) -- (m) -- (l) -- (b);
        \draw (A) -- (p) -- (q) -- (A);
        \draw (E) -- (s) -- (t) -- (E);
        \draw (A) -- (y) -- (E);
        \draw (LL) -- (a) -- (b) -- (A);
        \draw (E) -- (d) -- (e) -- (RR);
        \draw (LL) --+ (-0.5, 0) (RR) --+ (0.5, 0);
    }
    \caption{Consecutive degree-3 boundary vertices with interior $K_4$s on both sides.}
    \label{fig:doublekfour}
\end{figure}

\subsubsection{The remaining special case.}
Suppose that both $u$ and $x$ have an interior degree-3 neighbor,
then $\deg(u), \deg(x) \ge 5$.
Suppose that there are exactly 7 polygon vertices, i.e. $s = z$,
and that $\deg(z) = 3$. Then $\deg(t) = \deg(y) = 4$.
Let $q$ be the interior vertex adjacent to $z$, then $N[q]$ is a 5-wheel. See Figure~\ref{fig:doublekfourspecial}.
\begin{case}
\item[Construction] Delete $y, z, t, q$.
\item[Domination] If $S$ is a minimum dominating set in $H$, then $S \cup {q}$ dominates $G$. $\Delta s = 1$.
\item[Penalty] Deleting four vertices decreases $\Phi$ by $4$. No low-degree problems are created. $-\Delta \Phi = -4$.
\end{case}
\begin{figure}
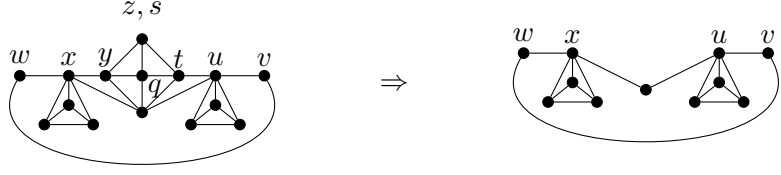

    \beforeafter[0.65]{
        \node (L) at (-2.5, 0) [label=above:$w$] {};
        \node (A) at (-1.5, 0) [label=above:$x$] {};
        \node (B) at (-0.75, 0) [label=above:$y$] {};
        \node (C) at (0, 0.75) [label={above:$z, s$}] {};
        \node (D) at (0.75, 0) [label=above:$t$] {};
        \node (E) at (1.5, 0) [label=above:$u$] {};
        \node (R) at (2.5, 0) [label=above:$v$] {};
        \node (x) at (0, 0) [label=below right:$q$] {};
        \node (y) at (0, -0.75) {};
        \node (p) at (-2, -1) {};
        \node (q) at (-1, -1) {};
        \node (r) at (-1.5, -0.6) {};
        \node (s) at (1, -1) {};
        \node (t) at (2, -1) {};
        \node (u) at (1.5, -0.6) {};
        \draw (L) -- (A) -- (B) -- (C) -- (D) -- (E) -- (R);
        \draw (B) -- (x) -- (D) (C) -- (x) -- (y);
        \draw (B) -- (y) -- (D);
        \draw (A) -- (p) -- (q) -- (A);
        \draw (r) -- (p) (r) -- (q) (r) -- (A);
        \draw (E) -- (s) -- (t) -- (E);
        \draw (u) -- (E) (u) -- (s) (u) -- (t);
        \draw (A) -- (y) -- (E);
        \draw (L) to[out=-120, in=-60, distance=75] (R);
    }{
        \node (L) at (-2.5, 0) [label=above:$w$] {};
        \node (A) at (-1.5, 0) [label=above:$x$] {};
        \node (B) at (-0.75, 0) [transparent] {};
        \node (C) at (0, 0.75) [transparent] {};
        \node (D) at (0.75, 0) [transparent] {};
        \node (E) at (1.5, 0) [label=above:$u$] {};
        \node (R) at (2.5, 0) [label=above:$v$] {};
        \node (x) at (0, 0) [transparent] {};
        \node (y) at (0, -0.75) {};
        \node (p) at (-2, -1) {};
        \node (q) at (-1, -1) {};
        \node (r) at (-1.5, -0.6) {};
        \node (s) at (1, -1) {};
        \node (t) at (2, -1) {};
        \node (u) at (1.5, -0.6) {};
        \draw (L) -- (A) (E) -- (R);
        \draw (A) -- (p) -- (q) -- (A);
        \draw (r) -- (p) (r) -- (q) (r) -- (A);
        \draw (E) -- (s) -- (t) -- (E);
        \draw (u) -- (E) (u) -- (s) (u) -- (t);
        \draw (A) -- (y) -- (E);
        \draw (L) to[out=-120, in=-60, distance=75] (R);

    }

    \caption{Consecutive degree-3 boundary vertices, the remaining special case.}
    \label{fig:doublekfourspecial}
\end{figure}

\subsubsection{Conclusion}
Now every degree-3 boundary vertex is adjacent to a distinct interior vertex.

\subsection{Degree Patterns \boldmath$5^{+}4$ and \boldmath$444$}

\subsubsection{Degree Pattern \boldmath$5^{+} 4$}
Suppose $\deg(v) \ge 5$ and $\deg(w) = 4$, then due to previous cases, $\deg(x) = 3$ and $\deg(y) \ge 4$.
Let $p$ be the interior vertex adjacent to $w, x, y$. See Figure~\ref{fig:degfivefour}.
Deleting $x$ does not create a low-degree problem at $w$, hence $y$ has no interior degree-3 neighbor due to Case~\ref{sect:boundarykfour}.
\begin{case}
\item[Construction] Delete $y$. Delete $x$ and $w$. Force $p$ by attaching a B to $v, p$.
\item[Domination] Any neat dominating set $S \subseteq H$ contains $p$, which dominates $w, x, y$. $\Delta s \le 0$.
\item[Penalty] Deleting three vertices decreases $\Phi$ by $3$. Attaching a B increases $\Phi$ by 2.5.
    Deleting $y$ creates exactly two low-degree problems, namely at $x$
    and $z$.
    Deleting $x, w$ removes the former and does not create a new one, as $\deg(v) \ge 5$.
    Overall, there remains a single low-degree problem, increasing $\Phi$ by 0.5.
    In total, $-\Delta \Phi \le -3+2.5+0.5 = 0$.
\item[Smaller] $H$ has fewer interior vertices.
\end{case}
\begin{figure}
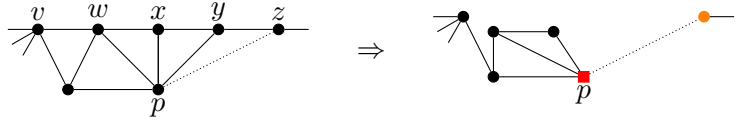

    \beforeafter[0.4]{
        \node (L) at (-2, 0) [label=above:$v$] {};
        \node (A) at (0, 0) [label=above:$w$] {};
        \node (B) at (2, 0) [label=above:$x$] {};
        \node (C) at (4, 0) [label=above:$y$] {};
        \node (R) at (6, 0) [label=above:$z$] {};
        \node (X) at (-1, -2) {};
        \node (Y) at (2, -2) [label=below:$p$] {};
        \draw (L) -- (A) -- (B) -- (C) -- (R);
        \draw (L) -- (X) -- (Y) -- (C);
        \draw (X) -- (A) -- (Y) -- (B);
        \draw (Y) -- (B);
        \draw [densely dotted] (Y) -- (R);
        \draw (L) --+ (-1, 0) (L) --+ (-150:1) (L) --+ (-120:1);
        \draw (R) --+ (1, 0);
    }{
        \node (L) at (-2, 0) {};
        \node (A) at (-1, -0.5) {};
        \node (B) at (1, -0.5) {};
        \node (C) at (4, 0) [transparent] {};
        \node (R) at (6, 0) [orange] {};
        \node (X) at (-1, -2) {};
        \node (Y) at (2, -2) [label=below:$p$, red,rectangle] {};
        \draw (L) -- (X) -- (Y) ;
        \draw [densely dotted] (Y) -- (R);
        \draw (X) -- (A) -- (B) -- (Y) -- (A);
        \draw (L) --+ (-1, 0) (L) --+ (-150:1) (L) --+ (-120:1);
        \draw (R) --+ (1, 0);
    }
    \caption{Degree Pattern $5^{+} 4$.}
    \label{fig:degfivefour}
\end{figure}

\subsubsection{Degree Pattern \boldmath$444$}
Suppose $\deg(u) = \deg(v) = \deg(w) = 4$, then $\deg(x) = 3$ and $\deg(y) \ge 4$.
Let $p$ be the interior vertex adjacent to $w, x, y$.
If $v \ne z$, then argue as in the previous case. It might happen that $\deg_H(v) = 3$,
but this cannot create a second low-degree problem:
Either $u = z$, or $\deg_H(u) = 4$ and $u, v$ is not a 3-pair. See Figure~\ref{fig:triplefour}.
If $v = z$, i.e. the polygon is a square, then the graph looks as in Figure~\ref{fig:squarefours}, but then, deleting the edge $\{u, v\}$
does not create a low degree problem, contradiction.

\begin{figure}
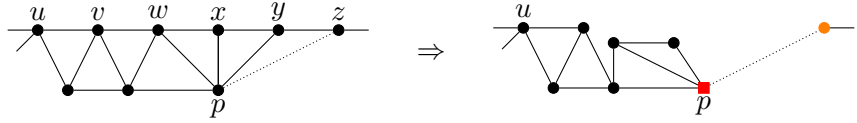

    \beforeafter[0.4]{
        \node (LL) at (-4, 0) [label=above:$u$] {};
        \node (L) at (-2, 0) [label=above:$v$] {};
        \node (I) at (-3, -2) {};
        \node (A) at (0, 0) [label=above:$w$] {};
        \node (B) at (2, 0) [label=above:$x$] {};
        \node (C) at (4, 0) [label=above:$y$] {};
        \node (R) at (6, 0) [label=above:$z$] {};
        \node (X) at (-1, -2) {};
        \node (Y) at (2, -2) [label=below:$p$] {};
        \draw (LL) -- (L) -- (A) -- (B) -- (C) -- (R);
        \draw (L) -- (X) -- (Y) -- (C);
        \draw (X) -- (A) -- (Y) -- (B);
        \draw (Y) -- (B);
        \draw [densely dotted] (Y) -- (R);
        \draw (LL) -- (I) -- (L) (I) -- (X);
        \draw (LL) --+ (-1, 0) (LL) --+ (-135:1);
        \draw (R) --+ (1, 0);
    }{
        \node (LL) at (-4, 0) [label=above:$u$] {};
        \node (L) at (-2, 0) {};
        \node (I) at (-3, -2) {};
        \node (A) at (-1, -0.5) {};
        \node (B) at (1, -0.5) {};
        \node (C) at (4, 0) [transparent] {};
        \node (R) at (6, 0) [orange] {};
        \node (X) at (-1, -2) {};
        \node (Y) at (2, -2) [label=below:$p$, red,rectangle] {};
        \draw (LL) -- (L) -- (X) -- (Y) ;
        \draw [densely dotted] (Y) -- (R);
        \draw (X) -- (A) -- (B) -- (Y) -- (A);
        \draw (LL) -- (I) -- (L) (I) -- (X);
        \draw (LL) --+ (-1, 0) (LL) --+ (-135:1);
        \draw (R) --+ (1, 0);
    }
    \caption{Degree Pattern $444$.}
    \label{fig:triplefour}
\end{figure}

\begin{figure}
    \begin{center} \begin{tikzpicture}[scale=0.6]
        \node (A) at (0, 0) [label=below left:$u$] {};
        \node (B) at (0, 3) [label=above left:$v$] {};
        \node (C) at (3, 3) [label=above right:$w$] {};
        \node (D) at (3, 0) [label=below right:$x$] {};
        \node (X) at (0.5, 1.5) {};
        \node (Y) at (1.5, 2.5) {};
        \node (Z) at (2, 1) {};
        \draw (A) -- (B) -- (C) -- (D) -- (A);
        \draw (A) -- (X) -- (Y) -- (C) -- (Z) -- (A);
        \draw (B) -- (X) -- (Z) -- (Y) -- (B);
        \draw (D) -- (Z);
        \path[fill=gray, opacity=0.4] (X.center) -- (Y.center) -- (Z.center) -- cycle;
    \end{tikzpicture} \end{center}

    \caption{Degrees $4443$ on a square. There might be additional vertices, but they all lie in the shaded triangle.}
    \label{fig:squarefours}
\end{figure}
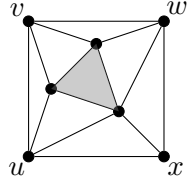

\subsubsection{Conclusion}

We can eliminate some degree patterns from Section~\ref{sect:alldegpatterns}.
The remaining possibilities are:
\begin{itemize}
    \item $3 5^{+} 3$.
    \item $3 4 4 3$.
    \item $3 4 3$.
\end{itemize}

\subsection{Degree \boldmath$3$ boundary vertex with interior problem} \label{sect:threeproblems}
Suppose $\deg(u) \ge 4$, $\deg(v) = 3$ and $\deg(w) \ge 4$.
Let $p$ be the interior vertex adjacent to $v$.
Suppose deleting $v$ and $p$ creates a low-degree problem not involving $u$ or $w$.
By Case~\ref{sect:doublethree}, it is not a degree-2 cut vertex.

\subsubsection{Interior octahedron} \label{sect:threeocta}
Suppose it is a 3-pair $r, s$ in a bad 5-wheel with central vertex $t$.
This 5-wheel together with $p$ forms an octahedron.
See Figure~\ref{fig:threeinteriorocta}.
\begin{case}
\item[Construction] Delete $v, p$. Delete $r, s$. Fuse a small LR to every remaining neighbor of $p$.
    Suppose $k$ small LRs are fused this way.
\item[Domination] If $S \subseteq H$ is a minimum dominating set and $L \subseteq S$ are the $k$ vertices in the small LRs,
    then $S \cup \{p\} \setminus L$ is a dominating set in $G$. $\Delta s \le 1-k$.
\item[Penalty] Deleting four vertices decreases $\Phi$ by 4. Fusing the small LRs increases $\Phi$ by $3.5 k$.
    The deletions may create an ear tip at $t$, but no other low degree problem: every other vertex that
    lost some neighbors got fused to a small LR. Overall, $-\Delta \Phi \le -4 + 3.5k + 0.5 = -3.5 (1-k)$.
\end{case}
\begin{figure}
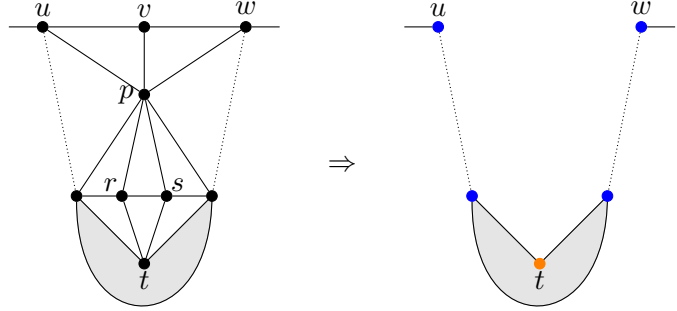

    \beforeafter[0.9]{
        \node (L) at (-1.5, 1) [label=above:$u$] {};
        \node (M) at (0, 1) [label=above:$v$] {};
        \node (R) at (1.5, 1) [label=above:$w$] {};
        \node (A) at (0, 0) [label=left:$p$] {};
        \node (p) at (-1, -1.5) {};
        \node (q) at (-0.33, -1.5) [label=above left:$r$] {};
        \node (r) at (0.33, -1.5) [label=above right:$s$] {};
        \node (s) at (1, -1.5) {};
        \node (I) at (0, -2.5) [label=below:$t$] {};
        \draw (A) -- (p)  (I) -- (q) -- (A) -- (r) -- (I)  (s) -- (A);
        \draw (p) -- (q) -- (r) -- (s);
        \draw (L) -- (M) -- (R);
        \draw (L) -- (A) (M) -- (A) (R) -- (A);
        \draw (L) --+ (-0.5, 0);
        \draw (R) --+ (0.5, 0);
        \draw [densely dotted] (L) -- (p) (R) -- (s);
        \begin{scope}[on background layer]
            \draw [transparent, fill=gray, opacity=0.2] (p) to[in=-90, out=-90, distance=60] (s.center) -- (I.center) -- (p.center);
            \draw (p) to[in=-90, out=-90, distance=60] (s.center) -- (I.center) -- (p.center);
        \end{scope}
    }{
        \node (L) at (-1.5, 1) [label=above:$u$, blue] {};
        \node (R) at (1.5, 1) [label=above:$w$, blue] {};
        \node (p) at (-1, -1.5) [blue] {};
        \node (s) at (1, -1.5) [blue] {};
        \node (I) at (0, -2.5) [label=below:$t$, orange] {};
        \draw (L) --+ (-0.5, 0);
        \draw (R) --+ (0.5, 0);
        \draw [densely dotted] (L) -- (p) (R) -- (s);
        \begin{scope}[on background layer]
            \draw [transparent, fill=gray, opacity=0.2] (p) to[in=-90, out=-90, distance=60] (s.center) -- (I.center) -- (p.center);
            \draw (p) to[in=-90, out=-90, distance=60] (s.center) -- (I.center) -- (p.center);
        \end{scope}
    }
    \caption{Degree $3$ vertex with interior Octahedron.  
    There may or may not be additional vertices in the shaded area.}
    \label{fig:threeinteriorocta}
\end{figure}

\subsubsection{Interior \boldmath$K_4$}
Suppose there is an interior degree-3 vertex $r$ adjacent to $p$.
By Case~\ref{sect:boundarykfour}, WLOG assume that $w$ has no interior degree-3 neighbor.
In particular, $r$ is not adjacent to $w$.
See Figures~\ref{fig:interiorkfourinterior} and~\ref{fig:interiorkfourproblems}.
\begin{case}
\item[Construction] Delete $w, v$. Force $p$ by attaching an A to $u, p$. Delete $r$.
    Fuse a small LR to every former neighbor of $r$ that ends up as a boundary vertex.
    Suppose $k$ LRs are fused this way.
\item[Domination] If $S$ is a minimum neat dominating set in $H$ and $L$ are the $k$ vertices in the LRs,
    then $S \setminus L$ contains $p$, which dominates $r, v, w$ and any former neighbor of $r$. $\Delta s = -k$.
\item[Penalty] Deleting three vertices decreases $\Phi$ by $3$. Attaching an A increases $\Phi$ by $2.5$.
    Fusing the LRs increases $\Phi$ by $3.5 k$.
    Deleting $w, v$ may create a low-degree problem involving $x$, but nowhere else.
    Deleting $r$ does not create low-degree problems due to the fused LRs.
    Overall, $-\Delta \Phi \le -3 + 2.5 + 3.5k + 0.5 = -3.5(-k)$.
\end{case}
\begin{figure}
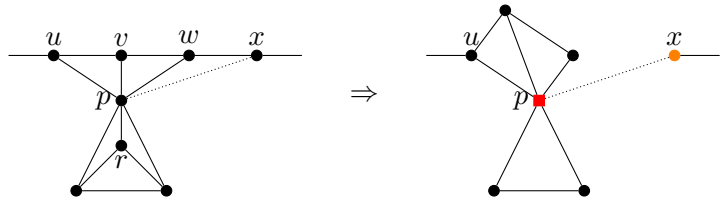

    \beforeafter[0.6]{
        \node (A) at (-1.5, 0) [label=above:$u$] {};
        \node (B) at (0, 0) [label=above:$v$] {};
        \node (C) at (1.5, 0) [label=above:$w$] {};
        \node (D) at (3, 0) [label=above:$x$] {};
        \node (O) at (0, -1) [label=left:$p$] {};
        \node (x) at (-1, -3) {};
        \node (y) at (0, -2) [label=below:$r$] {};
        \node (z) at (1, -3) {};
        \node (a) at (-0.75, 1) [transparent] {};
        \draw (A) -- (B) -- (C) -- (D);
        \draw (A)--+ (-1, 0) (D) --+ (1, 0);
        \draw (A) -- (O) -- (B) (C) -- (O) -- (y);
        \draw (x) -- (y) -- (z) -- (x) -- (O) -- (z);
        \draw[densely dotted] (O) -- (D);
    }{
        \node (A) at (-1.5, 0) [label=above:$u$] {};
        \node (B) at (0, 0) [transparent] {};
        \node (C) at (1.5, 0) [transparent] {};
        \node (D) at (3, 0) [label=above:$x$, orange] {};
        \node (O) at (0, -1) [label=left:$p$, red,rectangle] {};
        \node (a) at (-0.75, 1) {};
        \node (b) at (0.75, 0) {};
        \node (x) at (-1, -3) {};
        \node (y) at (0, -2) [transparent] {};
        \node (z) at (1, -3) {};
        \draw (A)--+ (-1, 0) (D) --+ (1, 0);
        \draw (A) -- (O);
        \draw (z) -- (x) -- (O) -- (z);
        \draw[densely dotted] (O) -- (D);
        \draw (A) -- (a) -- (b) -- (O) -- (a);
    }
    \caption{Interior $K_4$, deleting $r$ does not affect new boundary vertices.}
    \label{fig:interiorkfourinterior}
\end{figure}
\begin{figure}
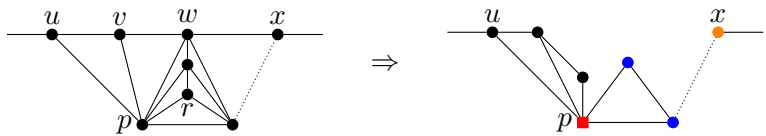

    \beforeafter[0.6]{
        \node (A) at (-2, 1) [label=above:$u$] {};
        \node (B) at (-0.5, 1) [label=above:$v$] {};
        \node (C) at (1, 1) [label=above:$w$] {};
        \node (D) at (3, 1) [label=above:$x$] {};
        \node (O) at (0, -1) [label=left:$p$] {};
        \node (x) at (2, -1) {};
        \node (y) at (1, -0.33) [label=below:$r$] {};
        \node (z) at (1, 0.33) {};
        \node (a) at (-0.75, 1) [transparent] {};
        \draw (A) -- (B) -- (C) -- (D);
        \draw (A)--+ (-1, 0) (D) --+ (1, 0);
        \draw (A) -- (O) -- (B) (C) -- (O) -- (y);
        \draw (x) -- (y) -- (z) -- (x) -- (O) -- (z);
        \draw (C) -- (z) (C) -- (x);
        \draw[densely dotted] (x) -- (D);
    }{
        \node (A) at (-2, 1) [label=above:$u$] {};
        \node (B) at (-0.5, 1) [transparent] {};
        \node (C) at (1, 1) [transparent] {};
        \node (D) at (3, 1) [label=above:$x$, orange] {};
        \node (O) at (0, -1) [label=left:$p$, red,rectangle] {};
        \node (x) at (2, -1) [blue] {};
        \node (y) at (1, -0.33) [transparent] {};
        \node (z) at (1, 0.33) [blue] {};
        \node (a) at (-1, 1)  {};
        \node (b) at (0, 0)  {};
        \draw (A)--+ (-1, 0) (D) --+ (1, 0);
        \draw (A) -- (O) (O) -- (z);
        \draw (x) -- (z) (x) -- (O);
        \draw[densely dotted] (x) -- (D);
        \draw (A) -- (a) -- (b) -- (O) -- (a);
    }
    \caption{Interior $K_4$, deleting $r$ after $v, w$ creates an ear.
    (There is a similar situation with a 3-pair in a bad 5-wheel instead of an ear.)}
    \label{fig:interiorkfourproblems}
\end{figure}

\subsubsection{Conclusion}
For $v, p$ as defined above, deleting $v$ and $p$ creates at most two low-degree problems and these involve $u$ and $w$ respectively.

\subsection{Deleting a degree-\boldmath$3$ boundary vertex and its interior neighbor} \label{sect:threedelete}
\begin{figure}
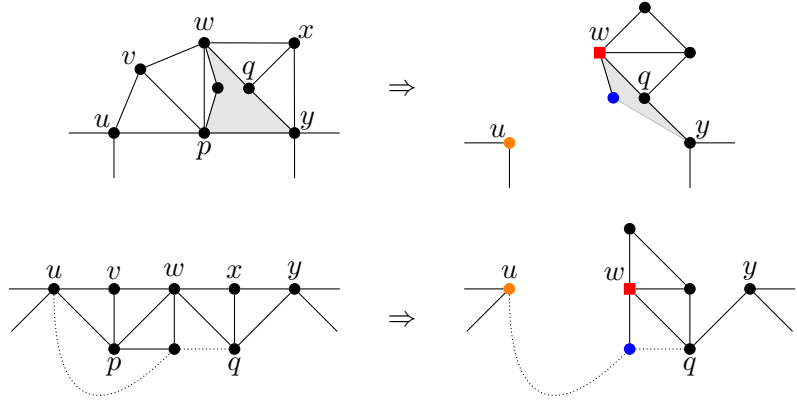

    \beforeafter[1.2]{
        \node (A) at (180:1) [label=above left:$u$] {};
        \node (B) at (135:1) [label=above left:$v$] {};
        \node (C) at (90:1) [label=above:$w$] {};
        \node (D) at (45:0.7) [label=above:$q$] {};
        \node (E) at (0:1) [label=above right:$y$] {};
        \node (X) at (45:1.41) [label=above right:$x$] {};
        \node (O) at (0:0) [label=below:$p$] {};
        \node (I) at (0.15, 0.5) {};
        \draw (A) -- (B) -- (C) -- (D) -- (E);
        \draw (X) -- (C) (X) -- (D) (X) -- (E);
        \draw (A) -- (O) -- (B) (C) -- (O) (E) -- (O);
        \draw (A) --+ (-0.5, 0) (E) --+ (0.5, 0);
        \draw (C) -- (I) -- (O);
        \begin{scope}[on background layer]
            \draw[fill=gray, opacity=0.2] (C.center) to (I.center) to (O.center) to (E.center) to (D) to cycle;
        \end{scope}
        \draw (E) --+ (0, -0.5) (A) --+ (0, -0.5);
    }{
        \node (A) at (180:1) [label=above left:$u$, orange] {};
        \node (B) at (135:1) [transparent] {};
        \node (C) at (90:1) [label=above:$w$, red,rectangle] {};
        \node (D) at (45:0.7) [label=above:$q$] {};
        \node (E) at (0:1) [label=above right:$y$] {};
        \node (X) at (45:1.41) [transparent] {};
        \node (O) at (0:0) [transparent] {};
        \node (S) at (0.5, 1.5) {};
        \node (I) at (0.15, 0.5) [blue] {};
        \node (T) at (1, 1) {};
        \draw (C) -- (D) -- (E);
        \draw (A) --+ (-0.5, 0) (E) --+ (0.5, 0);
        \draw (D) -- (T) -- (S) -- (C) -- (T);
        \draw (C) -- (I);
        \begin{scope}[on background layer]
            \draw[fill=gray, opacity=0.2] (C.center) to (I.center) to (E.center) to (D) to cycle;
        \end{scope}
        \draw (E) --+ (0, -0.5) (A) --+ (0, -0.5);
    }
    \beforeafter[0.8] {
        \node (A) at (-2, 0) [label=above:$u$] {};
        \node (B) at (-1, 0) [label=above:$v$] {};
        \node (C) at (0, 0) [label=above:$w$] {};
        \node (D) at (1, 0) [label=above:$x$] {};
        \node (E) at (2, 0) [label=above:$y$] {};
        \node (X) at (-1, -1) [label=below:$p$] {};
        \node (Y) at (1, -1) [label=below:$q$] {};
        \node (Z) at (0, -1) {};
        \node (S) at (0, 1) [transparent] {};
        \node (T) at (1, 0) [transparent] {};
        \draw (A) -- (B) -- (C) -- (D) -- (E);
        \draw (A) -- (X) -- (C) -- (Y) -- (E);
        \draw (B) -- (X) -- (Z) (D) -- (Y) (C) -- (Z);
        \draw (A) --+ (-0.75, 0);
        \draw (E) --+ (0.75, 0);
        \draw (A) --+ (-135:1) (E) --+ (-45:1); 
        \draw [densely dotted] (A) to[out=-90, in=-135, distance=50] (Z) to[out=0, in=-180, distance=0] (Y);
    }{
        \node (A) at (-2, 0) [label=above:$u$, orange] {};
        \node (C) at (0, 0) [label=above left:$w$, red,rectangle] {};
        \node (E) at (2, 0) [label=above:$y$] {};
        \node (Y) at (1, -1) [label=below:$q$] {};
        \node (Z) at (0, -1) [blue]{};
        \node (S) at (0, 1) {};
        \node (T) at (1, 0) {};
        \draw (Y) -- (T) -- (S) -- (C) -- (T);
        \draw (C) -- (Y) -- (E);
        \draw (C) -- (Z);
        \draw (A) --+ (-0.75, 0);
        \draw (E) --+ (0.75, 0);
        \draw (A) --+ (-135:1) (E) --+ (-45:1); 
        \draw [densely dotted] (A) to[out=-90, in=-135, distance=50] (Z) to[out=0, in=-180, distance=0] (Y);
    }
    \caption{Degree $3$ boundary vertex with deletable interior neighbor. Note: The blue vertex is equal to $q$ if $\deg(w) = 4$.
    Above: $p$ adjacent to $y$, but $\deg(y) \ge 5$.
    Below: $p$ not adjacent to $y$, then we assume that deleting $\{x, y\}$ does not create a 3-pair.}
    \label{fig:threedelete}
\end{figure}
The following technical case turns out to be useful in multiple later cases:
Suppose there are at least five polygon vertices.
Suppose $\deg(u) \ge 4$, $\deg(v) = 3$, $\deg(w) \ge 4$, $\deg(x) = 3$, then $\deg(y) \ge 4$.
Let $p$ be the interior vertex adjacent to $v$ and let $q$ be the interior vertex adjacent to $x$.
Suppose that at least one of the following is true (see Figure~\ref{fig:threedelete}):
\begin{itemize}
    \item $\deg(y) \ge 5$.
    \item $y$ is not adjacent to $p$ and deleting (only) the edge $\{x, y\}$ does not turn $y, z$ into a 3-pair in a bad 5-wheel.
\end{itemize}
\begin{case}
\item[Construction] Delete the edge $\{x, y\}$. Delete $v, p$. Delete $x$. Force $w$ by attaching an A to $w, q$.
    Fuse a small LR to the other $H$-boundary neighbor of $w$, which might be $q$.
\item[Domination] If $S \subseteq H$ is a minimum neat dominating set in $H$ and $L \subseteq S$ is the vertex in the small LR,
    then $S \setminus L$ is a $G$-dominating set, as it contains $w$,
    which dominates $v, p, x$ and the vertex to which the LR got fused. $\Delta s = -1$.
\item[Penalty] Deleting three vertices decreases $\Phi$ by $3$. Attaching the A increases $\Phi$ by $2.5$.
    Fusing the LR increases $\Phi$ by $3.5$. 
    Deleting the edge $\{x, y\}$ decreases the degree of $y$ by one, but does not create a low-degree problem involving $y$.
    By Case~\ref{sect:threeproblems}, deleting $v, p$ may create low degree problem at $u$, but nowhere else, as $w$ gets covered by the $A$.
    In total, $-\Delta \Phi \le -3 + 2.5 + 3.5 + 0.5 = -3.5 \cdot (-1)$.
\end{case}

\subsection{Degree pattern \boldmath$3443$, big polygon}
Suppose there are at least eight boundary vertices.
Suppose $\deg(v) = 3$, $\deg(w) = \deg(x) = 4$ and $\deg(y) = 3$.
Then $\deg(u) \ge 4$ and $\deg(z) \ge 4$.
Let $p$ and $q$ be the interior vertices adjacent to $v$ and $y$ respectively.
Let $r$ be the interior vertex adjacent to $p, w, x, q$.

\subsubsection{No degree-2 cut vertices}
Suppose that deleting $u$ and $z$ does not create any degree-2 cut vertices. See Figure~\ref{fig:bifourbig}.
Note that deleting $v$ (or $y$) does not create a low-degree problem at $w$ (or $x$),
hence $u$ and $z$ do not have interior degree-3 neighbors due to Case~\ref{sect:boundarykfour}.
\begin{case}
\item[Construction] Delete $u$ and $z$. Delete the edge $\{w, x\}$, turning $v, w$ and $x, y$ into an A with red vertex $p$ and a $B$ with red vertex $q$, respectively.
\item[Domination] A minimum neat dominating set in $H$ contains $p$ and $q$, which dominate $u$ and $z$, respectively. $\Delta s \le 0$.
\item[Penalty] Deleting two vertices decreases $\Phi$ by $2$. Deleting $u$ creates
    two low-degree problems, involving $t$ and $v$, due to Cases~\ref{sect:boundarykfour} and~\ref{sect:boundaryocta}. Similar for $z$.
    Deleting both $u$ and $z$ does not create any additional low-degree problems, by assumption.
    Deleting $\{w, x\}$ does not create any low-degree problems.
    Overall, $-\Delta \Phi \le -2 + 4 \cdot 0.5 = 0$.
\item[No Leaves] By assumption, there are at least eight boundary vertices, hence $u$ and $z$ do not share any boundary neighbors,
    that could end up as leaves.
\end{case}

\begin{figure}
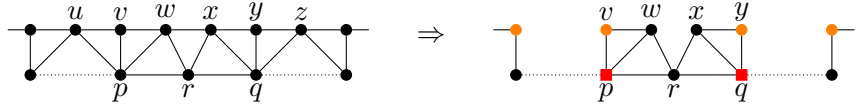

    \beforeafter[0.6]{
        \node (I) at (0, 1) [transparent] {};
        \node (A) at (-2, 0) [label=above:$u$] {};
        \node (B) at (-1, 0) [label=above:$v$] {};
        \node (C) at (0, 0) [label=above:$w$] {};
        \node (D) at (1, 0) [label=above:$x$] {};
        \node (E) at (2, 0) [label=above:$y$] {};
        \node (F) at (3, 0) [label=above:$z$] {};
        \node (L) at (-3, 0) {};
        \node (R) at (4, 0) {};
        \node (U) at (-3, -1) {};
        \node (V) at (-1, -1) [label=below:$p$] {};
        \node (W) at (0.5, -1) [label=below:$r$] {};
        \node (X) at (2, -1) [label=below:$q$] {};
        \node (Y) at (4, -1) {};
        \draw (L) -- (A) -- (B) -- (C) -- (D) -- (E) -- (F) -- (R);
        \draw (L) -- (U)  (V) -- (W) -- (X)  (Y) -- (R);
        \draw (U) -- (A) -- (V) -- (B);
        \draw (V) -- (C) -- (W) -- (D) -- (X) -- (E);
        \draw (X) -- (F) -- (Y);
        \draw[densely dotted] (U) -- (V) (X) -- (Y);
        \draw (L) --+ (-0.5, 0) (R) --+ (0.5, 0);
    }{
        \node (I) at (0, 1) [transparent] {};
        \node (A) at (-2, 0) [transparent] {};
        \node (B) at (-1, 0) [label=above:$v$, orange] {};
        \node (C) at (0, 0) [label=above:$w$] {};
        \node (D) at (1, 0) [label=above:$x$] {};
        \node (E) at (2, 0) [label=above:$y$, orange] {};
        \node (F) at (3, 0) [transparent] {};
        \node (L) at (-3, 0) [orange] {};
        \node (R) at (4, 0) [orange] {};
        \node (U) at (-3, -1) {};
        \node (V) at (-1, -1) [label=below:$p$, red,rectangle] {};
        \node (W) at (0.5, -1) [label=below:$r$] {};
        \node (X) at (2, -1) [label=below:$q$, red,rectangle] {};
        \node (Y) at (4, -1) {};
        \node (M) at (B) [transparent] {};
        \node (N) at (C) [transparent] {};
        \node (O) at (D) [transparent] {};
        \node (P) at (E) [transparent] {};
        \draw (L) -- (U)  (V) -- (W) -- (X)  (Y) -- (R);
        \draw[densely dotted] (U) -- (V) (X) -- (Y);
        \draw (L) --+ (-0.5, 0) (R) --+ (0.5, 0);
        \draw (N) -- (V) -- (M) -- (N) -- (W);
        \draw (O) -- (X) -- (P) -- (O) -- (W);
    }
    \caption{Degree pattern $3443$, big polygon, no degree-2 cut vertices.}
    \label{fig:bifourbig}
\end{figure}

\subsubsection{Interior 5-wheel}
Suppose that deleting $u$ and $z$ creates a degree-2 cut vertex $\ell$.
Then $N[\ell] \subseteq G$ is a 5-wheel with $u, z$ being antipodal vertices in $N(\ell)$.
The other two vertices in $N(\ell)$, say, $j, k$, are both interior vertices,
as there are at least eight boundary vertices. In particular, $\deg(u), \deg(z) \geq 5$
and $\deg(t) = 3$, $\deg(s) \geq 4$. See Figure~\ref{fig:bifourbigwheel}.
\begin{case}
\item[Construction] Delete $\ell$ and add the edge $\{j, k\}$. Delete $v$ and $t$. Force $u$ by attaching a small A to $u, p$.
\item[Domination] A minimum neat dominating set $S \subseteq H$ contains $u$ and hence dominates $G$, as $u$ dominates $v, t, j, k, \ell$. In particular, the added edge $\{j, k\}$ is irrelevant. $\Delta s \le 0$.
\item[Penalty] Deleting three vertices decreases $\Phi$ by $3$. Attaching a small A increases $\Phi$ by 2.5. Replacing $\ell$ by an edge
    only affects $u$ and $z$. Since $\deg_G(u), \deg_G(z) \ge 5$, this does not create low-degree problems. Deleting $v$ and $t$
    may create a low-degree problem involving $s$, but nothing else, increasing $\Phi$ by 0.5. Overall, $-\Delta \Phi \le -3 + 2.5 + 0.5 = 0$.
\item[Smaller] $H$ has fewer interior vertices.
\end{case}
\begin{figure}
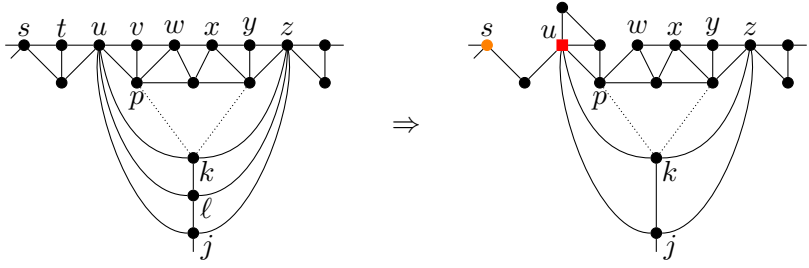

    \beforeafter[0.5]{
        \node (III) at (0, 1) [transparent] {};
        \node (J) at (-3, 0) [label=above:$t$] {};
        \node (A) at (-2, 0) [label=above:$u$] {};
        \node (B) at (-1, 0) [label=above:$v$] {};
        \node (C) at (0, 0) [label=above:$w$] {};
        \node (D) at (1, 0) [label=above:$x$] {};
        \node (E) at (2, 0) [label=above:$y$] {};
        \node (F) at (3, 0) [label=above:$z$] {};
        \node (L) at (-4, 0) [label=above:$s$] {};
        \node (R) at (4, 0) {};
        \node (U) at (-3, -1) {};
        \node (V) at (-1, -1) [label=below:$p$] {};
        \node (W) at (0.5, -1) {};
        \node (X) at (2, -1) {};
        \node (Y) at (4, -1) {};
        \node (a) at (0.5, -3) [label=below right:$k$] {};
        \node (b) at (0.5, -4) [label=below right:$\ell$] {};
        \node (c) at (0.5, -5) [label=below right:$j$] {};
        \draw (L) -- (J) -- (A) -- (B) -- (C) -- (D) -- (E) -- (F) -- (R);
        \draw (L) -- (U)  (V) -- (W) -- (X)  (Y) -- (R);
        \draw (U) -- (A) -- (V) -- (B);
        \draw (V) -- (C) -- (W) -- (D) -- (X) -- (E);
        \draw (X) -- (F) -- (Y);
        \draw (J) -- (U);
        \draw (L) --+ (-0.5, 0) (L) --+(-135:0.5) (R) --+ (0.5, 0);
        \draw (A) to[out=-80, in=-180, distance=40] (a);
        \draw (A) to[out=-90, in=-180, distance=40] (b);
        \draw (A) to[out=-100, in=-180, distance=40] (c);
        \draw (F) to[out=-80, in=0, distance=40] (c);
        \draw (F) to[out=-90, in=0, distance=40] (b);
        \draw (F) to[out=-100, in=0, distance=40] (a);
        \draw (a) -- (b) -- (c);
        \draw[densely dotted] (V) -- (a) -- (X);
        \draw (c) --+ (0, -0.5);
    }{
        \node (III) at (0, 1) [transparent] {};
        \node (J) at (-3, 0) [transparent] {};
        \node (A) at (-2, 0) [label=above left:$u$, red,rectangle] {};
        \node (B) at (-1, 0) [transparent] {};
        \node (C) at (0, 0) [label=above:$w$] {};
        \node (D) at (1, 0) [label=above:$x$] {};
        \node (E) at (2, 0) [label=above:$y$] {};
        \node (F) at (3, 0) [label=above:$z$] {};
        \node (L) at (-4, 0) [label=above:$s$, orange] {};
        \node (R) at (4, 0) {};
        \node (x) at (-2, 1) {};
        \node (y) at (-1, 0) {};
        \node (U) at (-3, -1) {};
        \node (V) at (-1, -1) [label=below:$p$] {};
        \node (W) at (0.5, -1) {};
        \node (X) at (2, -1) {};
        \node (Y) at (4, -1) {};
        \node (a) at (0.5, -3) [label=below right:$k$] {};
        \node (b) at (0.5, -4) [transparent] {};
        \node (c) at (0.5, -5) [label=below right:$j$] {};
        \draw (C) -- (D) -- (E) -- (F) -- (R);
        \draw (L) -- (U)  (V) -- (W) -- (X)  (Y) -- (R);
        \draw (U) -- (A) -- (V);
        \draw (V) -- (C) -- (W) -- (D) -- (X) -- (E);
        \draw (X) -- (F) -- (Y);
        \draw (L) --+ (-0.5, 0) (L) --+(-135:0.5) (R) --+ (0.5, 0);
        \draw (A) to[out=-80, in=-180, distance=40] (a);
        \draw (A) to[out=-100, in=-180, distance=40] (c);
        \draw (F) to[out=-80, in=0, distance=40] (c);
        \draw (F) to[out=-100, in=0, distance=40] (a);
        \draw (a) -- (c);
        \draw[densely dotted] (V) -- (a) -- (X);
        \draw (c) --+ (0, -0.5);
        \draw (y) -- (A) -- (x) -- (y) -- (V);
    }
    \caption{Degree pattern $3443$, big polygon, with interior 5-wheel.
    The red vertex is forced and the orange vertex marks a potential low-degree problem.
    There is a small LR fused to the blue vertex (not drawn).}
    \label{fig:bifourbigwheel}
\end{figure}

\subsubsection{Conclusion}
Now, if there are consecutive boundary vertices with degrees $3,4,4,3$,
then $G$ has at most seven boundary vertices.
Recall that due to previous cases, the polygon has at least four boundary vertices.

\subsection{Degree pattern \boldmath$3443$,  small polygon}
Suppose $\deg(v) = 3$, $\deg(w) = \deg(x) = 4$ and $\deg(y) = 3$.
Let $p$ be the interior vertex adjacent to $v$ and let $q$ be the interior vertex adjacent to $y$,
then $p \ne q$ due to Case~\ref{sect:doublethree}. 
Let $t$ be the shared interior neighbor of $p, w, x, q$.

\subsubsection{Square}
Suppose the polygon is a square. Then $v, y$ are adjacent and both of degree $3$, hence $p = q$,
which is covered by Case~\ref{sect:doublethree}.

\subsubsection{Pentagon}
Suppose the polygon is a pentagon. Then $u=z$ and $\deg(u) \ge 4$.
See Figure~\ref{fig:bifourpenta}. Note that $u$ has no interior degree-3 neighbor due to Case~\ref{sect:boundarykfour}.
\begin{case}
\item[Construction] Delete $u$. Delete the edge $\{w, x\}$, turning $v, w$ into an $A$ with red vertex $p$
    and $x, y$ into a B with red vertex $q$.
\item[Domination] There is a minimum neat dominating set in $H$ containing $p$ and $q$, which dominate $y$.
    This works even if $\deg_H(p) = 3$ or $\deg_H(q) = 3$.
\item[Penalty] Deleting $u$ decreases $\Phi$ by 1. This creates two low-degree problems, at $v$ and $y$,
    increasing $\Phi$ by 1, but nothing else. Deleting the edge $\{w, x\}$ does not create any low-degree problems.
    $-\Delta \Phi \le 0$.
\end{case}
\begin{figure}
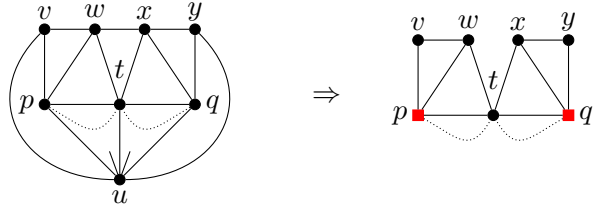

    \beforeafter[0.5]{
        \node (A) at (-2, 0) [label=above:$v$] {};
        \node (B) at (-0.66, 0) [label=above:$w$] {};
        \node (C) at (0.66, 0) [label=above:$x$] {};
        \node (D) at (2, 0) [label=above:$y$] {};
        \node (X) at (-2, -2) [label=left:$p$] {};
        \node (Y) at (0, -2) [label={[label distance=7]above:$t$}] {};
        \node (Z) at (2, -2) [label=right:$q$] {};
        \node (T) at (0, -4) [label=below:$u$] {};
        \draw (A) -- (B) -- (C) -- (D);
        \draw (X) -- (Y) -- (Z);
        \draw (A) -- (X) -- (B) -- (Y) -- (C) -- (Z) -- (D);
        \draw (T) -- (X) (T) -- (Z);
        \draw (A) to[out=-135, in=180, distance=70] (T);
        \draw (D) to[out=-45, in=0, distance=70] (T);
        \draw[densely dotted] (X) to[out=-30, in=-120, distance=30] (Y);
        \draw[densely dotted] (Z) to[out=-150, in=-60, distance=30] (Y);
        \draw (T) --+(110:0.8) (T) --+(70:0.8) (T) -- (Y);
    }{
        \node (A) at (-2, 0) [label=above:$v$] {};
        \node (B) at (-0.66, 0) [label=above:$w$] {};
        \node (C) at (0.66, 0) [label=above:$x$] {};
        \node (D) at (2, 0) [label=above:$y$] {};
        \node (X) at (-2, -2) [label=left:$p$, red,rectangle] {};
        \node (Y) at (0, -2) [label={[label distance=7]above:$t$}] {};
        \node (Z) at (2, -2) [label=right:$q$, red,rectangle] {};
        \node (T) at (0, -4) [transparent] {};
        \draw (A) -- (B)  (C) -- (D);
        \draw (X) -- (Y) -- (Z);
        \draw (A) -- (X) -- (B) -- (Y) -- (C) -- (Z) -- (D);
        \draw[densely dotted] (X) to[out=-30, in=-120, distance=30] (Y);
        \draw[densely dotted] (Z) to[out=-150, in=-60, distance=30] (Y);
    }
    \caption{Degree pattern $3443$, pentagon.}
    \label{fig:bifourpenta}
\end{figure}

\subsubsection{Hexagon}
Suppose the polygon is a hexagon. Then $\deg(u) = 4$, $\deg(z) = 4$ and $u, z$ are adjacent.
Let $s \notin \{p, q\}$ be the interior vertex adjacent to $u$ and $z$.
If $s = r$, then $G$ is the graph depicted in Figure~\ref{fig:bifourhexasmall}, with $\Phi = 9$ and $s = 2$.
Otherwise, $u$ and $z$ are not adjacent to $r$. Then, deleting the edge $\{u, z\}$
does not create any low-degree problems, contradiction. See Figure~\ref{fig:bifourhexabig}.

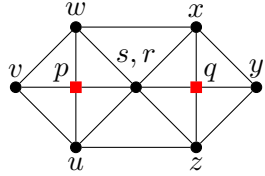
\begin{figure}
    \begin{center}\begin{tikzpicture}[scale=0.8]
        \node (A) at (-2, 0) [label=above:$v$] {};
        \node (B) at (-1, 0) [red,rectangle, label=above left:$p$] {};
        \node (C) at (0, 0) [label=above:{$s,r$}]{};
        \node (D) at (1, 0) [red,rectangle, label=above right:$q$] {};
        \node (E) at (2, 0) [label=above:$y$] {};
        \node (X) at (-1, 1) [label=above:$w$] {};
        \node (Y) at (1, 1) [label=above:$x$] {};
        \node (U) at (-1, -1) [label=below:$u$] {};
        \node (V) at (1, -1) [label=below:$z$] {};
        \draw (A) -- (X) -- (Y) -- (E) -- (V) -- (U) -- (A);
        \draw (A) -- (B) -- (C) -- (D) -- (E);
        \draw (X) -- (B) -- (U) -- (C) -- (Y) -- (D) -- (V) -- (C) -- (X);
    \end{tikzpicture} \end{center}
    \caption{Degree pattern $3443$, hexagon with few interior vertices. The depicted graph has $\Phi = 9$
    and a dominating set of size $2$, drawn in red.}
    \label{fig:bifourhexasmall}
\end{figure}

\begin{figure}
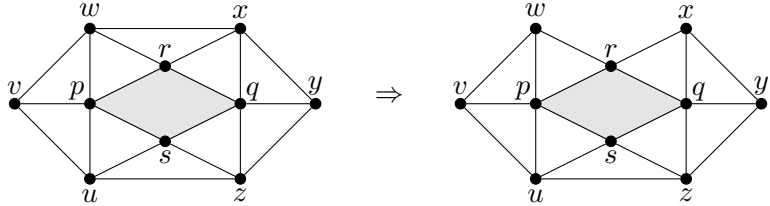

    \beforeafter[1.0]{
        \node (A) at (-2, 0) [label=above:$v$] {};
        \node (B) at (-1, 0) [label=above left:$p$] {};
        \node (C) at (0, 0.5) [label=above:$r$] {};
        \node (CC) at (0, -0.5) [label=below:$s$] {};
        \node (D) at (1, 0) [label=above right:$q$] {};
        \node (E) at (2, 0) [label=above:$y$] {};
        \node (X) at (-1, 1) [label=above:$w$] {};
        \node (Y) at (1, 1) [label=above:$x$] {};
        \node (U) at (-1, -1) [label=below:$u$] {};
        \node (V) at (1, -1) [label=below:$z$] {};
        \draw (A) -- (X) -- (Y) -- (E) -- (V) -- (U) -- (A);
        \draw (A) -- (B) -- (C) -- (D) -- (E);
        \draw (B) -- (CC) -- (D);
        \draw (X) -- (C) -- (Y) -- (D) -- (V) -- (CC) -- (U) -- (B) -- (X);
        \begin{scope}[on background layer]
            \draw[fill=gray, opacity=0.2] (B.center) to (C.center) to (D.center) to (CC.center) to cycle;
        \end{scope}
    }{
        \node (A) at (-2, 0) [label=above:$v$] {};
        \node (B) at (-1, 0) [label=above left:$p$] {};
        \node (C) at (0, 0.5) [label=above:$r$] {};
        \node (CC) at (0, -0.5) [label=below:$s$] {};
        \node (D) at (1, 0) [label=above right:$q$] {};
        \node (E) at (2, 0) [label=above:$y$] {};
        \node (X) at (-1, 1) [label=above:$w$] {};
        \node (Y) at (1, 1) [label=above:$x$] {};
        \node (U) at (-1, -1) [label=below:$u$] {};
        \node (V) at (1, -1) [label=below:$z$] {};
        \draw (A) -- (X)  (Y) -- (E) -- (V) -- (U) -- (A);
        \draw (A) -- (B) -- (C) -- (D) -- (E);
        \draw (B) -- (CC) -- (D);
        \draw (X) -- (C) -- (Y) -- (D) -- (V) -- (CC) -- (U) -- (B) -- (X);
        \begin{scope}[on background layer]
            \draw[fill=gray, opacity=0.2] (B.center) to (C.center) to (D.center) to (CC.center) to cycle;
        \end{scope}
    }
    \caption{Degree pattern $3443$, hexagon with many interior vertices.}
    \label{fig:bifourhexabig}
\end{figure}

\subsubsection{Conclusion}
Now, if there are consecutive boundary vertices with degrees $3,4,4,3$,
then $G$ has at exactly seven boundary vertices.

\subsection{Degree pattern \boldmath$3443$, Heptagon}
Suppose $G$ has exactly seven boundary vertices, namely $t, u, v, w, x, y, z$, with $t$ adjacent to $z$.
Suppose that $\deg(v) = 3, \deg(w) = \deg(x) = 4$ and $\deg(y) = 3$,
then always $\deg(t) = 3$, $\deg(u) \ge 4$, $\deg(z) = 3$ and $\deg(y) \ge 4$.
Let $p, q, s$ be the interior vertices adjacent to $v, y, t$ respectively
and let $r$ be the interior vertex adjacent to $p, w, x, q$. See Figure~\ref{fig:heptanames}.
Then, $p, q, s$ are interior vertices adjacent to a degree-3 boundary vertex and hence distinct,
and $p \ne r \ne q$ as $\deg(u), \deg(z) \ge 4$. To summarize, out of all vertices we defined so far,
the only two that can be equal are $s$ and $r$.

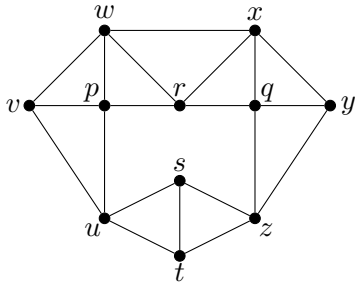
\begin{figure}
    \begin{center} \begin{tikzpicture}
        \node (A) at (-2, 0) [label=left:$v$] {};
        \node (B) at (-1, 0) [label=above left:$p$] {};
        \node (C) at (0, 0) [label=above:$r$] {};
        \node (D) at (1, 0) [label=above right:$q$] {};
        \node (E) at (2, 0) [label=right:$y$] {};
        \node (X) at (-1, 1) [label=above:$w$] {};
        \node (Y) at (1, 1) [label=above:$x$] {};
        \node (P) at (-1, -1.5) [label=below left:$u$] {};
        \node (Q) at (0, -1) [label=above:$s$] {};
        \node (R) at (0, -2) [label=below:$t$] {};
        \node (S) at (1, -1.5) [label=below right:$z$] {};
        \draw (A) -- (B) -- (C) -- (D) -- (E) -- (Y) -- (X) -- (A);
        \draw (B) -- (X) -- (C) -- (Y) -- (D);
        \draw (A) -- (P) -- (R) -- (S) -- (E);
        \draw (B) -- (P) -- (Q) -- (S) -- (D) (Q) -- (R);
    \end{tikzpicture} \end{center}
    \caption{Degree pattern $3443$, Heptagon.}
    \label{fig:heptanames}
\end{figure}

Note that deleting $\{x, y\}$ does not create a 3-pair involving $x$. Similar for $\{v, w\}$ and $w$.
If $s \ne r$, then $x$ is not adjacent to $s$ and Case~\ref{sect:threedelete} applies (which involves deleting $s, t, y$ and forcing $z$).
Therefore, assume $s = r$.
If $\deg(z) \ge 5$, then Case~\ref{sect:threedelete} applies once again.
Similar if $\deg(u) \ge 5$.
The only remaining case is $\deg(u) = \deg(z) = 4$,
then $G$ is the special 4343434-Heptagon.

\subsubsection{Conclusion}
Now $G$ has no consecutive boundary vertices with degrees $3, 4, 4, 3$.
The only remaining degree patters are $3, 5^{+}, 3$ and $3, 4, 3$.
In particular, every other boundary vertex has degree $3$.

\subsection{Degree \boldmath$5^{+}$ boundary vertices}
Suppose there is a boundary vertex of degree $\ge 5$.

\subsubsection{Big polygon}
Suppose there are at least five boundary vertices.
Let $\deg(y) \ge 5$. Then Case~\ref{sect:threedelete} applies.

\begin{figure}
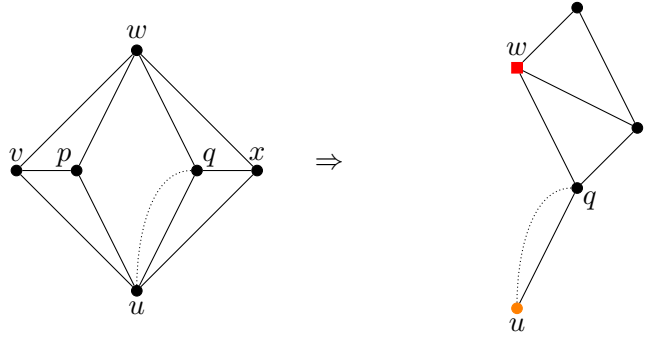

    \beforeafter[0.8]{
        \node (A) at (-2, 0) [label=above:$v$] {};
        \node (B) at (-1, 0) [label=above left:$p$] {};
        \node (C) at (1, 0) [label=above right:$q$] {};
        \node (D) at (2, 0) [label=above:$x$] {};
        \node (X) at (0, 2) [label=above:$w$] {};
        \node (Y) at (0, -2) [label=below:$u$] {};
        \draw (A) -- (X) -- (D) -- (Y) -- (A);
        \draw (A) -- (B) (C) -- (D);
        \draw (B) -- (X) -- (C) -- (Y) -- (B);
        \draw[densely dotted] (Y) to[out=90, in=180, distance=20] (C);
    }{
        \node (A) at (-2, 0) [transparent] {};
        \node (B) at (-1, 0) [transparent] {};
        \node (C) at (1, 0) [label=below right:$q$] {};
        \node (D) at (2, 0) [transparent] {};
        \node (X) at (0, 2) [label=above:$w$, red,rectangle] {};
        \node (Y) at (0, -2) [label=below:$u$, orange] {};
        \node (U) at (1, 3) {};
        \node (V) at (2, 1) {};
        \draw (X) -- (V) (X) -- (U) -- (V) -- (C);
        \draw (X) -- (C) -- (Y);
        \draw[densely dotted] (Y) to[out=90, in=180, distance=20] (C);
    }
    \caption{Degree $5^{+}$ boundary vertices, square polygon.}
    \label{fig:fiveplussquare}
\end{figure}

\subsubsection{Square polygon}
Suppose there are exactly four boundary vertices.
Let $\deg(u) \ge 5$, $\deg(v) = 3$, $\deg(w) \ge 4$ and $\deg(x) = 3$.
Let $p$ and $q$ be the interior vertices adjacent to $v$ and $x$, respectively.
See Figure~\ref{fig:fiveplussquare}.
The following construction closely mimics Case~\ref{sect:threedelete}.

\begin{case}
\item[Construction] Delete $v, p$. Delete $x$. Force $w$ by attaching an A to $w, q$.
\item[Domination] There is a minimum neat dominating set in $H$ that contains $w$, which dominates $v, p, x$. $\Delta s \le 0$.
\item[Penalty] Deleting three vertices decreases $\Phi$ by $3$. Attaching an A increases $\Phi$ by 2.5.
    The deletions may create a low-degree problem at $w$, but nowhere else. $-\Delta \Phi \le -3 + 2.5 + 0.5 = 0$.
\item[No leaves] $u$ does not end up as a leaf, as $\deg_G(u) \ge 5$.
\end{case}

\subsubsection{Conclusion}
Now, all boundary vertices have degrees $3$ and $4$, in alternating fashion.
If $\deg(u) = 3$, $\deg(v) = 4$ and $\deg(w) = 3$, then any interior vertex
adjacent to $v$ is adjacent to $u$ or $w$.
Therefore, $\Interior(G)$ is 2-connected unless it is a path on two vertices: any interior cut vertex would be adjacent to
two distinct degree-3 boundary vertices, see Figure~\ref{fig:interiorcut}.

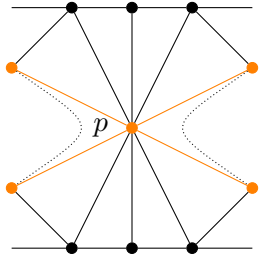
\begin{figure}
    \begin{center} \begin{tikzpicture}[scale=0.8]
        \node (A) at (-2, 1) [orange] {};
        \node (B) at (-1, 2) {};
        \node (C) at (0, 2) {};
        \node (D) at (1, 2) {};
        \node (E) at (2, 1) [orange] {};
        \node (O) at (0, 0) [orange, label={[label distance=5]left:$p$}] {};
        \node (U) at (-2, -1) [orange] {};
        \node (V) at (-1, -2) {};
        \node (W) at (0, -2) {};
        \node (X) at (1, -2) {};
        \node (Y) at (2, -1) [orange]{};
        \draw (A) -- (B) -- (C) -- (D) -- (E);
        \draw (U) -- (V) -- (W) -- (X) -- (Y);
        \draw (B) --+ (-1, 0) (D) --+ (1, 0);
        \draw (V) --+ (-1, 0) (X) --+ (1, 0);
        \draw[densely dotted] (A) to[out=-35, in=35, distance=50] (U); 
        \draw[densely dotted] (E) to[out=-145, in=145, distance=50] (Y); 
        \foreach \u in {B,C,D,V,W,X}{ \draw (O) -- (\u); }
        \foreach \u in {A,E,U,Y}{ \draw[orange] (O) -- (\u); }
    \end{tikzpicture} \end{center}
    \caption{$\Interior(G)$ drawn in orange. The interior cut vertex $p$ is adjacent to multiple
    boundary vertices of degree $3$.}
    \label{fig:interiorcut}
\end{figure}

\subsection{Degree pattern \boldmath$34343$}
Let $\deg(t) = 3$, $\deg(u) = 4$, $\deg(v) = 3$, $\deg(w) = 4$ and $\deg(x) = 3$.
Suppose $G$ has at least five boundary vertices, then $t \ne x$.
Let $p$ be the interior vertex adjacent to $v$ and
let $q$ be the interior vertex adjacent to $x$.

As $G$ has at least five boundary vertices, $\Interior(G)$ is not a path on two vertices.
Suppose we delete $\{w, x\}$. If this results in a 3-pair $u, v$ in some
bad 5-wheel, then $q \in \Interior(G)$ is a cut vertex, contradiction. 
See Figure~\ref{fig:endfivewheel}.

Therefore, deleting any one boundary edge never creates a 3-pair in a bad 5-wheel.
Then, the second case of Case~\ref{sect:threedelete} applies: Deleting $\{x, y\}$
does not create a 3-pair in a bad 5-wheel and $y$ is not adjacent to $p$:
otherwise, deleting $\{v, w\}$ would create the 3-pair $w, x$ in the bad 5-wheel $N[q]$,
see Figure~\ref{fig:endnonadjacent}. 

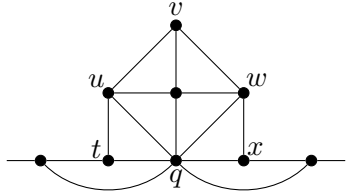
\begin{figure}
    \begin{center} \begin{tikzpicture}[scale=0.9]
        \node (L) at (-2, 0) {};
        \node (A) at (-1, 0) [label=above left:$t$] {};
        \node (B) at (-1, 1) [label=above left:$u$] {};
        \node (C) at (0, 2) [label=above:$v$] {};
        \node (D) at (1, 1) [label=above right:$w$] {};
        \node (E) at (1, 0) [label=above right:$x$] {};
        \node (R) at (2, 0) {};
        \node (I) at (0, 1) {};
        \node (J) at (0, 0) [label=below:$q$] {};
        \draw (L) -- (A) -- (B) -- (C) -- (D) -- (E) -- (R);
        \draw (A) -- (J) -- (B) -- (I) -- (D) -- (J) -- (E);
        \draw (C) -- (I) -- (J);
        \draw (L) to[out=-45, in=-135, distance=20] (J);
        \draw (R) to[out=-135, in=-45, distance=20] (J);
        \draw (L) --+ (-0.5, 0) (R) --+ (0.5, 0);
    \end{tikzpicture} \end{center}
    \caption{Degree pattern $34343$: If deleting $\{w, x\}$ creates a 3-pair $v, w$ in a bad 5-wheel,
    then $u, q$ is the base of the 5-wheel and $q$ is an interior cut vertex.}
    \label{fig:endfivewheel}
\end{figure}
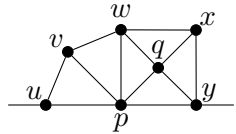
\begin{figure}
    \begin{center} \begin{tikzpicture}
        \node (A) at (180:1) [label=above left:$u$] {};
        \node (B) at (135:1) [label=above left:$v$] {};
        \node (C) at (90:1) [label=above:$w$] {};
        \node (D) at (45:0.7) [label=above:$q$] {};
        \node (E) at (0:1) [label=above right:$y$] {};
        \node (X) at (45:1.41) [label=above right:$x$] {};
        \node (O) at (0:0) [label=below:$p$] {};
        \node (I) at (0.15, 0.5) [transparent] {};
        \draw (A) -- (B) -- (C) -- (D) -- (E);
        \draw (X) -- (C) (X) -- (D) (X) -- (E);
        \draw (A) -- (O) -- (B) (C) -- (O) (E) -- (O);
        \draw (A) --+ (-0.5, 0) (E) --+ (0.5, 0);
        \draw (O) -- (D);
    \end{tikzpicture} \end{center}
    \caption{Degree pattern $34343$: If $y$ is adjacent to $p$, then deleting $\{v, w\}$ creates the 3-pair $w, x$ in a bad 5-wheel $N[q]$.}
    \label{fig:endnonadjacent}
\end{figure}

\subsubsection{Conclusion}
Now, $G$ has exactly four boundary vertices, with degrees $3, 4, 3, 4$.
Therefore, $G$ is the 3-bifan.

\section{A quadratic time algorithm}
Finally, let us discuss how to turn our proof into an algorithm.

\begin{theorem} \label{theo:mainalgorithmic}
Let $G$ be a skeletal triangulation on $n$ vertices that is not a sporadic example.
Then, there is an algorithm that finds a dominating set of size $\lfloor \frac{\Phi(G)}{3.5} \rfloor$
in $\cO(n^2)$ time.
\end{theorem}

Roughly speaking, our proof of Theorem~\ref{theo:mainresult} involves two different types of steps:
\begin{itemize}
\item[(a)] Find a small configuration in $G$, delete some vertices and edges to obtain $H$ and inductively
find a dominating set in $H$. Turn that into a neat dominating set in $H$ and then into a dominating set in $G$.
\item[(b)] Split the graph into two or more parts, determine the act as type of each part and, for each part,
recursively find a rooted dominating set that conforms to this Act as type.
\end{itemize}

Note that sections~\ref{sect:bridge}, \ref{sect:cutvertex} and \ref{sect:nontrivialchord} are the only ones that involve steps of type (b).
Steps of type (a) can easily be implemented with a single recursive call to Theorem~\ref{theo:mainalgorithmic}
and $\cO(n)$ additional time. Steps of type (b) are problematic, as determining the Act as type
is hard. Trying multiple possible acts-as types is also not feasible,
as doing multiple recursive calls into the same part results in exponential running time.

To get around this, for each part, we instead compute $\phi$ and then guess
the acts type to be the
``worst'' possible one according to Theorems~\ref{theo:skeletalact} and~\ref{theo:nearact}.
This guess turns out to have all the properties needed for our proof,
even if it might not match the actual acts as type.
We also show that a ``conforming'' dominating set can be found with a single recursive call to Theorem~\ref{theo:mainalgorithmic} on each part. This ensures an $\cO(n^2)$ running time.

\begin{lemma} \label{lemm:skeletalboundalgo}
Let $(G, u)$ be a rooted skeletal triangulation. Let $s$ be an arbitrary integer and let $\phi = \phi(G, u)$.
\begin{enumerate}
\item If $\phi < 3.5s$, then $G$ has a rooted dominating set of size $s$ that contains $u$.
\item If $\phi < 3.5s + 1.5$, then $G$ has a dominating set of size $s$.
\item If $\phi < 3.5(s+1) - 1$, then $G$ has a rooted dominating set of size $s$.
\end{enumerate}
Moreover, these dominating sets can be found algorithmically via a single call to Theorem~\ref{theo:mainalgorithmic}.
\end{lemma}
\begin{proof}
We mimic the proof of Proposition~\ref{prop:skeletalact}.
\begin{enumerate}
\item Suppose $\phi < 3.5s$. Let $\{u, v\}$ be a boundary edge incident to $u$. Attach a small $A$ with red vertex $u$ to $u, v$.
Then $\phi(H) = \phi + 3.5 < 3.5 (s+1)$. By Theorem~\ref{theo:mainalgorithmic},
$H$ has a dominating set of size $\lfloor \frac{\Phi(H)}{3.5} \rfloor = s$.
In linear time, we turn this into a neat domianting set of size $s$, which contains $u$ due to the attached $A$.
\item Suppose $\phi < 3.5s + 1.5$. If $G$ is a sporadic example, check by hand.
If $\deg(u) = 1$, then apply case (1) to $G - u$. Otherwise,
$G$ is a skeletal triangulation, with $\Phi(G) < 3.5s + 3 + 0.5$.
Then, Theorem~\ref{theo:mainalgorithmic} yields a dominating set of size $\lfloor \frac{\Phi(G)}{3.5} \rfloor = s$.
\item Suppose $\phi < 3.5(s+1) - 1$. Fuse a small LR to $u$, then $\Phi(H) = \phi + 4.5 < 3.5 (s+2)$.
By Theorem~\ref{theo:mainalgorithmic}, $H$ has a dominating set $S$ of size $\lfloor \frac{\Phi(H)}{3.5} \rfloor = s+1$, then $S \cap H$ is a rooted $G$-dominating set of size $s$.
\end{enumerate}
\end{proof}

With this lemma, we guess as follows:
\begin{enumerate}
\item If $3.5s - 1 \le \phi < 3.5s$, guess that $G$ acts as AB.
\item If $3.5s \le \phi < 3.5s + 1.5$, guess that $G$ acts as LR.
\item If $3.5s + 1.5 \le \phi < 3.5(s+1) - 1$, guess that $G$ acts as Nope.
\end{enumerate}
This guessing strategy ensures that $G$ both has the requisite dominating sets (Lemma~\ref{lemm:skeletalboundalgo}) and that $\phi$ satisfies the bounds in Theorem~\ref{theo:skeletalact}. This ensures that the steps in Cases~\ref{sect:bridge} and~\ref{sect:cutvertex} work.

\paragraph{Chords} For Step~\ref{sect:nontrivialchord}, we use similar ideas, but some care has to be taken to distinguish A from B and L from R.

\begin{lemma} \label{lemm:nearboundalgo}
Let $(G, u, v)$ be a rooted near-triangulation. Let $s$ be an arbitrary integer and let $\phi = \phi(G, u, v)$.
\begin{enumerate}
\item If $\phi < 3.5s - 1$, then $G$ has two dominating set of size $s$, one containing $u$ and one containing $v$.
\item If $\phi < 3.5s$, then $G$ has a dominating set of size $s$ that contains $u$ \emph{or} $v$.
\item If $\phi < 3.5s + 0.5$, then $G$ has two rooted dominating set of size $s$, one dominating $u$ and one dominating $v$.
\item If $\phi < 3.5s + 1.5$, then $G$ has a rooted dominating set of size $s$ that dominates $u$ \emph{or} $v$.
\item If $\phi < 3.5(s+1) - 2$, then $G$ has a rooted dominating set of size $s$.
\end{enumerate}
Moreover, these dominating sets can be found algorithmically via Theorem~\ref{theo:mainalgorithmic}.
\end{lemma}
\begin{proof}
The proof follows along the same lines as the proof of Proposition~\ref{prop:nearact}. We only prove (2) and (4), which are the most interesting parts.
\begin{itemize}
\item[2.] Suppose $\phi < 3.5s$. Attach a small OR to $u, v$. Then, $\Phi(H) = \phi(G) + 3.5 < 3.5(s+1)$. By Theorem~\ref{theo:mainalgorithmic},
$H$ has a dominating set of size $\lfloor \frac{\Phi(H)}{3.5} \rfloor = s$, which, due to the OR,
contains $u$ or $v$.
\item[4.] Suppose $\phi < 3.5s + 1.5$. If $\phi(G, u) = \phi(G, u, v) + 1$,
then $\phi(G, u) < 3.5(s+1) - 1$ and Lemma~\ref{lemm:skeletalboundalgo} yields a rooted dominating set of size $s$ that dominates $v$. Similarly, if $\phi(G, v) = \phi(G, u, v) + 1$, then there is a rooted dominating set of size $s$ that dominates $u$. In all other cases, $u, v$ is a 3-pair in a bad 5-wheel. Let $H = G / \{u, v\}$,
then $\phi(H, uv) = \phi(G, u, v) < 3.5s + 1.5$. By Lemma~\ref{lemm:skeletalboundalgo},
there is a $H$-dominating set of size $s$. A neat such set does not contain $uv$, and hence
is a rooted $G$-dominating set that dominates at least one of $u$ and $v$.
\end{itemize}
\end{proof}

With this lemma, we guess as follows:
\begin{enumerate}
\item If $3.5s -2 \le \phi < 3.5s - 1$, guess that $G$ acts as OR.
\item If $3.5s - 1 \le \phi < 3.5s$, guess that $G$ acts as one of A, B.
\item If $3.5s \le \phi < 3.5s + 0.5$, guess that $G$ acts as L OR R.
\item If $3.5s+0.5 \le \phi < 3.5s + 1.5$, guess that $G$ acts as one of L, R.
\item If $3.5s + 1.5 \le \phi < 3.5(s+1) - 2$, guess that $G$ acts as None. (This never happens.)
\end{enumerate}
Proceed as in Step~\ref{sect:nontrivialchord} of the proof. Note that in cases (2) and (4),
the algorithmic nature of Lemma~\ref{lemm:nearboundalgo} yields the actual (rooted) dominating sets,
which then allows us to distinguish A from B and L from R.
Also note that, instead of guessing that $G$ acts as None, we guess that $G$ acts as OR.
Intuitively, this makes sense, as we can just add one of $u, v$ to a rooted dominating set.

\bibliography{references,zotero}

\begin{thebibliography}{10}

\bibitem{brinkmann2007fast}
Gunnar Brinkmann, Brendan~D McKay, et~al.
\newblock Fast generation of planar graphs.
\newblock {\em MATCH Commun. Math. Comput. Chem}, 58(2):323--357, 2007.

\bibitem{campos_dominating_2013}
C.~N. Campos and Y.~Wakabayashi.
\newblock On dominating sets of maximal outerplanar graphs.
\newblock {\em Discrete Applied Mathematics}, 161(3):330--335, February 2013.

\bibitem{ClaverolGHHMMT21}
Merc{\`{e}} Claverol, Alfredo Garc{\'{\i}}a, Carlos~G. Hern{\'{a}}ndez, Carmen
  Hernando, Montserrat Maureso, Merc{\`{e}} Mora, and Javier Tejel.
\newblock Total domination in plane triangulations.
\newblock {\em Discret. Math.}, 344(1):112179, 2021.

\bibitem{DorflingHJ16}
Michael Dorfling, Johannes~H. Hattingh, and Elizabeth Jonck.
\newblock Total domination in maximal outerplanar graphs {II}.
\newblock {\em Discret. Math.}, 339(3):1180--1188, 2016.

\bibitem{FuruyaMatsumoto}
Michitaka Furuya and Naoki Matsumoto.
\newblock A note on the domination number of triangulations.
\newblock {\em J. Graph Theory}, 79(2):83--85, 2015.

\bibitem{GoddardHenning}
Wayne Goddard and Michael~A. Henning.
\newblock Domination in planar graphs with small diameter.
\newblock {\em J. Graph Theory}, 40(1):1--25, 2002.

\bibitem{HonjoKawarabayashi}
Tatsuya Honjo, Ken{-}ichi Kawarabayashi, and Atsuhiro Nakamoto.
\newblock Dominating sets in triangulations on surfaces.
\newblock {\em J. Graph Theory}, 63(1):17--30, 2010.

\bibitem{KingPelsmajer10}
Erika L.~C. King and Michael~J. Pelsmajer.
\newblock Dominating sets in plane triangulations.
\newblock {\em Discret. Math.}, 310(17-18):2221--2230, 2010.

\bibitem{LemanskaZZ17}
Magdalena Lemanska, Rita Zuazua, and Pawel Zylinski.
\newblock Total dominating sets in maximal outerplanar graphs.
\newblock {\em Graphs Comb.}, 33(4):991--998, 2017.

\bibitem{LiuPelsmajer11}
Hong Liu and Michael~J. Pelsmajer.
\newblock Dominating sets in triangulations on surfaces.
\newblock {\em Ars Math. Contemp.}, 4(1):177--204, 2011.

\bibitem{MacGillivray}
Gary MacGillivray and Karen Seyffarth.
\newblock Domination numbers of planar graphs.
\newblock {\em J. Graph Theory}, 22(3):213--229, 1996.

\bibitem{matheson_dominating_1996}
Lesley~R. Matheson and Robert~E. Tarjan.
\newblock Dominating {Sets} in {Planar} {Graphs}.
\newblock {\em European Journal of Combinatorics}, 17(6):565--568, August 1996.

\bibitem{plummer_dominating_2016}
Michael~D. Plummer, Dong Ye, and Xiaoya Zha.
\newblock Dominating plane triangulations.
\newblock {\em Discrete Applied Mathematics}, 211:175--182, October 2016.

\bibitem{plummer_dominating_2020}
Michael~D. Plummer, Dong Ye, and Xiaoya Zha.
\newblock Dominating maximal outerplane graphs and {Hamiltonian} plane
  triangulations.
\newblock {\em Discrete Applied Mathematics}, 282:162--167, August 2020.

\bibitem{Plummer}
Michael~D. Plummer and Xiaoya Zha.
\newblock On certain spanning subgraphs of embeddings with applications to
  domination.
\newblock {\em Discret. Math.}, 309(14):4784--4792, 2009.

\bibitem{thomassen1983theorem}
Carsten Thomassen.
\newblock A theorem on paths in planar graphs.
\newblock {\em Journal of Graph Theory}, 7(2):169--176, 1983.

\bibitem{tokunaga_dominating_2013}
Shin-ichi Tokunaga.
\newblock Dominating sets of maximal outerplanar graphs.
\newblock {\em Discrete Applied Mathematics}, 161(18):3097--3099, December
  2013.

\bibitem{spacapan_domination_2020}
Simon Špacapan.
\newblock The domination number of plane triangulations.
\newblock {\em Journal of Combinatorial Theory, Series B}, 143:42--64, July
  2020.

\end{thebibliography}

\end{document}